\documentclass{amsart}
\usepackage{verbatim}
\usepackage{amssymb,amsmath}
\usepackage{amsthm}
\usepackage{graphicx}
\usepackage{epsfig}
\newtheorem{corollary}{Corollary}[section]
\newtheorem{theorem}[corollary]{Theorem}
\newtheorem{lemma}[corollary]{Lemma}
\newtheorem{proposition}[corollary]{Proposition}
\newtheorem{conjecture}[corollary]{Conjecture}
\newcommand{\Prob} {{\mathbb P}}
\newcommand{\Z}{{\mathbb Z}}
\newcommand{\E}{{\mathbb E}}

\newcommand{\R}{{\mathbb{R}}}
\newcommand{\C}{{\mathbb C}}

\newcommand{\dist}{{\rm dist}}

\def \Im {{\rm Im}}
\def \Re {{\rm Re}}
\def \p {\partial}

\def \Half {{\mathbb H}}

\def \Disk {{\mathbb D}}

\def \diam {{\rm diam}}

\def \hcap {{\rm hcap}}
\def \ppsum {{\sum^{PP}}}

\def \F  {{\mathcal F}}

\def \exc {{\mathcal E}}
\def \z {{\bf z}}
\def \w {{\bf w}}

\def \n{{\bf n}}
\def \cottwo {{\cot_2}}
\def \annmass {{F}}

\newenvironment{definition}[1][Definition]{\begin{trivlist}
\item[\hskip \labelsep {\bfseries #1}]}{\end{trivlist}}
\newenvironment{example}[1][Example]{\begin{trivlist}
\item[\hskip \labelsep {\bfseries #1}]}{\end{trivlist}}
\def \cent{{\bf c}}

\def \cpois {\mathcal H}
\def \tmass {\Psi}

\def \funnerone {\tilde \funone}
\def\funnerthree{\tilde \funthree}
\def \funnertwo{ \biggy}
\def \funone{{\bf A}}
\def  \funthree{{\bf L}}

\def \biggy{\tilde {\bf H}_I}

\def \zhanh {{\mathbf H}}
\def \zhan {\zhanh}
\def \nzhan {{\mathbf J}}
\def \newzhan{{\Theta}}

\def \Heuristic {\noindent {$\clubsuit$}}

\def \linehere { {\hrule}}
\def \labove { \mtwo \linehere \linehere \linehere \ms   }

\def \lbelow {{\ms \linehere \linehere \linehere \mtwo}}

\def \mtwo {{\medskip \medskip}}
\def \ms {{\medskip}}

\title[SLE in multiply connected domains]
{Defining SLE in multiply connected domains with the Brownian loop measure}
\author{Gregory F. Lawler}
\thanks{Research supported by National
 Science Foundation grant DMS-0907143.}
 
\begin{document}
\begin{abstract}
We  define  the Schramm-Loewner evolution ($SLE_\kappa$)
in multiply connected domains for $\kappa \leq 4$ using the Brownian loop measure.
We show that in the case of the annulus, this is the same as the measure
obtained recently by Dapeng Zhan.  We use the loop formulation to give a
different derivation of the partial differential equation for the partition
function in the annulus.
\end{abstract}
\maketitle

\section{Introduction}

The Schramm-Loewner evolution ($SLE$) is a conformally
invariant or conformally covariant family of measures on curves
in the plane.  It was proposed by Schramm \cite{Schramm}
as a candidate for the scaling limit of loop-erased
walk and percolation interfaces, and it has turned out
to be the crucial tool in the rigorous development of
two-dimensional critical phenomenon.  Before $SLE$, there
had been much theoretical, but mathematically nonrigorous,
development using conformal field theory.

In conformal field theory, the standard parameter to
characterize a field is the central charge $\cent$.  There
is a major difference between $\cent \leq 1$ and
$\cent > 1$, and $SLE$ appears in the former case
which is all we consider
in this paper.   The parameter for $SLE$ is denoted $\kappa > 0$.
For each $\cent < 1$, there are two values of $\kappa$, one
less than four and one greater than four, given by
\[     \cent = \frac{(6 - \kappa)(3\kappa-8)}
  {2 \kappa}. \] The smaller value
corresponds to the simple curve case, and we concentrate
on this in this paper.  For $\cent = 1$, $\kappa = 4$ is
a double root which also corresponds to simple curves.
Important examples are $\kappa = 2, \cent = -2$ (loop-erased
walks), $\kappa = 8/3, \cent = 0$ (self-avoiding walks),
$\kappa = 3, \cent = 1/2$ (interfaces of Ising clusters),
$\kappa = 4, \cent = 1$ (interfaces of free fields).  
In all cases, but for self-avoiding walk, $SLE$ has been
proved to be the scaling limits of the
 models \cite{LSWlerw,StasIsing,SSgaussian}
 
 \labove \textsf
{\begin{small}  \Heuristic  
 The letter $c$ is standard in the physics literature for central
 charge. Since we use $c$ for arbitrary constants, it is not
 a good choice for a parameter.  Our compromise is to use
 a bold-face $\cent$.
 \end{small}}
 \lbelow

Another conformally invariant measure on (in this case,
nonsimple) curves in the plane
is given by Brownian motion.   A variant of this measure,
called the Brownian loop measure, arises in the 
study of $SLE$ \cite{LSWrest,LW}.  This is a $\sigma$-finite
measure on nonsimple curves that arises as a scaling limit
of a random walk loop measure, see \cite{LJose} and \cite[Chapter 9]
{LLimic}.  It is closely related to the determinant of the
Laplacian and the Gaussian free field, see, e.g., \cite{DubPart},
but we will only need to view it as a measure on paths.
The key properties of the measure are conformal invariance
and the restriction property. 

In this paper we view $SLE_\kappa$ as a (positive)
measure 
$\mu_D(z,w)$ on curves (modulo increasing reparametrization)
in a domain $D$ of total mass $\tmass_D(z,w)$
connecting two distinct points $z,w$.
Here $z,w$ can be interior points or boundary points
but in the latter case we make some smoothness assumptions
on the boundary.  We expect these curves to 
arise as normalized
limits of measures on lattice curves. If $0 < \tmass_D(z,w)
< \infty$, we can normalize the measure to produce a
probability measure that we denote by $\mu_D^\#(z,w)$.
There are various assumptions we can make on the measures.
We will be more precise later, but let us discuss them now.
The first is conformal covariance:
\begin{itemize}
\item {\bf Conformal covariance.}
 There exist boundary and interior scaling exponents
$b,\tilde b$ such that if $f:D \rightarrow f(D)$ is
a conformal transformation,
\[       f \circ \mu_D(z,w) = |f'(z)|^{b_z}
\, |f'(w)|^{b_w} \, \mu_{f(D)}(f(z),f(w)),\]
where $b_\zeta = b$ if $\zeta \in \p D$ and
$b_\zeta = \tilde b$ if $z \in D$.
\end{itemize}
This implies  conformal \textit{invariance}  of
the probability measures,
\[  f \circ \mu_D^\#(z,w) =  
  \mu_{f(D)}^\#(f(z),f(w)).\]
If one is only considering the probability measures, then
one does not need to make smoothness assumptions at
the boundary. The domain Markov property below uses
the probability measures for nonsmooth boundary points.

There are three other assumptions we will discuss.  It
turns out that they are redundant, so we do not need
to make all of them assumptions, but this is not
obvious.
\begin{itemize}

\item \textbf{Reversibility.}  The measure $\mu_D(w,z)$
can be obtained from $\mu_D(z,w)$ by reversing the paths.

\item \textbf{Domain Markov property}.  In the probability
measure $\mu_D^\#(z,w),$ given an initial segment of
the curve $\gamma_t = \gamma(0,t]$, 
the conditional distribution of
the remainder of the curve is $\mu_{D \setminus \gamma_t}
^\#(\gamma(t),w).$

\item  \textbf{Boundary perturbation.}  Suppose $D_1 \subset 
D$ and the domains agree in neighborhods of $z,w$.  Then
$\mu_{D_1}(z,w)$ is absolutely continuous with respect
to $ \mu_{D}(z,w)$.  In fact, if $\gamma$ is a curve connecting $z$
and $w$ in $D_1$, then the Radon-Nikodym derivative
is given by
\[        \exp \left\{\frac \cent 2 \, m_D(
   \gamma, D \setminus D_1)\right\} , \]
where $m_D(\gamma,D \setminus D_1)$ denotes the
   (Brownian) loop
measure of loops in $D$ that intersect both $\gamma$
and $D_1$. 

\end{itemize}

Schramm \cite{Schramm}  studied the probability
measures $\mu_D^\#(z,w)$ where $z \in \p D$ and $w \in 
D$ or $w \in \p D$.  He showed that for simply connected
$D$, there is only a one-parameter family of measures
satisfying conformal invariance and the domain
Markov property.  He used $\kappa$ as the parameter
and these are now
called   radial and chordal $SLE_\kappa$ (in $D$
from $z$ to $w$), respectively.  It is known \cite{RS,Beffara}
that for $\kappa \leq 4$, the measure is supported
on simple curves of Hausdorff dimension $d = 1
+ \frac \kappa 8 \in (1,\frac 32]$. 
The following has been proved for $SLE_\kappa,
0 < \kappa \leq 4$ in simply connected domains.
\begin{itemize}

\item   Let
\[ b = \frac{6-\kappa}{2\kappa} 
, \;\;\;\;
 \tilde b =  \frac{b(\kappa-2)}{4
 } = \frac{2b + \cent}{12}.\]
Let $\tmass_\Half(0,1) = 1, \tmass_\Disk(1,0) = $
and define $\tmass_D(z,w)$ for other simply
connected domains by
the  scaling rule
\[     \tmass_D(z,w) = |f'(z)|^b \, |f'(w)|^{b_w} \,
 \tmass_{f(D)}(f(w),f(w)), \]
 where $b_w = b$ if $w \in \p D$ and $b_w = \tilde b$
 if $w \in D$. 
Then \cite{LSWrest,LW,LPark} 
if $\mu_D(z,w)
 = \tmass_D(z,w) \, \mu_D^\#(z,w)$, 
 the family $\{\mu_D(z,w)\}$ restricted to
 simply connected domains satisfies conformal
 covariance, domain Markov property, and the
 boundary perturbation rule.
 
 \item If $w \in \p D$, then \cite{Zhanrev}
 $\mu_D^\#(w,z)$ is the same as the reversal
 of $\mu_D^\#(z,w)$.

\end{itemize}

In the chordal case, $\tmass_D(z,w) = H_{\p D}(z,w)^b$
where $H_{\p D}(z,w)$ denotes a multiple of the boundary
Poisson kernel.  This follows from the scaling rule for
the kernel,
\[    H_{\p D}(z,w) = |f'(z)| \, |f'(w)| \, H_{\p f(D)}
  (f(z), f(w)). \] 
If $w \in D$, the Poisson kernel satisfies
\[        H_D(w,z) = |f'(z)| \, H_{f(D)}(f(w), f(z)) , \]
and hence $\tmass_D(w,z)$ is not given by a power of the
Poisson kernel.  If $\kappa = 2$, for which $b=1,\tilde b=0$,
the partition function is given by the boundary Poisson
kernel (chordal case) or Poisson kernel (radial case).  One
can also see this from the relationship with loop-erased
random walk.

In his argument, Schramm uses the fact that if
one slits a simply connected domain $D$ at its boundary
then the resulting domain $D \setminus \gamma_t$
is also simply connected and hence by the Riemann
mapping theorem is conformally equivalent
to the original domain.  If $D$ is not simply
connected, or $D$ is ``slit on the inside'', this
is no longer true.  For this reason, conformal
invariance of the probability measures and the
domain Markov property are not sufficient to determine
the measures $\mu_D^\#(z,w)$ for nonsimply connected
domains. In \cite{LJSP} it was suggested to use
the boundary perturbation rule to extend the definition.
We continue this approach in this paper.  There have been
other approaches, see, e.g., \cite{BF2,BF1,HLeD,HBB}, but none
have directly used the boundary perturbation rule.   

We
will show the following.  (If $z$ or $w$
are boundary points, we implicitly
assume sufficient smoothness
at the boundary.)
\begin{itemize}
\item  There is a unique way (up to some arbitrary
multiplicative constants)  to extend the measures
$\mu_D(z,w)$ so that it satisfies conformal
covariance and the boundary perturbation rule.
\item  If $\kappa \leq 8/3$ ($\cent \leq 0)$, then
$\tmass_D(z,w) < \infty$, and the probability measures
satisfy the domain Markov property.
\item  If $8/3 < \kappa \leq 4$, and $D$ is $1$-connected,
$\tmass_D(z,w) < \infty$.
\end{itemize}
The key observation is that the restriction property
for the Brownian loop measure holds for multiply
connected domains.
We conjecture that $\tmass_D(z,w) < \infty$ for all
$\kappa \leq 4$, but have not shown this.  However, we prove
a weaker fact.
\begin{itemize}
\item  If $\kappa \leq 4$ and $D_1$ is a simply connected
subdomain and $\mu_D(z,w;D_1)$ denotes the measure $\mu_D(z,w)$ restricted
to curves staying in $D_1$, then
\[             \|\mu_D(z,w;D_1)\| < \infty. \]
\item  The probability measures
$\mu_D^\#(z,w;D_1)$ satisfy the domain Markov property.
\item  If $\tmass_D(z,w)< \infty$ for all $k$-connected domains,
then the measures $\mu_D^\#(z,w)$, restricted to $k$-connected
domains, satisfy the domain Markov property.
\end{itemize}
The next property will
 follow from the definition and Zhan's result for
simply connected domains \cite{Zhanrev}.
\begin{itemize}
\item  The measure $\mu_D(w,z)$ is the reversal
of $\mu_D(z,w)$.  
\end{itemize}

Zhan \cite{Zhanannulus} recently took a different approach
to extending $SLE_\kappa$ in the case of an annulus.
Roughly speaking, he shows that there is a unique way
of defining $\mu_D^\#(z,w)$ for conformal annuli so that
it satisfies the domain Markov property and reversibility.
(Note that the combination of the two properties allows
one to describe conditional distributions given both an
initial segment and a terminal segment of the path.)

In this paper, we consider our process for $1$-connected
domains and show that it is the same as that defined
by Zhan.  In particular,  reversibility of the process
follows.  We use the boundary perturbation rule to 
give an equation for the partition function and give a
somewhat more direct proof of existence of the solution.
Although this paper does not directly
use the results in \cite{Zhanannulus},
it does use an idea from that paper.  In particular, the
annulus Loewner equation is used to find PDEs and
the 
Feynman-Kac formula is used to analyze PDEs that
arise.  

We now summarize the contents of the paper.
We describe in Section \ref{latticesec} a model
introduced in \cite{KL} called the $\lambda$-SAW.
It is a two-parameter family of lattice models
for which it is conjectured that there is a
one-parameter subfamily of critical models.  One of
the parameters in \cite{KL} was denoted $\lambda$ but
we have chosen to set $\lambda = -\cent/2$ here. 
It is a generalization of the loop-erased walk
($\cent = -2$) and self-avoiding walk ($\cent=
0$).  This model was created after studying
$SLE$.  While we cannot prove that this has a limit
at the moment (except for $\cent = -2$ and a somewhat
different version for $\cent = 1$ for which we can
use current results), it is useful 
for heuristic understanding of  our definition of
  $SLE$ in multiply connected domains.

Section \ref{prelsec} contains many results that are needed
in the paper most of which have been proved elsewhere.  
This can be skimmed at first reading and referred back
to as needed.  Section \ref{poissec} reviews facts about
the Poisson kernel and sets some notation; this is followed
by discussion of the annulus version.  The annulus Poisson
kernel is often written in terms of theta functions.  We
choose instead to write the functions in terms of infinite
sums which arise naturally when raising the annulus
to the covering space of an infinite strip.  The next
three subsections review the important tools in this area:
$SLE$ in $\Half$, the Brownian bubble measure, and the
Brownian loop measure.   Section \ref{simplesec} 
reviews the methods to analyze $SLE$ in simply connected
domains in terms of the Brownian loop measure and extends
this idea to shrinking domains.  This will allow us to
view radial $SLE$ or annulus $SLE$ in terms of chordal
$SLE$ in $\Half$ where the domain is shrinking by all the
translates of the path. 
 In the case of annulus $SLE$
we get a process that we call locally chordal $SLE_\kappa$.
We write this using an annulus parametrization and
this leads to  the annulus
Loewner equation which we write as an equation in the covering
infinite strip.   

The definiton of $SLE$ is given in Section \ref{defsec}.  In the
boundary to boundary case, this is essentially the same
definition as in \cite{LJSP}.  We extend this to boundary/bulk
and bulk/bulk cases.  One nice thing about our definition
is that reversibility is immediate, given reversibility for
chordal $SLE$ in simply connected domains.
There are some subtleties in defining the bulk/bulk measure
in subdomains of $\C$ in terms of the measure on $\C$,
see Proposition  \ref{prop.subtle}.  The definitions make use
of facts about annulus $SLE$ that are discussed in the next
section.  The extension of the definition to multiple paths
with disjoint endpoints is immediate as in \cite{KL}.

The next two  sections discuss the results about  annulus
$SLE_\kappa$.  Most of the results in this section were
proved in \cite{Zhanannulus}, but there are some differences
in our approach.  We focus on the ``crossing'' case
although the ``chordal'' case can be done similarly as
we point out.   In Section \ref{crosssec} we study
 annulus $SLE_\kappa$ with a given winding number.
By taking its premage under the logarithm, we can consider
it as a measure on curves connecting points of an infinite
strip, and we in turn can compare this measure to chordal
$SLE_\kappa$ in the strip.  This requires comparing the
loop measures in the strip to the preimage of the loop
measure in the annulus.  (Although the loop measure is
conformally invariant, the logarithm is a multi-valued
function, so some care is needed.)  At an intermediate
step we consider the locally chordal $SLE_\kappa$ discused
in Section 3. 
Although this latter process is not
the same as annulus $SLE_\kappa$, it turns out that the
partition function for annulus $SLE_\kappa$ can be given in
terms of a functional of this process.   As in \cite{Zhanannulus}, 
we can then use the Feynman-Kac theorem to write a PDE for
the partition function and this allows us to show that it
gives the quantity we want.

 Section \ref{partsec} takes a different approach and
  derives
the differential equation for the partition function
in the annulus by comparing annulus $SLE_\kappa$ to radial
$SLE_\kappa$.  Smoothness of the partition function follows
from the work of the previous section, so only the
It\^o formula calculation is needed.  The work here shows
that the process we get is the same as the process in \cite{Zhanannulus}.
Our approach  gives a little more than what
is stated explicitly in \cite{Zhanannulus}.  The annulus partition
function is of the form $\tmass(r,x)$, which denotes the
total mass of $SLE_\kappa$ from $1$ to $e^{-r + ix}$ in 
the annulus $A_r = \{e^{-r} < |z| < 1\}$.  The probability
measure is obtained by normalization.  Multiplying the
partition function by a function of
$r$ does not change the probability
measure.  Here we get not only the probability measure but
the correct $r$  dependence.

I would like to thank Dapeng Zhan for useful conversations.

\section{The lattice model}  \label{latticesec}

Here we describe a lattice model for random
walks called the $\lambda$-SAW \cite{KL}.  For
simplicity, we will start with the bulk/bulk version
in a bounded domain $D$.  For convenience, we will
use the integer lattice $\Z^2 = \Z + i \Z$, but the scaling
limit should be independent of the lattice.

A self-avoiding walk (SAW) $\omega = [\omega_0,\ldots,
\omega_n]$ of length $n$
is a finite nearest neighbor path in $\Z^2$
such that $\omega_j \neq \omega_k$ for $j < k$.
Let $|\omega| = n$ denote the length.

A rooted (random walk) loop $\eta = [\eta_0,\ldots,\eta_{2n}]$
of length $2n >0$ is a finite nearest neighbor path
(not necessarily self-avoiding) with $\eta_0 = \eta_{2n}$.
Again we write $|\eta| = 2n$ for the length.  An
unrooted loop is an equivalence class of loops under
the equivalence relation
\[ [\eta_0,\ldots,\eta_{2n}]
  \sim [\eta_j,\eta_{j+1},\ldots,\eta_{2n},
  \eta_1,\ldots,\eta_j] \]
  for each $j$.
The rooted random walk loop measure is the measure
on rooted loops, which assigns measure
$4^{-|\eta|}/|\eta|.$ to each loop $\eta$ with $|\eta|>0$.
This induces a measure $m^{RW}$ 
on unrooted loops called
the {\em random walk loop measure} by giving each
unrooted loop the sum of the weights of the different
rooted loops that give the unrooted loop.

\labove \textsf
{\begin{small}  \Heuristic  
One may think of the unrooted loop measure as assigning
measure $4^{-n}$ to each unrooted loop $\eta$ with 
$|\eta| = n$.  However, this is not exactly correct.
For example, if $n=4$ and $\eta = [x,y,x,y,x]$, then
there are only two different rooted loops that generate
the unrooted loop, and hence this unrooted loop has
measure $4^{-n}/2$.
\end{small}}
\lbelow

Suppose $D$ is a bounded domain in $\C$ and
$z,w$ are distinct points in $D$.  Let $\beta,
\lambda$ be fixed constants which are the parameters
of the model.  For each $n$, let $L_n
= n^{-1} \, \Z^2 \cap D$ and let $z_n,w_n$ be
points in $L_n$ closest to $z,w$ (if there is a tie
for ``closest'', we can choose arbitrarily).
Define the measure $\nu_n = \nu_{n,D,z,w}$ on
SAWs $\omega$ in $L_n$ with endpoints $z_n,w_n$
which gives $\omega$ measure
\[      \exp\left\{-\beta |\omega| + \lambda  
   \, m^{RW}(\omega,D,n) \right\},\]
 where  $m^{RW}(\omega,D,n)$ denotes the
 total $m^{RW}$ measure of (unrooted)
 loops $\eta$ in $L_n$
 that intersect $\omega$.
 Let $Z_n(D) = Z_n(D;\beta,\lambda)$ denote the 
 total mass of the measure.  This is also called
 the {\em partition function}.
 
 This model has two parameters but the conjecture
 is that there is a one-parameter family of critical
 models.  Let us write $\lambda = -\cent/2$ and
 write $\beta = \beta_{\cent}$ for the corresponding
 value of $\beta$.

\labove \textsf
{\begin{small}  \Heuristic  
The value of the critical $\beta$ is a lattice
dependent quantity.  The value   $\lambda = -\cent/2$
is not lattice dependent as long as we define the
random walk loop measure correctly.  For a given
lattice, the  rooted loop
measure is defined to give
 measure $p(\eta)/|\eta|$ to
every loop $\eta$ where $p(\eta)$ is the probability
that simple random walk 
 in the lattice starting at $\eta_0$
 produces the
loop $\eta$.  The value $\cent$ is
the ``central charge''   but 
we can think of it as a free parameter.
\end{small}}
\lbelow

\begin{conjecture}  For each $\cent \leq 1$, there corresponds
a (lattice dependent) $\beta$ and a (lattice independent) scaling
exponent $\tilde b$ such the measure $\nu_n$
has the following properties. 
\begin{itemize}
\item  For each bounded $D$ and distinct $z,w$ in $D$
there exists $\tmass_D^*(z,w)\in(0,\infty)$ such that
\[     Z_n \sim  n^{- 2 \tilde b}
  \, \tmass_D^*(z,w) , \;\;\;\;
     n \rightarrow \infty . \]
\item  There exists a limit measure
on simple curves
\[         \nu_D(z,w) = \lim_{n \rightarrow \infty}
     n^{2 \tilde b} \, \nu_n . \]
\item The family of measures $\{\nu_D(z,w)\}$
  satisfies the conformal covariance
relation: if $f: D \rightarrow f(D)$ is a conformal
transformation, 
\[   f\circ   \nu_D(z,w) = |f'(z)|^{\tilde b}
  \, |f'(w)|^{\tilde b} \, \nu_{f(D)}(f(z),f(w)). \]
\end{itemize}
\end{conjecture}

 There is also a boundary version of this conjecture.  
 Suppose $z$ is a boundary point of $D$ and let us assume
 that $\p D$ is analytic near $z$.  One can define
 the measure $\nu_n$ as above, but there are lattice
 issues involved.  We will not deal with them here
 and just state the following rough conjecture; see
 \cite{KLP} for a more precise statement including
 lattice issues.
 We also assume smoothness near the appropriate boundary
 points.
 
 \begin{conjecture}  For each $\cent \leq 1$, there corresponds
a (lattice dependent) $\beta$ and (lattice independent) scaling
exponents $b,\tilde b$ such the measure $\nu_n$
has the following properties. 
\begin{itemize}
\item  For each bounded $D$ and distinct $z,w$ in $\overline D$
there exists $\tmass_D^*(z,w)$ such that
\[     Z_n \sim  n^{-(b_z + b_w)}
  \, \tmass_D^*(z,w) , \;\;\;\;
     n \rightarrow \infty . \]
\item  There exists a limit measure on simple curves
\[         \nu_D(z,w) = \lim_{n \rightarrow \infty}
     n^{b_z+b_w} \, \nu_n . \]
\item The family of measures $\{\nu_D(z,w)\}$
  satisfies the conformal covariance
relation:
\[   f\circ   \nu_D(z,w) = |f'(z)|^{b_z}
  \, |f'(w)|^{b_w} \, \nu_{f(D)}(f(z),f(w)). \]
\end{itemize}
Here $b_\zeta = b \mbox { or } \tilde b$, respectively,
if $\zeta$ is a boundary point or an interior point.
\end{conjecture}

The conjectures are open, but let us assume that the
conjectures do hold.  Let
\[   \nu_D^\#(z,w) = \frac{\nu_D(z,w)}{\tmass^*_D(z,w)}\]
be the corresponding probability measures which are
conformally {\em invariant}:
\[   f\circ   \nu_D^\#(z,w) =  
    \nu_{f(D)}^\#(f(z),f(w)). \]
 Schramm \cite{Schramm} showed that if $D$ is simply
connected and $z \in \p D$, there is only a one-parameter
family of possible limit measures for $   \nu_D^\#(z,w) $
which are now called \textit{chordal} (if $w \in \p D$)
or \textit{radial} (if $w \in D$) \textit{Schramm-Loewner evolution
with parameter $\kappa$ ($SLE_\kappa$).} 
Analysis of $SLE$ \cite{RS,LSWrest}
shows that $0 < \kappa \leq 4$ (if we want a measure on
simple curves) and the other parameters are given by 
   \begin{equation}\label{bandc}  
 b = \frac{6-\kappa}{2\kappa}, \;\;\;\;
 \tilde b =  \frac{b(\kappa-2)}{4
 }, \;\;\;
  \cent=  6 \tilde b-b = b\,(3\kappa -8). 
  \end{equation} 

Suppose $z,w \in D$ and $D_1 \subset D$, and let
$\nu_n, \nu_n^1$ be the corresponding measures
as above and $L_n = D \cap n^{-1} \Z^2,
 L_n^1 = D_1 \cap n^{-1} \Z^2$. Then
 if $\omega$ is a SAW in $L_n$ connecting
 $z_n$ and $w_n$,
 \[  \frac{\nu_n^1(\omega)}
          {\nu_n(\omega)}  =
        \exp \left\{\frac \cent 2 \, [m^{RW}(\omega,D,n)
          - m^{RW}(\omega,D_1,n)]\right\}. \]
 As $n \rightarrow \infty$, the quantity on the right
 has a limit \cite{LJose}
 in terms of the Brownian loop measure
 \[   \lim_{n \rightarrow \infty}
    [m^{RW}(\omega,D,n)
          - m^{RW}(\omega,D_1,n)] = m_{D}(\omega,D\setminus D_1),\]
  where the right-hand side denotes the Brownian loop measure
 \cite{LW} of loops in $D$ that intersect
  both $\omega$ and $D \setminus D_1$.
 Hence the limit measures should satisfy for $\gamma \subset
 D_1$, 
\begin{equation}  \label{nov29.1}
  \frac{d \nu_{D_1}(z,w)}{d\nu_D(z,w)}
   (\gamma) = \exp \left\{ \frac{\cent}{2}
   \, m_{D}(\gamma,D\setminus D_1)\right\}.
   \end{equation}
For simply connected $D,D_1$ with $z \in \p D$, this was
established in \cite{LSWrest,LW}.
   
  Schramm's construction of $SLE$ makes generalizations
  to nonsimply connected domains difficult.  The purpose of 
this paper is to show that one can use the relation \eqref{nov29.1}
to define it.  This requires some work.
While we do not prove the  conjectures stated in this
section, it is helpful to remember that the definitions
we give in this paper are those of the conjectured
scaling limit of the $\lambda$-SAW with $\lambda = - \cent/2$.

\section{Preliminaries}  \label{prelsec}

In this paper, we  assume that $\kappa \in (0,4]$
and $\cent,b,\tilde b$ are as in \eqref{bandc}.
We  also  set
\[   a = \frac 2 \kappa. \]
Constants throughout may depend implicitly on
$\kappa$.


\subsection{Poisson kernel and related}  \label{poissec}

We establish notation and review facts about the
Poisson kernel.

\begin{itemize}

\item $\Half$ denotes the open
upper half plane, $\Disk$ the open unit
disk, and if $r > 0$,
\[    A_r = \{z \in \Disk:  e^{-r} < |z|  \}, \;\;\;\;
  S_r = \{z \in \Half: \Im(z) < r \},
  \] \[ \Disk_r = e^{-r} \Disk, \;\;\;\;
   C_r = \p \Disk_r  . \]
Under this notation $A_r = \Disk \setminus \overline{{\Disk_r}},
\p A_r = C_0 \cup C_r$.
Throughout this paper we fix
\[ \psi(z) = e^{iz} \] and note that $\psi$ maps $S_r$
(many-to-one) onto $A_r$.   We write $+ \infty,-\infty$ for the
two infinite points in $\p S_r$.
 
 \item   If $D$ is a domain,  then
$z$ is {\em $ \p D$-analytic} if $z \in \p D$ and
there is a neighborhood $N$ of $z$
and a conformal transformation
\[             \phi: N \rightarrow \phi(N) \]
with $\phi(z) = 0$ and $\phi(N \cap D) = \phi(N) \cap \Half$.
We say that $z$ is \textit{$D$-analytic} 
 if $z \in D$ or
$z$ is $\p D$-analytic.

\item If $\gamma$ is a curve, we write $\gamma_t$ for
$\gamma[0,t]$.
\item If  $z,w \in \p D$ and $\gamma:[0,t_0] \rightarrow D$
is a curve with $\gamma(0) = z, \gamma(t_0) = w$ we abuse
notation by writing
$\gamma \subset D$ if $\gamma(0,t_0) \subset D$.  If $t < t_0$,
we write $\gamma_t \subset D$ if $\gamma(0,t] \subset D$.

\item  If  $z \in D$ and $w$ is
$\p D$-analytic, let
 $H_D(z,w) $ denote the Poisson
kernel (that is, the inward
normal derivative
of the Green's function at $w$) normalized so that
\[               H_\Half(x+iy,0) = \frac{y}{x^2 + y^2} . \]
It satisfies the scaling rule
\[
     H_{ D}(z,w) =  |f'(w)| \,
      H_{ f(D)}(f(z),f(w)).
\]
(When writing rules like this, it will be implicitly assumed
that the quantities are well defined.  For example, in this
case $z \in D$,  $w$ is  $\p D$-analytic, and $f(w)$ is
$\p f(D)$-analytic.)
Under our normalization, the probability that a complex
Brownian motion starting at $z$ exits $D$ at
$V \subset \p D$ is
\[  \frac 1   \pi \int_V H_D(z,w) \, |dw|.\]

 \item  If $z,w $ are
distinct $\p D$-analytic points, 
we write $H_{\p D}(z,w)$ for the \textit{boundary} or
\textit{excursion Poisson kernel} given by 
\[         H_{\p D}(z,w) = \p_\n H_D(z,w) = H_{\p D}(w,z) , \]
where $\n = \n_z$ denotes the (inward) normal derivative
at $z$.  It satisfies
the scaling rule: 
\begin{equation}  \label{poisscale}
     H_{\p D}(z,w) = |f'(z)| \, |f'(w)| \, 
      H_{\p f(D)}(f(z),f(w)).
\end{equation}

\item If $D$ is simply connected, there is a complex form
of the Poisson kernel $\cpois_D(z,w)$ such that
$H_D(z,w) = \Im \cpois_D(z,w)$.  This is defined up to
a real translation, and we choose the translation so that
\[   \cpois_\Half(z,0) = - \frac 1 z .\]
The function $f(z) = \cpois_D(z,w)$ can be characterized
 as the unique conformal
transformation $f: D \rightarrow \Half$ such that 
\[          f(w + \epsilon \n_w) =  \frac  i
 \epsilon +o(1), \;\;\;\epsilon \downarrow 0+.\]

\item The Poisson and boundary Poisson kernel for
the strip $S_r$ can be 
computed  using conformal invariance,
\begin{equation}  \label{dec1.3}
    H_{\p S_r}(z,0) =  -\frac{\pi}{2r}
 \, \coth \left(\frac{\pi z}{2r} \right), 
 \end{equation}
\begin{equation}  \label{nov20.1}
          H_{\p S_r}(0,x) = 
\frac{ \pi^2}{4 \, r^2} \left[\sinh\left( \frac{\pi x}{2r}\right)
  \right]^{-2} , 
  \end{equation}
  \begin{equation}  \label{jan31.1}    H_{\p S_r}(0,x +ir) = 
\frac{ \pi^2}{4 \, r^2} \left[\cosh\left( \frac{\pi x}{2r}\right)
  \right]^{-2} . \end{equation}
              
\item  If $z,w$ are distinct boundary points of $D$,
$D_1 \subset D$ with
 $\dist(z,D \setminus D_1)>0,$ $\dist(w,D \setminus D_1) > 0$, let
\[    Q_D(z,w;D_1) \]
denote the probability that a Brownian excursion in $D$
from $z$ to $w$ stays in $D_1$.  (A Brownian excursion
in $D$ is a Brownian motion starting at $z$ and conditioned
to go immediately into $D$ and exit at $w$.  It is not
difficult to make this precise.)  
We note
that $Q_D(z,w;D_1)$ is invariant under conformal
transformations of $D$, and if $z,w$ are $\p D$-analytic,
\[   Q_D(z,w;D_1) = \frac{H_{\p D_1}(z,w)}{H_{\p D}
   (z,w)}. \]
If $D \subset \Half$ is simply connected with $\Half
\setminus D$ bounded and $\dist(0,\Half \setminus
D)>0$, then  \cite[Proposition 5.15]{LBook}
\[            Q_\Half(0,\infty;D) = \Phi_D'(0),\]
where $\Phi_D: D \rightarrow \Half$ is a conformal
transformation with $\Phi(z) \sim z $ as $z \rightarrow 
\infty$. 
\end{itemize}

When studying $SLE$ it is useful to consider subdomains
of $\Half$ and the boundary point infinity.  In order to
make a number of formulas work in this case, it is useful
to adapt the following ``abuse of notation'' about derivatives.  
This can
be considered a kind of normalization at infinity.

\begin{itemize}

\item  When we consider
the conformal transformation $g:\Half \rightarrow
\Half$ given by  $g(z) = -1/z$,
then we write 
\begin{equation} \label{fundev}
 g'(0) = g'(\infty) =-1 .
\end{equation}

\item  If $D \subset \Half$ and $\Half \setminus D$
is bounded, then we say that $\infty$ is  
$\p D$-analytic.  If $D_1,D_2$ are two such  domains
and $f:D_1 \rightarrow D_2$ is a  conformal transformation
with $f(\infty) = \infty$, we define $f'(\infty)$
by
\[          f(z)  \sim \frac{z}{f'(\infty)}, \;\;\;
  z \rightarrow \infty.\]
Equivalently, if $F(z) = -1/f(-1/z) = g \circ f \circ g(z)$, then
\[               f'(\infty) = F'(0).\]

\item More generally, if $F:D \rightarrow D'$ is a conformal
transformation with $F(z) = \infty$ or $F(\infty) = z$,
we  compute derivatives  using the chain
rule and  \eqref{fundev}.

\end{itemize}

The boundary Poisson kernel $H_{\p D}(z,w)$
can be defined if $z$ or $w$ equals infinity using
the scaling rule \eqref{poisscale}.  Under our normalization
$H_{\p \Half}(x,\infty) = 1$.

\labove \textsf
{\begin{small}  \Heuristic  
 If $D,D'$ are simply connected domains,
$z,w$ are distinct $\p D$-analytic points, and $z',w'$
are distinct $\p D'$-analytic points,  then there is
a one parameter family of conformal transformations $f:D
\rightarrow D'$ with $f(z) = z', f(w) = w'$.  The
quantity $f'(z) \, f'(w)$ is invariant of the choice of
the transformation.  Our definitions of derivatives at
infinity are made so that this property holds as well when
$w = \infty$ or $w' = \infty$.
\end{small}}
\lbelow

\subsection{The annulus}  \label{annfunction}

The functions that arise from the Poisson kernel of
the annulus will be important. 
By considering different winding numbers,
 using the scaling rule, and applying \eqref{jan31.1},
  we can see that 
\[
H_{\p A_r}(1,e^{-r+ix}) = e^{r}  \sum_{k=-\infty}^\infty
         H_{\p S_r}(0, x + ir )
          = \frac{e^{r}}{2} \, \nzhan(r,x) , \]
 where $\nzhan(r,x)$ is defined by
 \begin{equation}  \label{nzhan}
   \nzhan(r,x) = \frac {\pi^2}{2 r^2} \sum_{k=-\infty}^\infty          
               \left[\cosh\left( \frac{\pi (x+ 2k\pi)}{2r}\right)
  \right]^{-2}. 
  \end{equation}
  We will view $\nzhan(r,x)$ as a function on $(0,\infty)
  \times \R$ satisfying $\nzhan(r,x) = \nzhan(r,x+ 2 \pi)$.
  Under our normalization of the Poisson kernel,
\begin{equation}  \label{dec1.1}
 e^{-r}  \int_0^{2\pi} H_{\p A_r}(1, e^{-r + ix}) \, dx
  = \frac \pi {r} , 
  \end{equation}
  which implies
\[     \int_0^{2\pi} \nzhan(r,x) \, dx =
  2 \int_0^{\pi} \nzhan(r,x) \, dx  = \frac {2\pi}r. \]
Indeed, $(r/2\pi) \nzhan(r,x)$ has the interpretation as the density of
the angle of the hitting point of an $h$-process in $A_r$
started at $1$ conditioned to leave $A_r$ at $C_r$ (in other
words, the $h$-process associated to the harmonic
function $h(z) = -\log |z|$).   Using this interpretation,
we can see that there exists $\rho  >0$ such that for all
$r$ sufficiently small
\begin{equation}  \label{dec10.1}
        \frac {\rho\, \pi}r
\leq      \int_0^r \nzhan(r,x) \, dx \leq  \frac {(1-\rho)\, \pi}r
.
\end{equation}
\labove \textsf
{\begin{small}  \Heuristic  
 To see \eqref{dec1.1}, recall that under our normalization
 of the Poisson kernel
 \[  e^{-r}  \int_0^{2\pi} H_{ A_r}(e^{-\epsilon}
 , e^{-r + ix}) \, dx\]
is $\pi$ times the probability that a Brownian motion
starting at $e^{-\epsilon} = 1 - \epsilon+ O(\epsilon^2)$ 
leaves $A_r$ at $C_r$.  A
standard estimate for Brownian motion  tells us that
this probability equals $\epsilon/r$.
\end{small}}
\lbelow

  \begin{lemma} \label{dec4.1}
   There exist $c < \infty$ such that 
if $r \geq 1, x \in \R$,
\[   \left|\nzhan(r,x) - \frac 1r\right| \leq c \, e^{-r}.\]
\end{lemma}

\begin{proof} We will assume $r \geq 2$ (the case
$1 \leq r \leq 2$ is easy).
 Let $V$ be a subset of $[0,2\pi)$ which
can also be viewed as  
a periodic subset of $\R$. We need to show that
\[     \frac 1{2\pi}
\int_V \nzhan(r,x) \, dx = \frac{l(V)}{r}
 \, [1 + O(re^{-r})], \]
 where $l$ denotes length.  
 By definition,
\begin{eqnarray*}  \frac 1{2\pi}
\int_V \nzhan(r,x) \, dx& = &  
     \frac{e^{-r}}\pi  \, \int_V  H_{\p A_r}(1,e^{-r
      + ix} ) \, dx\\
      &  = & 
      \frac{e^{-r}}\pi  \, \int_V  H_{\p A_r}(e^{-r},e^{ix}
      ) \, dx . 
      \end{eqnarray*}
Let $B_t$ denote a complex Brownian motion and
$T_s = \inf\{t: B_t \in C_s\}$.  Let
\[   p(z;V) = \Prob^z\{B_{T_0} \in V\}, \;\;\;\;
  q(z;V) = \Prob^z\{ B_{T_0} \in V \mid T_0 < T_r\}, \]
and let $q_\pm(r;V)$ be the maximum and minimum of $q(z,V)$
on $C_{r-1}$.   Then,
\[  q_-(r,V) \leq 
  \frac {re^{-r}} \pi \int_V H_{\p A_r} (e^{-r},e^{ix}) \, dx  \leq \
  q_+(r,V).\]
Hence it suffices to show that if $z \in C_{r-1}$,  
\[         q(z,V) = l(V) \,[1 + O(r \, e^{-r})],\]
where the error term is uniform in $z$.
%
If $z \in C_{r-1}$, then $\Prob^z\{T_0 < T_r\} = 1/r$, and
hence
\[
 p(z,V)  =  r^{-1} \, q(z,V) + (1- r^{-1}) \, \Prob^z
 \{B_{T_0} \in V \mid T_r < T_0\}. \]
Using the strong Markov property and the exact form of the
Poisson kernel in the disk, we see that
 \[  p(z,V)= l(V) \, [1+ O(|z|)], \;\;\;\;
 \Prob^z\{B_{T_0} \in V \mid T_r < T_0\} 
 = l(V) \, [1+ O(|z|)] , \]
 and hence if $z \in C_{r-1}$,
 \begin{eqnarray*}
      r^{-1} \, q(r,V) & =  & l(V) \, [1+ O(|z|)]
   - (1- r^{-1})\, l(V) \, [1+ O(|z|)]
     \\& = & l(V) \, [r^{-1} + O(e^{-r})]\\
     & = & r^{-1} \, l(V) \, [1 + O(re^{-r})].
     \end{eqnarray*}

\end{proof}

 Another important function will be
\[        \zhan_I(r,x)  = - \frac x{r}+ \int_0^x \nzhan(r,y) \, dy
  =   \int_0^x \left[\nzhan(r,y) - \frac 1{r}\right] \, dy,\]
which satisfies
$    \zhan_I(r,x) = \zhan_I(r,x + 2 \pi)$ and
\[    \zhan_I'(r,x) = J(r,x) - \frac 1r . \]
Here we are using the notation from \cite{Zhanannulus},
and the prime denotes an  $x$-derivative.

\begin{lemma}  \label{dec2.lemma5}
Let $K(r,x) = r  \zhanh_I(r,x)$.  Then for all $r$, $K$
is an odd function of period $2\pi$ satisfying $K(r,\pi -x)
= K(r,\pi + x)$, $K(0) = K(\pi) = 0$, and
  \[  K(r,x) \leq \pi - x, \;\;\;\; 0 \leq x \leq \pi.\]
   Moreover, there exists $\epsilon > 0$
  such that for all $r$ sufficiently small and all $x$,
  \[      K(r,x) \leq \pi - \epsilon r.\]
  \end{lemma} 

\begin{proof}  This is straightforward.  The last estimate
uses \eqref{dec10.1}.
\end{proof}


\labove \textsf
{\begin{small}  \Heuristic  
  Although we will not need it for our main
   theorem, in a comment in Section
   \ref{diffsec} we will use the
   fact  
   that the function $\Phi(r,x) = r \, \nzhan(r,x)
 = r \, \zhan_I'(r,x)+1 $ satisfies
the differential equation
\begin{equation}  \label{kappa2}
         \dot \Phi =   \Phi''
  + \zhanh_I \, \Phi + \zhanh_I' \, \Phi.
  \end{equation}
 Here, as later in the paper, we use dots for $r$-derivatives
 and primes for $x$-derivatives.
To see this, we will need
   the following fact from 
  \cite{Zhanannulus}:  
 \[   \dot \zhan_I = \zhan_I'' + \zhan_I' \, \zhan_I
.\]
Hence $G = \zhan_I'$ satisfies
\[  \dot G =G''
   + \zhan_I \, G' + \zhan_I' \,
    G,\]
and 
\begin{eqnarray*}
\dot \Phi & = &  G + r \, G''
   + r\, \zhan_I \, G' +r\,  \zhan _I' \,
    G \\
    & = & G + \Phi'' +  \zhan_I\, \Phi' + \zhan_I' \, (\Phi-1)\\
      & = &  \Phi'' +  \zhan_I\, \Phi' + \zhan_I' \, \Phi\\
\end{eqnarray*}
 \end{small}}
 \lbelow

%
%
%
%
%
%

 \begin{lemma}  \label{dec4.lemma2}
  There exists $c >0$ such
 that the following holds.  Suppose $r \geq 1$ and
 $  f: D \rightarrow A_r $
 is a conformal transformation with $f(C_0) = C_0$
  where 
  $D = \Disk \setminus K$ and $K$ is a compact
  subset containing the origin.  Then for
  $|z| = 1$,
  \[  \left|\, |f'(z)| - 1\, \right|
    \leq c \,e^{-r}, \;\;\;\;|f''(z)| \leq 
       c e^{-r} . \]
\end{lemma}

\begin{proof}  Let $\phi_D$ be the harmonic function on
$D$ with boundary values $0$ on $C_0$ and $1$ on
$K$ and let $\phi_r = \phi_{A_r}$.
 By conformal invariance,
\[     \phi_D(z) = \phi_r(f(z))= \frac{-\log |f(z)|}{r}.\]
Since $f$ maps $C_0$ to $C_0$, this implies
\[   r\, \p_\n \phi_D(z) = |f'(z)| , \]
where $\n$ denotes the inward unit normal.
Also, conformal invariance of excursion measure gives
\[  \int_{C_0}  \p_\n \phi_D(z) \, |dz|
  = \int_{C_0} \, \p_\n \phi_r( z ) \, |dz| = \frac{2\pi}
   r . \]
 Hence to prove the first estimate, 
 it suffices to show for $z,w \in C_0$, 
 \begin{equation}  \label{dec4.2}
  \p_\n \phi_D(z) = \p_\n \phi_D(w)
  \, [1 + O(e^{-r})]. 
  \end{equation}
 Using Koebe estimates, we can find a universal $s$ such
 that for $r$ sufficiently large, $K \subset \Disk_{r-s}$.
Suppose we start Brownian motions at $e^{-\epsilon} \, z$
and $e^{-\epsilon} \, w$, respectively.  The probability
that they reach $C_{r-s}$ without hitting $C_0$ is
$\epsilon/(r-s)$.  On $C_{r-s}$, $\phi_D) = 1 - O(r^{-1})$.
Using Lemma \ref{dec4.1}, we can see that the
 conditional distributions on $C_{r-s}$ given that
the Brownian motions reach $C_{r-s}$ are the same for
$z,w$ up to an error of order $O(re^{-r})$.  Hence,
if $q(z) = q(z,r,\epsilon)$ denotes the probability that
the Brownian motion starting at $z$ reaches $C_{r-s}$
before $C_0$ but does not hit $K$ before $C_0$, then
\[    |q(z) - q(w)| \leq c\frac{ \epsilon}{r-s}
 \, r^{-1} \, O(re^{-r})   \leq \frac{\epsilon}{r}
  \, O(e^{-r}), \] from
  which we conclude \eqref{dec4.2}.  Indeed, we conclude
  the stronger fact,
  \[     q(z) = -\frac{\log |z|}{r}
  [1 + O(e^{-r})]
  ,  \,\;\;\;\;e^{-1} < |z| < 1. \]
  This implies
  \[   |f(z)| = |z| \, [ [1 + O(e^{-r})]
  ,  \,\;\;\;\;e^{-1} < |z| < 1. \]

   For the second
  estimate, fix $z$ and assume without loss of generality
  that $z=1$ and $f(1) = 1$.  By Schwarz reflection, we
  can extend $f$ to a neighborhood of radius $1/2$ about
  $1$.  Let $, L(z) = \log z,
  g(z)  = \log f$. where $L(1) = g(1) = 0$.
   We have  $|\Re g(z) - \Re L(z)| 
   \leq e^{-r}$ and $g(1) = L(1)$.  From this 
   we can use standard arguments to conclude that
   $|g(z) - L(z)| = O(e^{-r})$.  Using the
   Cauchy integral formula, we get
   $|g'(z) - L'(z)|, |g''(z) - L''(z)| 
   \leq c \, O(e^{-r})$.

\end{proof}

A computation that we will do a little later will
give us a particular annulus function $\funone(r,x)$
which we now
define. Suppose
that $D = S_r, z=0, w = x + ri$ and let $\gamma_t$ be
a curve starting at the  origin parametrized so that
$\hcap[\gamma_t] = t$.  Let
 $D_t$ be the domain obtained by splitting $\Half$ by
 the nontrivial 
 $2\pi k$ translates of $\gamma_t$,
\[            D_t    =   S_r \setminus \bigcup_{k \in \Z \setminus \{0\}}
          [\gamma_t + 2\pi k] , \]
          and let
          \[   Q_t = Q_{D}(0,w;D_t).\]
Then (see the end of Section \ref{fradsec}),
one can check that as $t \rightarrow 0$,
\begin{equation}  \label{tomato}
         Q_t = 1 - \funone(r,x) \, t + o(t) , 
         \end{equation}
where
\[            \funone(r,x) = \sum_{k \in \Z \setminus \{0\}}
              \frac{H_{\p S_r}(0,2\pi k) \, H_{\p S_r}(2 \pi k,x+ir)}
                  {H_{\p S_r}(0, x+ir)}.\]
 Using
\eqref{nov20.1}  and  \eqref{jan31.1}, we get
\begin{equation}  \label{fun1}
   \funone(r,x) = \frac{\pi^2} {4  r^2}
  \sum_{ k \in \Z \setminus \{0\}}  \frac{
    \cosh^2(\pi x/2r) }{\sinh^2(\pi^2 k/r)
      \, \cosh^2(\pi(x-2 \pi k)/2r)},
      \end{equation}


%
%
%
%

 \begin{proposition}  \label{feb1.prop1}
  For fixed $r$,
$\funone(r,\cdot)$ is a positive, even function, that is increasing
in $|x|$.   There exists $ c> 0$ such that
 If $0  <  r \leq 1$ and $0 \leq x \leq  \pi$,
  \begin{equation}  \label{oct7.2}  \funone(r,x) \leq  \frac{c}{ r^2}
   \,  \exp \left\{- \frac{2 \pi} r ( \pi - x)\right\} 
   .
   \end{equation}
  \end{proposition}

 \begin{proof}  The definition implies $\funone(r,x)
 = \funone(r,-x)$.   For $r \leq 1$, $0 \leq x < \pi$,  
 \[  \cosh^2(\pi x/2r) \asymp e^{\pi x/r}, \]
 \[   
 \sinh^2(\pi^2 k/r)
      \, \cosh^2(\pi(x-2 \pi k)/2r)
       \asymp e^{2|k|\pi^2/r} \, e^{\pi  |2 \pi k-x|/r}
          \geq    e^{2|k|\pi^2/r} \, e^{\pi (2 \pi-x)/r}. \]
 By summing over $k$, we get \eqref{oct7.2}.         
    The monotonicity in $|x|$ will follow if
    we show that   that for each integer $k$, 
 \[   \frac{
    \cosh^2(\pi x/2r) }{ 
     \cosh^2(\pi(x-2 \pi k)/2r)}+
       \frac{
    \cosh^2(\pi x/2r) }{ 
       \cosh^2(\pi(x +2 \pi k)/2r)}   \]
      is an increasing function of $|x|$. 
      
Indeed, we will now show that if $y \in \R$ and
\[         f(x) = \frac{\cosh^2x}{\cosh^2(x-y)}
          + \frac{\cosh^2 x}{\cosh^2(x+y)}, \]
then $f$ is increasing for $x \geq 0$.
Since
\[   f(x) = \frac{\cosh(2x) + 1}{\cosh(2x - 2y) + 1}
              + \frac{\cosh(2x) + 1}{\cosh(2x + 2y) + 1},\] 
               it suffices to show for every $y \in \R$,  that
\[   F(x) = \frac{\cosh x + 1}{\cosh(x-y)+ 1}
              + \frac{\cosh x + 1}{\cosh(x+y)+ 1}, \]
is increasing for $x \geq 0$. 
Using the sum rule, we
get
\[ \cosh(x-y)+ 1 +  \cosh(x +y)+ 1 =
         2 \cosh x \cosh y + 2 , \]
         Letting $r = \cosh y \geq 1$, we get
\begin{eqnarray*}
 [\cosh(x-y) + 1] \, [\cosh(x+y) + 1]
& = &      (\cosh x \cosh y + 1)^2 - \sinh^2 x \sinh^2 y \\
 & = &     (r \cosh x + 1)^2 - (r^2 - 1) (\cosh^2x - 1) \\
 & = & \cosh^2 x + 2r \cosh x + r^2 \\
  & = & (\cosh x + r)^2 .
 \end{eqnarray*}
Therefore,
\[
    F(x)    =  \frac{2r\, ( \cosh x + r^{-1}) \, (\cosh x + 1)}{(\cosh x + r)^2}
      = 2 r \, e^{G(\cosh x)}, \]
  where
  \[   G(t) = \log (t + \frac 1 r) + \log (t+1) -
       2 \log \, (t+r) . \]
 Since $r \geq 1$, $G'(t) > 0$ for $t > 0$ and hence $G$ and
 $F$ are increasing.
      \end{proof}

 \subsection{$SLE_\kappa$ in $\Half$}  \label{fslesec}
 
If $\kappa = 2/a \in (0,4]$, then 
 {\em chordal $SLE_\kappa$ (in $\Half$ from $0$
 to $\infty$)}  is the  solution to
   the {\em chordal Loewner equation}
\begin{equation}  \label{chordal}
    \p_t g_t(z) = \frac{a}{g_t(z) - U_t },\;\;\;\; g_0(z) = z,
 \end{equation}
 where $U_t = - B_t$ is a standard Brownian motion.
 With probability one \cite{RS}, this generates a 
 random path $\gamma:(0,\infty)\rightarrow \Half$
 such that the domain of  $g_t$ is $\Half \setminus
 \gamma_t$.  The curve is parametrized so that 
 $\hcap[\gamma_t] = at$ (see \cite[Chapter 3]{LBook}
 for definitions); in other words,
 \[    g_t(z) = z + \frac{at}{z} + O(|z|^{-2}),
 \;\;\; z \rightarrow \infty . \]

   For every $r > 0$, let
\[            \tau_r = \inf\{t > 0: \gamma(t) \not \in S_r\}
 = \inf\{t>0 : \Im \gamma(t) = r \}. \]

\labove \textsf
{\begin{small}  \Heuristic  
 $SLE_\kappa$ for $\kappa > 4$ is also very interesting, but the
 paths are not simple.  We restrict in this paper to $\kappa
 \leq 4$.
 \end{small}}
 \lbelow

Chordal $SLE_\kappa$ produces a probability measure on curves,
modulo (increasing) reparametrization, from $0$ to $\infty$.
By conformal transformation, we get a probability measure
on curves connecting distinct boundary points $z,w$ of
simply connected domains $D$.  We will denote this measure
by
  $\mu_D^\#(z,w)$.

\labove \textsf
{\begin{small}  \Heuristic  
To get a measure on parametrized curves, one should use
the natural parametrization as described in \cite{LShef}.
This parametrization satisfies a conformal covariance rule
under conformal transformations.
We would extend our definitions in this paper to 
parametrized curves, but it would not add anything to
our arguments here.  For this reason
we will consider curves modulo reparametrization as in
\cite{Schramm}.  
\end{small}}
\lbelow

{\em Radial $SLE_\kappa$ from $1$ to $0$ in $\Disk$} is
defined  by
the transformations on the disk
\[              \tilde g_t(e^{iz}) = e^{h_t(e^{iz})} \]
where $h_t$ satisfies 
\begin{equation}  \label{radloew}
             \p_th_t(z) = \frac{a}{2} \, \cot_2(h_t(z)
- U_t) , 
\end{equation}
where, as in \cite{Zhanannulus}, we write $\cot_2(z)
= \cot(z/2)$, and $U_t$ is a standard Brownian motion.  By
conformal invariance, this gives a probability measure on
curves $\mu^\#(z,w)$
connecting one boundary point $z$ and one interior point
$w$.

\subsection{Brownian bubble measure}

Our main interest is the Brownian loop measure.  However,
computations of the measure lead to considering excursions
and the 
boundary bubble measure.  

 Suppose
  $D$ is a domain with smooth (not necessarily connected) boundary.  
  For each $z \in \p D, V,V_1 \subset \p D$, we define
  {\em (Brownian) excursion measures}  by  
  \[  \exc_D(z,V) = \int_V H_{\p D}(z,w) \, |dw|, \]
  \[   \exc_D(V_1,V) = \int_{V_1}\exc_D(z,V)
   \, |dz| = \int_{V}\int_{V_1}  H_{\p D}(z,w) \, |dz|\,
   |dw|.\]
  They satisfy the scaling rules
  \[          \exc_D(z,V) = |f'(z)| \, \exc_{f(D)}(f(z),f(V)),\]
  \[           \exc_D(V_1,V) = \exc_{f(D)}(V_1,V).\]
In particular $\exc_D(V_1,V)$ is a conformal invariant and hence
  is well defined even if the boundaries are not smooth.  
  The quantity $\exc_D(z,V)$ needs local smoothness at $z$ to
  be defined.

Boundary bubbles in $D$ are loops rooted at $z \in \p D$
and otherwise staying in $D$.    We review the definitions
(see \cite[Section 5.5]{LBook}). The bubble measure is a $\sigma$-finite
measure on bubbles.  In $\Half$ we can define the measure, by
specifying for each simply connected domain $D \subset \Half$ with
$\dist(0, \Half \setminus D) > 0$, the measure of the set  
of bubbles at $0$ that do not lie in $D$.
  
  \begin{definition}
   If $D \subset \Half$ is a subdomain, $x \in \R$, and
 $\dist(x, \Half \setminus D) > 0$, then  
 \[    \Gamma(x;D) = \Gamma_\Half(x;D) =
 \p_y [H_\Half(z,x) - H_D(z,x)]\mid_{z = x }. \]
 \end{definition}

The quantity $\Gamma(x;D)$
is   the  bubble measure 
(in $\Half$ rooted at $x$) of bubbles that intersect
$\Half \setminus D$.
Alternatively, we can write 
 \begin{equation}  \label{nov4.1}
  \Gamma(x;D) = \lim_{\epsilon \downarrow 0}
\epsilon^{-1} \,      \E^{x+i\epsilon} [H(B_\tau,x)] , 
\end{equation}
where $B$ is a complex Brownian motion and 
$\tau = \tau_D = \inf\{t: B_t \not\in D\}$.  Note that
\[   \Gamma(x;D_1) - \Gamma(x;D_2) =  
\p_y [H_{D_2}(z,x) - H_{D_1}(z,x)]\mid_{z = x }. \] We can similarly
define $\Gamma_{D}(z;D')$ if $z$ is $\p D$-analytic, 
 $D' \subset D$, and $\dist(z,D \setminus D') >0$.  It is 
 defined  as in \eqref{nov4.1}, which we can also
write as
\[   \Gamma_D(z;D') = \int_{D \cap \p D'}
     H_D(w,z) \, d\exc_D(z,w).\]
  It satisfies the  following scaling rule:
if  $f:D \rightarrow f(D)$ is a conformal
transformation,  then
\[             \Gamma_D(z;D') = |f'(z)|^2 \,
  \Gamma_{f(D)}(f(z) ,f(D')) . \]
  If $D\subset \Half $ is simply connected, this quantity
  can be computed \cite[Proposition 5.22]{LBook}: 
  if $f:D \rightarrow \Half$ is
  a conformal transformation with $f(x) = x$, then
  \begin{equation}  \label{schw}             \Gamma(x;D) = - \frac 16 \, Sf(x) , 
  \end{equation}
  where $S$ denotes the Schwarzian derivative.  
Particular cases of importance to us are considered in the following
proposition. 

\begin{proposition} \label{deltaprop}
\[    \Gamma_{\Half}(0;S_r)= \frac{\pi^2}{12r^2}, \]
If $\Gamma(r) = \Gamma_\Disk(1;A_r)$, then 
as $r \rightarrow \infty$,
\begin{equation}  \label{Gammadef}
   \Gamma(r)=  
   \frac{1 +O(e^{-r})}{2r}
                 .
                \end{equation}
Moreover,
\[    \Gamma(r) =  \frac{\pi^2}{12r^2}
               + \delta(r), \]
 where
  \begin{equation}  \label{deltadef}
  \delta(r) = \sum_{k \in \Z \setminus \{0\}}
    [H_{\p \Half}(0,2\pi k) -
    H_{\p S_r}(0,2\pi k)] = \frac 1{12}- \frac{\pi^2}{2r^2} \sum_{k=1}^\infty
\left[ \sinh \left(\frac{  k \pi^2}{r}\right) \right]^{-2}.
  \end{equation}
  In particular,
  \[   \delta(r) = \frac{1 +O(e^{-r})}{2r} - \frac{\pi^2}{12r^2}.\]
\end{proposition}

\begin{proof}  Since $S_r$ is simply connected, we can use
\eqref{schw} with   $f(z) = e^{\pi z/r} -1$ 
to get
  \[     
    \Gamma_{\Half}(0;S_r) = -\frac {Sf(0)}6  = \frac{\pi^2}{12r^2}. \]
    The second equality follows from 
  \eqref{nov4.1}.
 Indeed, as noted previously 
 \[ \lim_{\epsilon \downarrow 0}
   \epsilon^{-1} \,
 \Prob^{1-\epsilon}
     \{B_\tau \in C_r\} = 1/r, \]
and 
  the exact form
 of Poisson kernel in $\Disk$ shows that
 \[    H_\Disk(z,1) = \frac 12 + O(e^{-r}),
 \;\;\;\; z \in C_r. \]

To each Brownian bubble in $\Disk$ rooted at $1$ that intersects $C_r$,
there is a corresponding path in $\Half$ that starts at $0$,
ends at $2 \pi k$ for some integer $k$,  and does not stay in $S_r$.
Only those paths that end at $0$ are Brownian bubbles in $\Half$
rooted at $0$.  Therefore, to compute $\Gamma_\Half(1;S_r)$ we can
subtract the measure of the other bubbles.  To get the measures
of the bubbles to be subtracted we consider the measure of excursions
in $\Half$ minus the measure of excursions in $S_r$.  We therefore
get  
\begin{eqnarray*}   \Gamma_{\Half}(0;S_r)
& = &  \Gamma_{\Disk}(1;A_r) - 
\sum_{k \in \Z \setminus \{0\}}
    [H_{\p\Half}(0,2\pi k) -
    H_{\p S_r}(0,2\pi k)]\\
    & = & \Gamma_{\Disk}(1;A_r) - \delta(r) .
    \end{eqnarray*}
 Using \eqref{nov20.1} and $H_{\p \Half}(0,x)  = x^{-2}$, we
 see that
 \begin{eqnarray*}
   \delta(r) &  = & 2\sum_{k=1}^\infty \left(\frac{1}{(2\pi k)^2} -
\frac{\pi^2}{4r^2} 
\left[ \sinh \left(\frac{  k \pi^2}{r}\right) \right]^{-2}\right)\\
&  = & \frac{1}{12} - \frac{\pi^2}{2r^2} \sum_{k=1}^\infty 
 \left[\sinh \left(\frac{  k \pi^2}{r}\right)\right] ^{-2}.
 \end{eqnarray*}

    \end{proof}

\subsection{Brownian loop measure}

In order to describe $SLE_\kappa$ in other domains, we introduce the
Brownian loop measure as first introduced in \cite{LW}.

\begin{definition}  The {\em rooted Brownian loop  measure}
on $\C$ is the measure on loops given by
\begin{equation}  \label{oct3.1}
         \frac{1}{2 \pi t^2} \, dt \times {\rm area} \times
\nu^{BB} , 
\end{equation}
where $\nu^{BB}$ denotes the probability measure induced
by a Brownian bridge of time duration one at the origin.
\end{definition}

To be more precise, a rooted loop is a continuous function
$\eta:[0,t_\eta]\rightarrow \C$ with $\eta(0) = \eta(t_\eta)$.  Such a
loop can be described by a triple $(t,z,\bar \eta)$ where
$t > 0$ is the time duration, $z = \eta(0)$ is the root,
and $\bar \eta:[0,1] \rightarrow \C$ is a loop of time
duration one starting at the origin.  The rooted loop
measure is obtained by choosing $(t,z,\bar \eta)$ according
to the measure \eqref{oct3.1}.  If $D \subset \C$, the the
rooted loop measure in $D$ is the rooted loop measure in $\C$
restricted to loops that lie in $D$.

\begin{definition}  The rooted loop measure on a domain
$D$ induces a measure on {\em unrooted loops} which we denote
by $m_D$.  We consider this as a measure on unrooted loops
modulo reparametrization. (However, the proof of conformal
invariance requires  considering  the parametrized loops.)

\end{definition}

For the purposes of this paper, we do not need to worry about
the time parametrization of the loops. The fundamental fact
that explains the importance of the loop measure is the following.
   We do this to emphasize that we do not need to assume
   that   
$D$ is simply connected.

\begin{proposition}[Conformal invariance]  If $f:D \rightarrow f(D)$ is
a conformal transformation, then
\[         f \circ m_D = m_{f(D)}. \]
\end{proposition}

\labove \textsf
{\begin{small}  \Heuristic  
 We have stated the proposition for loops, modulo reparametrization.
One can get a similar result for parametrized loops but then one
must change the parametrization as in the conformal invariance
of Brownian motion.
\end{small}} \lbelow

\begin{proof}[Sketch of proof]

  Let $\rho(z,z;t)$ be the measure
on paths associated to Brownian loops at $z$ of time duration $t$.
It is a measure of total mass $p_t(z,z) = (2 \pi t)^{-1}$ that can
be defined using standard Brownian bridge techniques.  Let
\[         \rho(z,z) = \int_0^\infty \rho(z,z;t) \, dt , \]
which is an infinite measure.  For any $D$, we define
$\rho_D(z,z;t), \rho_D(z,z)$ by restriction.   If
$f:D \rightarrow f(D)$ is a conformal transformation, and $\eta$
is a loop in $D$, we write $f \circ \eta$ for the corresponding loop
in $f(D)$ obtained using Brownian scaling on the parametrization.  In other
words, if $\eta$ has time duration $t_\eta$, then $f \circ \eta$
has time duration
\[      \int_0^{t_\eta} |f'(\eta(s))|^2 \, ds . \]
The measure $\rho_D(z,z)$ induces a measure $f \circ \rho_D(z,z)$ by
considering $f \circ \eta$.  Using the conformal invariance
of Brownian motion, one can check that 
\begin{equation}  \label{cc44}
           f \circ \rho_D(z,z) =  \rho_{f(D)}
   (f(z),f(z)).
  \end{equation}
   
   Suppose $h$ is a continuous, nonnegative function on $D$.  Then $h$
induces a measure on (rooted) loops by
\[      \rho_{D,h} = \int_D \rho_D(z,z) \, h(z) \, dA(z) , \]
where $A$ denotes area.  We can also consider this as a measure
on unrooted loops by forgetting the root.  We write $\rho_D$
for $\rho_{D,h}$ with $h \equiv 1$.  Another way to define the
Brownian loop measure $\mu_D$ on unrooted loops is
\[          \frac{d\mu_D}{d\rho_D}(\eta) = \frac 1{t_\eta }, \]
where $t_\gamma$ denotes the time duration of $\gamma$. More
generally,
\[            \frac{d\mu_D}{d\rho_{D,h}}(\eta) = \left[
  \int_0^{t_\eta} h(\eta(s)) \, ds \right]^{-1}. \]
Suppose $h(z) = |f'(z)|^{2}$.  Then
\eqref{cc44}
implies that 
\begin{eqnarray*}
               f \circ \rho_{D,h} & = &
 \int_D  \rho_{f(D)}(f(z),f(z)) \, |f'(z)|^2 \, dA(z)\\
   & = & \int_{f(D)} \rho_{f(D)}(w,w) \, dA(w) =  \rho_{f(D)}.
   \end{eqnarray*}
Also,
\[     \int_0^{t_\eta} h(\eta(s)) \, ds
          = \int_0^{t_\eta} |f'(\eta(s))|^{2} \,
             ds = \frac{1}{t_{f \circ \eta}}. \]
\end{proof}

By construction, $m_D$ also satisfies the restriction property.

\begin{itemize}
\item  If $V_1,V_2$ are subsets, we write either
$      m(V_1,V_2;D) $ or $m_D(V_1,V_2)$
for the $m_D$ measure of the set of loops in $D$ that intersect
both $V_1$ and $V_2$.

\item  Suppose $D \subset \Half$ is a domain (not necessarily
simply connected) with
$\dist(0,\Half \setminus D) > 0$.  Suppose
$\gamma$ satisfies \eqref{chordal} and 
$t < T:= \inf\{t: \dist(\gamma(t),\Half \setminus D) = 0\}$.
Then
\begin{equation}  \label{nov4.2}
  m(\gamma_t,\Half \setminus D; \Half) = a\int_0^t 
    \Gamma(U_s;g_s(D)) \, ds . 
\end{equation}
    If $D$ is simply connected, we can use \eqref{schw} to write
\begin{equation}  \label{nov4.2.alt}
m(\gamma_t,\Half \setminus D; \Half) = - \frac{a}6\int_0^t 
    Sf_s(U_s) \, ds ,
\end{equation}
where $f_s$ is a conformal transformation of $g_s(D)$
onto $\Half$ with $f(U_s) \in \R$.

\end{itemize}

\labove \textsf
{\begin{small}  \Heuristic  
The only functionals of the Brownian loop measure
that we will need are of the type on the left-hand 
side of \eqref{nov4.2}.  We might consider using
the right-hand side of \eqref{nov4.2} as the {\em definition}
of  $m(\gamma_t,\Half \setminus D; \Half)$.
However, it is not so easy to see from this formulation
to see
that if $\gamma$ is a curve in $\Half$ connecting
boundary points $0,x$, then
\[  m(\gamma_t,\Half \setminus D; \Half) = 
 m(\gamma_t^R, \Half \setminus D; \Half), \]
where $\gamma_t^R$ denotes the reversal of the path.
This is immediate from the loop measure description of
the quantity.  
\end{small}}\lbelow

\labove \textsf
{\begin{small}  \Heuristic  
The formula \eqref{nov4.2} comes from a Brownian
bubble analysis of the Brownian loop measure.  Suppose
$\gamma$ is a simple curve from $0$ to $\infty$
in $\Half$.  If $l$ is a loop in $\Half$ that intersects
$\gamma$, we can consider the first time (using the
time scale of $\gamma$) that the loop intersects $\gamma$.
If $l$ intersects $\gamma$ first at time $t$, then
$l$ is a ``boundary bubble'' in $\Half \setminus \gamma_t$ rooted at
$\gamma(t)$. We therefore can write the Brownian loop
measure, restricted to loops intersecting $\gamma$, as
an integral of the Brownian bubble measure in decreasing
family of domains $\Half \setminus \gamma_t$.  We can think of $\gamma$ as an
``exploration process'' for the Brownian loop measure.  This
idea is used in the construction of conformal loop ensembles
by Sheffield and Werner \cite{CLE}.  This exploration
idea is important in our analysis of $SLE_\kappa$ in an
annulus.
\end{small}}\lbelow

Although the Brownian loop measure is a measure on
unrooted loops, it is often convenient to choose roots
of the loops.  For example, if $\eta$ is an unrooted loop,
we can choose the root to be the closest point to the origin,
say $e^{-r + i \theta}$.
(Except for a set of measure zero, this point will be unique).
The rooted loop is then a Brownian bubble in the domain
$ O_r :=
\C \setminus \overline{\Disk_{r}}$.  This is the basis
for the following computation.

\begin{proposition}  \label{persimmon}

Suppose $D \subset \Disk$ is a simply connected domain with
$\dist(0, \p D) > e^{-r}$.  Then,
\[           m(\overline \Disk_r, \Disk \setminus D; \Disk)
   = \frac 1 \pi \int_r^\infty \int_0^{2\pi} \Gamma_{O_s}
   (e^{-s+i\theta}; D) \,  ds \, d\theta, \]
where $O_s = \Disk \setminus \overline \Disk_s .$
   \end{proposition}

\begin{lemma} \label{lemma.jul13}
There exists $c < \infty$ such that the following
is true.  Suppose $D \subset \Disk$ is a simply connected domain
containing the origin and $g: D \rightarrow \Disk$ is the conformal
transformation with $g(0) = 0, g'(0) > 0$.  Suppose that
$r \geq \log g'(0) + 2$.  Let $\phi: g(A_r \cap D) \rightarrow A_s$
be a conformal transformation sending $C_0$ to $C_0$ and
let $h = \phi \circ g$ which maps $A_r \cap D$ onto $A_s$.
 Then if $u = r - \log g'(0)$, $z \in C_r, w \in C_0$,
 \[              |s - u| \leq c \, e^{-u} , \;\;\;\;
 |\phi'(w) - 1| \leq c \, e^{-u} , \;\;\;\;
  |h'(z) - g'(0)| \leq c \,g'(0) \,e^{-u}, \] 
\[              \left|m_\Disk(\overline \Disk_r, \Disk \setminus
D) - \log(r/u) \right| \leq c \, e^{-u}. \]
\end{lemma}

\begin{proof}  The Koebe-$1/4$ theorem applied to $g^{-1}$
shows that $\dist(0,\p D) \geq [4g'(0)]^{-1}.$  Applying the
distortion theorem to $g$ restricted to $\Disk_{u + \frac 32}$,
we see that there exists $c < \infty$ such that if
$|w| \leq e^{-r}$,
\[          |g(w) - g'(0) \, w| \leq c \, e^{-2u},\]
\[           |g'(w) - g'(0)| \leq c\, e^{-u}.\]
In particular,   if $|w| = e^{-r}$, then
\begin{equation}  \label{jul14.1}
      |g(w)| = e^{-u} \, [1 + O(e^{-u})]. 
      \end{equation}
Using this and monotonicity, we see that
\begin{equation}  \label{jul13.1}
        s = u + O(e^{-u}).
        \end{equation}

Let $U$ denote the 
 conformal annulus $g(A_r \cap D)$ so that
 $\phi$ maps $U$ onto 
the annulus $A_s$.  By conformal invariance we see that
$\log |g(z)|/s$ equals the probability that a Brownian
motion starting at $z$ exits $U$ at $g(C_r)$.  However,
we know that the inner boundary of $U$ lies within distance
$O(e^{-2u})$ of $C_u$.  If the Brownian motion gets
that close to $C_u$, the probability that it does not exit
at $C_u$ is $O(e^{-u}/u)$.  Therefore,
\[ \frac{\log |g(z)|}{s} =             \frac{\log |z|}{u} \, [1+O(e^{-u}/u)]. \]  
Hence, from \eqref{jul13.1}, we get
\[          \log |g(z)| = \log |z| \, [1 + O(e^{-u})], \]
which implies $|g'(e^{i\theta})| = 1 + O(e^{-u})$.   The argument
to show that $|h'(e^{-r + i \theta})| = g'(0) [ 1 + O(e^{-u})]$ is similar.


  By conformal invariance and symmetry, 
 \[  \exc_{A_r \cap D}(C_r, \p D)
 = \exc_{A_s}(C_s,C_0) =  \exc_{A_s}(C_0,C_s) = 2 \pi s^{-1} .\]
 Similarly, if
 \[            \hat v(z)  = \Prob^z\{B_{\hat \sigma} \in C_0\} =
    1-\frac{\log |z|}{r} ,\]
  where $\hat \sigma = \inf\{t: B_t \not \in A_r\}$, then
  \[   2 \pi r^{-1} = \exc_{A_r}(C_r,C_0) =
       \int_{C_r} \p_n \hat v(z) \, |dz| .\]
       By the strong Markov property, we can write
 \[       \exc_{A_r}(C_r,C_0)  = \int_{\Disk \cap \p D}
      \left[1 - \frac{\log |z|}{r}\right]
       \, d\exc_{A_r}(C_r,dz).\]
 The term $1 - \frac{\log |z|}{r}$ is the probability that a
 Brownian motion starting at $z$ exits $A_r$ at $C_0$.    
Therefore, using \eqref{jul13.1},
\[    \int_{\Disk \cap \p D} \frac{\log |z|}{r}  \,  d \exc_{A_r}(C_0,dz) =
   2 \pi \, [s^{-1} - r^{-1}] =  2 \pi \, [u^{-1} - r^{-1} + O(e^{-u}/u^2)].\]
  Lemma      \ref{dec4.1} implies  that if $V
  \subset \p D$ and $z,w \in C_r$,
  \[     \exc_{A_r}(z,V) = \exc_{A_r}(w,V) \, [1 + O(e^{-u})], \]
  and hence
  \[          \exc_{A_r}(z,V) = \frac{1}{2\pi} \, \exc_{A_r}(C_0,V)
  \, [1 + O(e^{-u})].\]
  Lemma \ref{dec4.1} can also be used to see that if $w \in \Disk
  \cap \p D, z
  \in C_r$,
  \[           H_{A_r}(w,C_r) = \frac 12 \, \frac{\log |w|}{r} \,
   [1+O(e^{-u})].\]
   Therefore, using \eqref{jul13.1},
   \[  \Gamma_{A_r}(z,A_r \cap D) =
      \int_{\Disk \cap \p D} H_{A_r}(w,z) \, d\exc_{A_r \cap D}
        (z,w) = \frac{u^{-1} - r^{-1}}{2}  \, [1 + O(e^{-u})]. \]
      
 From   Proposition \ref{persimmon} we know that
 the quantity we are interested in can be written as
 \[   \frac 1 \pi \int_0^{2\pi} \int_{r}^\infty 
   \Gamma_{A_t}(e^{-t + i \theta}; A_t \cap D)
     \, dt \, d \theta  = 
       \int_0^\infty \left[\frac{1}{u+ t} - \frac{1}{r+t}\right]
        \, [1 + O(e^{-u-t})] \,dt.\]
       By computing the  integral we see that this quantity 
       equals
    \[        \log(r/u) + O(e^{-u}).\]

\end{proof}
We will need to consider the Brownian loop measure in
an annulus.  If we fix the origin as a marked point, we can
divide loops into two sets: those with nonzero winding
number around zero and those with zero winding number.
   If $A$ is a conformal annulus
such that $0$ and $\infty$ lie in different components
of $A^c$, then the measure of the set of loops in $A$
with nonzero winding number is finite.  It is a conformal
invariant which we calculate in the next proposition.

\begin{proposition}  \label{windloops}
Let $m^*(r)$ denote the Brownian  loop measure of loops
in $A_r$ that have nonzero winding number.  Then
\[  m^*(r) = \frac{r}{6} - 2\int_0^r \delta(s) \, ds, \]
where $\delta(s)$ is defined as in \eqref{deltadef}.  In
particular,  
there exists $C > 0$ such that as $r \rightarrow 
\infty$,
\begin{equation}  \label{nov20.2}
   e^{m^*(r)} =   C\, r^{-1} \, e^{r/6}\,[1+O(r^{-1})]   .
 \end{equation}
\end{proposition}

\begin{proof}  By focusing on the point of the loop
of largest radius (see the appendix of \cite{LJSP}),
 we can give the expression
\[  m^*(r) = 2\pi \int_0^r \sum_{k \in \Z \setminus \{0\}}
             \frac 1 {\pi} \, H_{S_s}(0,2 \pi k)
             = \frac 16 - 2\delta(r)  . \]
 Proposition \ref{deltaprop} implies that
 there exists $c$ such that 
 \[    m^*(r) = \frac{r}{6} - \log r + c + O(r^{-1}), \;\;\;\;
 r \rightarrow \infty , \]
 from which \eqref{nov20.2} follows with $C = e^c$.

\end{proof}

\begin{corollary} \label{torontocor}  $\;$

\begin{itemize}

\item
Suppose $D\subset \Disk$ is a simply connected domain containing
the origin and suppose that $\dist(0,\p D) > e^{-r}$.   Let
$0 \leq s < r$ be defined by saying that the annulus
$D \setminus \overline \Disk_r$ is conformally equivalent
to $A_s$.  Then
the Brownian loop measure of loops in $A_r$ of nonzero winding
number that intersect $\Disk \setminus D$ is 
$                       m^*(r) - m^*(s)$.

\item  Under the same assumptions, the Brownian loop measure
of loops in $\Disk$ of nonzero winding number that intersect
$\Disk \setminus D$ is $\log g'(0)/6$ where $g:D \rightarrow 
\Disk$ is the conformal transformation with $g(0) = 0,g'(0) > 0$.

\end{itemize}

\end{corollary}

\begin{proof}  The first assertion follows immediately and the second
is obtained by considering comparing $\Disk \setminus A_r$
and $D \setminus A_r$ as $r \rightarrow \infty$.  \end{proof}

We ill use the following estimate in the discussion in the
next section but it will not figure in our main results.
See  \cite{Lnote} for a proof.

\begin{proposition}   \label{prop.nov30}
Let $k(r)$ denote the $m_{\Disk_{-r}}$ measure of loops that
intersect both $A_{-r} \setminus A_{-r+1}$ and $\Disk$.
Let $k'(r)$ be the measure of such loops that do not
separate the origin from $C_0$.  Then as $r \rightarrow 
\infty$,
\[  k(r) = r^{-1} + O(r^{-2}), \;\;\;\;
  k'(r) = O(r^{-2}). \]
In particular, if $V_1,V_2$ are disjoint
compact sets,  
 then there exists $\Lambda(V_1,V_2)$ such that
 as $r \rightarrow \infty$,
\[  m_{\Disk_{-r}}(V_1,V_2) =
    \log r - \Lambda(V_1,V_2) + o(1).\] 
\end{proposition}

\subsection{Chordal $SLE_\kappa$ in simply connected domains}
   \label{simplesec}
   
We will review two equivalent ways to construct $SLE_\kappa$
in simply connected domains for $\kappa  = 2/a\leq 4$. 
See \cite{LSWrest,LW,LBook,LPark} for more details.
 Suppose
$D$ is a simply connected subdomain of $\Half$
with $\dist(0,\Half \setminus D) > 0$.  Let $w$ be 
a nonzero $\p D$-analytic point; we allow $w=\infty$
as a possibility.  Let $\Phi: D \rightarrow \Half$
be the unique conformal transformation with $\Phi(0) =0,
\Phi(w) = \infty, |\Phi'(w)| = 1$.  Here we are using
the conventions about derivatives as discussed in Section
\ref{poissec}.  
The most important example for this paper is $D = S_r$
and $w = x + i r$ for some $x \in \R$.

Let   $g_t$ be the solution
of the Loewner equation
\[       \p_t g_t(z) = \frac{a}{g_t(z) - U_t} , \;\;\;\;
    g_0(z) = z , \]
where  $U_t = - B_t$ is a standard Brownian motion
defined on the probability space $(\Omega,\F,\Prob)$.  Then
the corresponding curve $\gamma$ is $SLE_\kappa$ in 
$\Half$ from $0$ to $\infty$ which with $\Prob$-probability
one is a simple curve with $\gamma(0,\infty)  \subset \Half$.

Let
\[   T = T_D= \inf\{t > 0 : \gamma(t) \not \in D \} . \]
For $t < T$,  let
\[     w_t = g_t(w), \;\;\;\;
      \gamma_t^* =  \Phi \circ \gamma_t ,\]
and let $\hat g_t$ be the unique conformal transformation of 
$\Half \setminus \gamma_t^*$ onto $\Half$ with $\hat g_t(z)
= z + o(1)$ as $z \rightarrow \infty$.  Let
\[         \Phi_t =  \hat g_t \circ \Phi \circ g_t^{-1} . \]
Then $\hat g_t$ satisfies the Loewner equation
\[   \p_t \hat g_t(z) = \frac{a \, \Phi_t'(U_t)^2}
                   {\hat g_t(z) - \hat U_t}, \;\;\;\;
  \hat g_0(z) = z , \]
where $\hat U_t = \hat g_t(\gamma^*(t)) = \Phi_t(U_t)$. 
Then $\Phi_t$ is the unique conformal
transformaton of $g_t(D\setminus \gamma_t)$ onto $\Half$ with $\Phi_t(U_t)=
\hat U_t, \Phi_t(w_t) = \infty, |\Phi_t'(w_t)| =
  |g_t'(w)|^{-1} . $
Moreover, using only the Loewner equation, one can show
that
\begin{equation}  \label{dots}
    \dot \Phi_t(U_t) = -\frac{3b}{2} \, \Phi_t''(U_t), \;\;\;\;
   \dot \Phi_t'(U_t) =  
\frac{a\Phi''_t(U_t)^2}{4
  \Phi'_t(U_t)} - \frac{2a \, \Phi'''_t(U_t)}{3} 
\end{equation}
where $\dot \Phi_t(U_t), \dot \Phi_t'(U_t)$ denote
 $\p_t \Phi_t(x), \p_t \Phi_t'(x)$ evaluated
at $x = U_t$. 

Let
\begin{equation}  \label{kt}
 H_t =  H_{\p g_t(D \setminus \gamma_t)}(x,w_t), \quad
  K_t = |g_t'(w)|^b \, H_t(U_t)^b =
     \Phi_t'(U_t)^b . 
\end{equation}
   The second equality for $K_t$ 
    follows from the scaling
rule for the Poisson
kernel.  A straightforward It\^o's formula calculation
using \eqref{dots} shows that
  \[    dK_t = K_t \, \left[ \frac{a\cent}{12} \, S\Phi_t
(U_t) \, dt + \frac{b\, H_t'(U_t)}{H_t(U_t)} \,
  dU_t \right], \]
where $S$ denotes the Schwarzian derivative.
Let 
\begin{eqnarray*}
  M_t & = & \exp \left\{-\frac {a\cent} {12} \int_0^t
  S\Phi_s(U_s) \, ds \right\} \, K_t\\
&  = & \exp \left\{ \frac \cent 2 \, m_\Half(
\gamma_t,\Half \setminus D) \right\}\,|g_t'(w)|^b\,
 H_t(U_t)^b .
 \end{eqnarray*}
 (To check the second equality, recall that we have
 parametrized so that  $\hcap(\gamma_t)
 = at$.)
Then $M_t$ is a local 
martingale satisfying
\[    dM_t = \frac{b\, H_t'(U_t)}{H_t(U_t)}
  \, M_t \, dU_t = \frac{b\, \Phi_t''(U_t)}{\Phi_t'(U_t)}
  \, M_t \, dU_t. \]

We can use Girsanov theorem to define a new probability
measure $\Prob^*$ obtained by weighting by the local
martingale $M_t$.  (The Girsanov theorem is stated for
nonnegative martingales; since we only have a local martingale,
we need to use stopping times.  However, as long as $t < T$,
there is no problem.)  The Girsanov theorem states that
\begin{equation}  \label{nov7.1}
           dU_t =  \frac{b\, H_t'(U_t)}{H_t(U_t)}\, dt
  + dW_t, \;\;\;\;  t < T , 
\end{equation}
where $W_t$ is a standard Brownian motion with respect
to $\Prob^*$.  

Another application of It\^o's formula using \eqref{dots}
shows that if
$U_t$ satisfies \eqref{nov7.1}, then $\hat U_t = \Phi_t(U_t)$ satisfies
\[          d \hat U_t = \Phi_t'(U_t) \, dW_t. \]
The upshot is that, {\em with respect to
 the measure $\Prob^*$},
$\eta_t$ has the distribution of (a time change of)
$SLE_\kappa$ from $0$ to $\infty$ in $\Half$.  Since
$\gamma_t = \Phi^{-1} \circ \eta_t$, this implies that
with respect to $\Prob^*$,
  $\gamma_t$ has the distribution of $SLE_\kappa$
from $0$ to $w$ in $D$.
 The Girsanov transformation \eqref{nov7.1}
 is sufficent
for understanding the probability measure $\mu^\#_D(0,w)$.  Note
that it is determined by the logarithmic derivative of $H_t$; the
``compensator'' terms do not need to be computed.

The example of importance in this paper is $D = S_r$ and
$w = x+ ir$.  It will suffice for us to consider the
probability measure $\mu^\#_{S_r}(0,w)$. The
drift term in \eqref{nov7.1} is somewhat complicated
to write down; however, at time $t=0$,  we can use
\eqref{jan31.1} to see that it equals
$b \funthree(r,x)$ where 
\begin{equation} \label{fun3}
 \funthree(r,x) = 
\frac{ H_{\p S_r}'(0,x+ir)}{H_{\p S_r}(0,x+ir)}=
 \frac{ \pi 
 }{r} \, \tanh \left(\frac {\pi x }{2r}\right)  ,
 \end{equation}
 where the prime denotes derivative in
 the {\em first} component.
 This measure is the same (modulo time change) as
 the conformal image of $SLE_\kappa$ from $0$ to $\infty$
 in $\Half$; in particular, with probability one, the
 path leaves $S_r$ at $w$.

%
%
%
%
%

In analyzing annulus $SLE_\kappa$ we will be studying measures
that will turn out to be absolutely continuous with  respect to
$\mu_{S_r}^\#(0,x+ir)$.  To review the issues that we need to
address, let us recall the case of $SLE_\kappa$ from $0$ to
$\infty$ in a simply connected domain $D$ with $\Half \setminus D$
bounded and $\dist(0,\Half \setminus D) > 0$.  In this case,
when we weight by the appropriate local martingale $M_t$,
 then with $\Prob^*$-probability one, $T = \infty$
 and $\gamma(t) \rightarrow
 \infty$.  
If $T = \infty$ and $\gamma(t) \rightarrow \infty$, 
then a deterministic estiamte gives
\[    M_\infty = \exp \left\{\frac \cent 2 \, m_\Half(\gamma,
  \Half \setminus \gamma) \right\}\, 1\{\gamma \subset D\}, \]
and since this happens with $\Prob^*$-probability one,
\begin{equation}  \label{jan31.2}
      \E[M_\infty] = M_0 = \Phi'(0)^b . 
      \end{equation}

\labove \textsf
{\begin{small}  \Heuristic  
The argument we will use for the annulus is similar
to the proof for simply connected domains, so it is worth
reviewing the main steps.  Suppose $D$ is a simply connected
domain with $\Half \setminus D$ bounded
and $w = \infty$.
Here we were able to guess  the exact form  
for the partition function for $\mu_D(0,\infty)$, $\Phi_D'
(0)^b$.  Direct It\^o's formula calculation shows that $M_t$
as above
gives a local martingale.  However, to justify
\eqref{jan31.2}, we need that fact that  
the curve \textit{weighted by the local martingale}
goes to infinity without leaving the domain.  This gives the
necessary  ``uniform integrability''.\\ 
In the annulus case, we will consider two measures on
curves from $0$ to $w=x+ir$ in $S_r$. 
We will use the Feynman-Kac theorem
applied to a slightly different process to give a candidate
for the partition function.  Although we will not have an
explicit form of it, we will know that it satisfies a certain
PDE and hence  gives us a local martingale.  Having
a local martingale is not sufficient; we will also
need to show that the process weighted by the local martingale
leaves  the domain at $w$.  This will give the analogue of
\eqref{jan31.2}.  The argument for the annulus, as well as
the argument here, will require $\kappa \leq 4$.
\end{small}}
\lbelow

\subsection{Shrinking domains}  \label{shrinksec}

We will need a generalization of this where the
domain $D$ is replaced with a decreasing family
of domains $\{D_t: t > 0\}$.  Although what we describe
can be done more generally, we will restrict to the case
that we need in this paper.  This will lead to
a process that we call {\em locally chordal $SLE_\kappa$
in an annulus.} 

Let $D = S_r$
and $w \in \p S_r \setminus \{0\}$.  
(The case $S_\infty = \Half, w = \infty$ 
corresponds to radial $SLE$ and is
discussed in the next subsection.)
Let
\begin{equation}  \label{nov21.6}
        \tilde \gamma_t = \bigcup_{k \in \Z \setminus
\{0\}} (\gamma_t + 2 \pi k), \;\;\;\;
            D_t = D \setminus \tilde \gamma_t. 
      \end{equation}  and
\[         \hat D_t = D \setminus (\tilde \gamma_t \cup \gamma_t)
 =  \psi^{-1}[\Disk \setminus \eta_t],\]
where
$   \eta_t = \psi \circ \gamma_t.$
In other words, when we slit the domain $D= S_r$ by
$\gamma_t$ we also add slits at the $2\pi k$ translates
of $\gamma_t$. 

Let $T$ denote
the first $t > 0$ such that either $\gamma(t)
\in \p S_r$ or $\eta_t$ disconnects the origin from 
the unit circle,
\[  T = \inf\{t > 0 : \gamma_t \not\subset D_t\}.\]
  Let $\tilde D_t = g_t(\hat D_t)$, and, as before,
$U_t = g_t(\gamma(t))$.  We want to study the process that
evolves at time $t$ like chordal $SLE_\kappa$ from $\gamma(t)$
to $w$ in the domain $D_t$.   Equivalently, the process after
conformal transformation by $g_t$ evolves like chordal $SLE_\kappa$
from $U_t$ to $w_t = g_t(w)$ in $\tilde D_t$.  The latter
process can be defined in two equivalent ways.  Let $H_t(x)
= H_{\p g_t(D \setminus \gamma_t)}(x,w_t)$ as in the previous
section and let 
\[ \tilde H_t(x) = H_{\p \tilde D_t}(x,w_t), \quad
  Q_t(x) =  \frac{\tilde H_t(x)}{H_t(x)}. \]
 The process can be considered as either of the following.
\begin{itemize}

\item   $SLE_\kappa$ in
$\Half$ from $0$ to $\infty$ 
weighted by  $ \tilde H_t(U_t)^b .$

\item  $SLE_\kappa$ in $S_r$ from $0$ to $w$ weighted
by  $Q_t(U_t)^b$.
  
  \end{itemize}

\labove \textsf
{\begin{small}  \Heuristic  If $J_t$ is a positive process,
then 
 ``weighting by $J_t$''
 is in the sense of the Girsanov thoerem.  If
$J_t$ satisfies
\[             dJ_t = J_t \left[R_t \, dt + A_t \, dU_t\right] . \]
then
\[               N_t := \exp \left\{-  \int_0^t R_s  \, ds \right\}
 \, J_t , \]
 is a local martingale satisfying
\[                  dN_t = A_t \, N_t \, dU_t . \]
When we use the Girsanov theorem (using stopping times so that the local
martingale is a martingale), then 
\[                     dU_t = A_t \, dt + dW_t , \]
where $W_t$ is a Brownian motion in the new measure.
\end{small}}
\lbelow


%

Let
\begin{equation}  \label{dderiv}
\Delta_t = \frac{ \dot Q_t(U_t)}
   {Q_t(U_t)},
\end{equation}
where $\dot Q_t(U_t)$ denotes $\p_t Q_t(x)$ evaluated
at $x = U_t$.
Our assumptions allow us to conclude
that $\Delta_t$ is well defined and continuous
for $t < T$.

As in  \eqref{kt}, we define  
\[  K_t= |g_t'(w)|^b \, \tilde H_t(U_t)^b
 = |g_t'(w)|^b \, H_t(U_t)^b \, Q_t(U_t)^b . \] 
 Using the previous calculation and the chain rule, we see that
$K_t$ satisfies
\[  dK_t = K_t \, \left[\left(- b \Delta_t  +
\frac{a\cent}{12} \, S\Phi_t
(U_t)\right) \, dt + \frac{b\, \tilde H_t'(U_t)}{\tilde H_t(U_t)} \,
  dU_t \right].\]
If
\[   C_t =  \exp \left\{\int_0^t \Delta_s
 \, ds\right\},\]
\[  M_t = 
   C_t^b \,  \exp \left\{-\frac {a\cent}{12} \int_0^t
  S\Phi_s(U_s) \, ds \right\} \, K_t,
 \] 
then $M_t$  is a local 
martingale satisfying
\[    dM_t = \frac{b\, \tilde H_t'(U_t)}{\tilde H_t(U_t)}
  \, M_t \, dU_t. \]

The term 
\[  
- \frac a6 \int_0^t
  S\Phi_s(U_s) \, ds  \]
can be interpreted in terms of Brownian loops, but we need
to be careful.  At time $s$, $-S \Phi_s(U_s)/6$ represents
the measure of Brownian bubbles in $\Half$ rooted at $U_s$
that intersect $g_s(D_s)$.  
For every Brownian
loop $l$, let $s(l)$ be the smallest $s$ such that
$s(l) \cap \gamma_s \neq \emptyset$.  Then  
\[ - \frac a 6 \int_0^t
  S\Phi_s(U_s) \, ds  = 
   \tilde m_t   , \]
 where $\tilde m_t = \log \tilde \Lambda_t$ is the Brownian loop measure
of  $l$
in $D$ 
with $s(l) \leq t$ and $l \cap D \setminus D_{s(l)}
 \neq \emptyset $. 
 Then the local martingale is
\[            M_t = C_t^b \, \tilde \Lambda_t^{\cent/2} \, H_t(U_t)^b\,
  \, Q_t^b  =
    C_t^b \, \tilde \Lambda_t^{\cent/2} \, \tilde H_t(U_t)^b. \]

Note that the only term in $M_t$ that has nontrivial
quadratic variation is $\tilde H_t(U_t)^b$.  
  Therefore, 
when we weight by the  local martingale, the process looks locally
like $SLE_\kappa$ from $\gamma(t)$ to $w$ in $D_t$.  We
call it {\em locally chordal $SLE_\kappa$} (we have defined it only
for $\kappa \leq 4$.)
This gives a probability measure on paths starting at $0$ in $S_r$.
We will use $\kappa \leq 4$ to show that with probability one the
paths leave $S_r$ at $w$.  We can also view the paths as living 
in the annulus $A_r$ and going from $1$ to $e^{-r+ix}$ with a
known total winding number.  In Section \ref{anneqsec} we will
use an annulus reparametrization of the curve.



\subsection{Radial $SLE_\kappa$ raised to $\Half$}  \label{fradsec}

Suppose $D$ is a simply connected domain, $z \in \p D$,
$w \in D$, and $\p D$ is locally analtyic at $z$.  Radial $SLE_\kappa$
in $D$ from $z$ to $w$ is a measure on paths
\[           \mu_D(z,w) = \tmass_D(z,w) \, \mu_D^\#(z,w) , \]
that satisfies the conformal covariance rule
\[           f \circ \mu_D(z,w) = |f'(z)|^{ b} \, |f'(w)|^{\tilde b}
 \, \mu_{f(D)}^\#(f(z),f(w)) . \]
 The conformal covariance rule determine the total mass up to a multiplicative
 constant and for convenience we choose the constant so that
 $\tmass_\Disk(1,0) = 1$.  
 
 To obtain the probability measure $\mu_D^\#(0,w)$ where $w \in \Half$,
we weight  chordal $SLE_\kappa$
by a particular local martingale. 
 Let $g_t$ be the
conformal maps for chordal $SLE_\kappa$ from $0$ to $\infty$,
and let $w \in \Half$.  Let $Z_t = g_t(w) - U_t$ and
\[            M_t = |g_t'(w)|^{\tilde b} \, H_\Half(
           Z_t,U_t)^b , \]
where $b,\tilde b$ are the boundary and interior scaling
exponents, respectively, as in \eqref{bandc}.  Then $M_t$
is a local martingale and the measure on the paths obtained
by weighting by this local martingale is that of 
radial $SLE_\kappa$.   In the weighted measure, the
path stops at finite (half plane capacity) time $T_w$ at which 
$\gamma(T_w) = w$.   This determines the probability measure
$\mu_\Half^\#(0,w)$ and conformal invariance 
determines the measure for all simply connected $D$.  Although
this is not the same definition as originally given by Schramm
\cite{Schramm}, the Girsanov theorem shows that it is equivalent.

One can also understand the relationship between radial 
and chordal $SLE_\kappa$ using the Brownian
loop measure.  Suppose
that $\gamma_t$ is a simple
curve in $\Half$ starting at the origin
and let $\eta_t = \psi \circ \gamma_t$.  We will assume that $t$
is small so that $\eta_t$ is also simple. Let
\[  \tilde h_t: \Disk \setminus \eta_t \rightarrow \Disk\]
be the conformal transformation with $\tilde h_t'(0) > 0$ and
suppose the curve has been parametrized so that $h_t'(0)
=e^{t}$. Let $g_t: \Half \setminus \gamma_t$ be the usual
conformal transformation with driving function
$U_t$; one can show that
\[             \p_t \hcap[\gamma_t] \mid_{t=0} = 2 , \]
which is why this is a standard choice of parametrization
for chordal $SLE$.
Let $\tilde \gamma_t, \hat \gamma_t$ be as in the previous
subsection and let $h_t$ be a conformal transformation
$h_t : \Half \setminus \hat \gamma_t \rightarrow \Half$
 such that $\psi(h_t(z))=
\tilde h_t(\psi(z))$.  This transformation is determined uniquely by
requiring that
\[          h_t(iy) = i[y - t] + o(1) , \;\;\;\;
  y \rightarrow \infty . \]
   We define $\phi_t$ by
\[                 h_t = \phi_t \circ g_t.\]

Let $\mu_1,\mu_2$ denote $\mu_\Half(0,\infty)$ and $\mu_\Disk(1,0)$.
The latter measure can be viewed as a measure on curves $\gamma_t$
by pulling back by $\psi$. (Note that $|\psi'(0)| = 1$ so the
derivative factor in the scaling rule equals one.) 
 We view these measures on the initial
segment $\gamma_t$.  The measure $\mu_2$ is supported on curves
such that $\gamma_t \cap \tilde \gamma_t \neq \emptyset$.
 Note that $\mu_2 \ll \mu_1$, and let $Y_t(\gamma_t)$
denote the Radon-Nikodym derivative so that $d \mu _2 = Y \, d\mu_1$.
Let $\tmass^*$ denote the partition function for the raised radial
$SLE$; in particular, $\tmass_\Half^*(0,\infty) = 1$.
 
  Although the
loop measure is conformally invariant, we must be careful here
because $\psi: \Half \rightarrow \Disk$ is not one-to-one.  Indeed,
each loop $l'$ in $\Disk$ has an infinite number of preimages in
$\Half$.  If $l'$ is a loop in $\Disk$ that intersects $\eta_t$, we
can specify a unique preimage by considering the smallest $s$ such
that $\eta_s \in l'$ and then rooting $l'$ at $\eta_s$.  We associate
to $l'$ the corresponding loop $l$ in $\Half$ rooted at $\gamma_s$.

Also, the loops of nonzero winding number in $\Disk$
have preimages that are not loops in $\Half$.  Since the
paths have been parametrized so that $\tilde h'(0) = e^{t}$,
Corollary \ref{torontocor}
implies that
the measure of such loops is deterministic
and equal to $t/6$. 
Using this idea, we get the formal expression
\[    Y(\gamma_t) = C_t 
 \,  \exp \left\{\frac \cent 2 \,[\hat m(\gamma_t)- (t/6)]\right\}
  \, \frac{\tmass_{\Half \setminus \hat \gamma_t}^*(\gamma(t),0)}{\tmass_{\Half
   \setminus \gamma_t}(\gamma(t),0)}.\]
   Here $C_t$ is a normalization to make this a probability measure
   and 
 $\hat m(\gamma_t)$ denotes the measure of loops $l$ in $\Half$
that intersect $\gamma_t$ with the following property.
\begin{itemize}
\item  Let $s$ be the smallest time with $\gamma_t \in l$.  Then
\[    l \cap \tilde \gamma_s \neq \emptyset . \]
In other words, the loop hits a translate of $\gamma_t$ before it
hits $\gamma_t$ where time is measured on the curve $\gamma_t$.
\end{itemize}
The ratio of partition functions is only formal but we can make
sense of it by writing
\[ \frac{\tmass_{\Half \setminus \hat \gamma_t}^*(\gamma(t),\infty)}{\tmass_{\Half
   \setminus \gamma_t}(\gamma(t),\infty)}=
    \frac{\tmass_{\Half \setminus \hat \gamma_t}^*(\gamma(t),\infty)}
    {\tmass_{\Half
   \setminus \hat \gamma_t}(\gamma(t),\infty)}\; \frac{\tmass_{\Half \setminus \hat \gamma_t}(\gamma(t),\infty)}{\tmass_{\Half
   \setminus \gamma_t}(\gamma(t),\infty)}  .\]
The first term on the right equals one since, formally,
\[  \frac{\tmass_{\Half \setminus \hat \gamma_t}^*(\gamma(t),\infty)}{\tmass_{\Half
   \setminus \hat \gamma_t}(\gamma(t),\infty)}
    = \frac{|h_t'(\gamma(t))|^b \, \tmass_\Half^*(h_t(\gamma(t)),\infty)}
      {|h_t'(\gamma(t))|^b \, \tmass_\Half(h_t(\gamma(t)),\infty)}=1 .\]
For the second term, we use the formal computation 
\[  
 \frac{\tmass_{\Half \setminus \hat \gamma_t}(\gamma(t),\infty)}{\tmass_{\Half
   \setminus \gamma_t}(\gamma(t),\infty)} = \frac{|g_t'(\gamma(t))|^b
    \, \tmass_{g_t(\Half \setminus \hat \gamma_t)}(g_t(\gamma(t)),\infty)}
    {|g_t'(\gamma(t))|^b \, \tmass_\Half(U_t,\infty)} =
     \tmass_{g_t(\Half \setminus \hat \gamma_t)}(U_t,\infty), \]
and conformal covariance,
\[ \tmass_{g_t(\Half \setminus \hat \gamma_t)}(U_t,\infty)=
 \phi_t'(U_t)^b . \]
 Therefore,
 \[            Y_t(\gamma_t) =
 C_t \, e^{-\cent t/12}\,
  \exp \left\{\frac \cent 2 \, \hat m(\gamma_t)\right\}
  \, \phi_t'(U_t)^b . \]
  
This is a local martingale (and a martingale for $\kappa \leq 4$)
for chordal $SLE_\kappa$ and when we weight by the martingale we get
locally chordal $SLE_\kappa$ from $\gamma(t)$ to $\infty$
in $\Half \setminus \hat \gamma_t$.  Although we are considering
chordal $SLE_\kappa$, we are using the radial parametrization.
 This is the same as
radial $SLE_\kappa$ viewed on the covering
space $\Half$.  It remains to find the normalization factor
$C_t$. Since the weighted measure locally looks like chordal $SLE_\kappa$
in the infinitely slit domain and hence after mapping by $h_t$
looks like chordal $SLE_\kappa$, we get that $C_t = e^{\tilde b t}$
for some $\tilde b$.  To find the exponent we need only differentiate
at $0$.
The measure of loops that hit both $\gamma_t$ and a translate of $\gamma_t$
is of order $t^2$ and hence
\[   \p_t  \hat m(\gamma_t)\, \big|_{t=0} = 0.\]
We claim that 
\begin{equation}  \label{toronto.1}           
      \p_t\phi_t'(U_t) \,\big|_{t=0} = -\frac 1 6, \end{equation}
and hence
\[                 \tilde b = \frac{\cent}{12} + \frac b 6.\]

Let us sketch the proof of \eqref{toronto.1}.  We write ``small error''
for errors that are $o(t)$ as $t \downarrow 0$.  The quantity $\phi_t'(U_t)$
is the probability that a Brownian excursion in $\Half \setminus \hat \gamma_t$
from $\gamma(t)$ to $\infty$ does not hit $\tilde \gamma_t$.  Up to
small error, it is the probability that an excursion in $\Half$ from $0$
to $\infty$ does not hit $\tilde \gamma_t$.  The set $\tilde \gamma_t$
is a union of curves of half-plane capacity $2t$ rooted at the
points $2\pi k$, $k \in \Z \setminus \{0\}$.  The probability that an
excursion hits the translate $\gamma_t + 2 \pi k$ is exactly
\[              \p_y q(iy) \]
where $q(z) = \E^z[\Im[B_\tau]]$, $B$ is a standard Brownian motion
and $\tau$ is the first time that it leaves $\Half \setminus [\gamma_t
+ 2 \pi k]$.  As $t \downarrow 0$, up to small error this equals
\[                   \frac{1}{(2\pi k)^2} \, \hcap[\gamma_t]
 = \frac{t}{2 \pi^2}  . \]
The probability of hitting more than one translate is $O(t^2)$, and hence,
up to small error, the probability that the excursion hits $\tilde \gamma_t$
is 
\[          \sum_{k \in \Z \setminus \{0\}} 
   \frac{t}{2 \pi^2} = \frac t6.\]

\labove \textsf
{\begin{small}  \Heuristic  
In the last computation we use the fact that for a small curve rooted
at $x \in \R$, the expected value of $\Im(B_\tau)$ is given by
the half-plane capacity times a multiplicative constant of the Poisson
kernel.  In order to keep track of constants (perhaps made especially
confusing by our definition of $H$),  it is useful to remember
that for large $y$ if $D = \Half \setminus \overline \Disk$,
\[        \E^{iy}[\Im(B_\tau)] \sim \, \frac{1}{y} = H_{\Half}(y,0)
. \]
Hence, we get the general relation,
\[         \E^{z}[\Im(B_\tau)] \sim   H_\Half(z,x)\, \hcap[\gamma_t].\]
\end{small}}
\lbelow

The estimate \eqref{tomato} is done similarly.  In this case, the
probability that an excursion from $0$ to $x+ir$ in $S_r$ hits
the translate $\gamma_t + 2 \pi k$ is exactly,
$\p_y q(y)$ where 
\[       q(z) = \frac{\E^z[H_{S_r}(B_\tau,x+ir)]}
             {H_{\p S_r}(0,x+ir)}. \]
Here $\tau$ is the first time that the Brownian motion leaves
$S_r \setminus [\gamma_t + 2 \pi k]$.
Up to small error, if $B_\tau \not\in \p S_r$,
\[                H_{S_r}(B_\tau,x+ir) = \Im[B_\tau]
  \, H_{S_r}(2\pi k,x +ir) .\]
  Also, as $y \downarrow 0$,
 \[  \p_y \E^{iy} \left[\Im(B_\tau)\right]\mid_{y=0} =  \hcap[\gamma_t] \,
       H_{S_r}(0,2 \pi k)\, [1 + o(1)]. \]

 \subsection{Annulus Loewner equation}  \label{anneqsec}

We will need to consider the annulus Loewner equation
which is similar to the    chordal
equation \eqref{chordal}.  We will need to define
the 
  annulus
equation in the covering space $S_r$. We start
with some defintions.  Assume $U:[0,\infty)
\longrightarrow \R$ is continuous with $U_0 = 0$ and such that
the chordal equation \eqref{chordal} produces a simple
curve. 
Recall that $\psi(z) = e^{iz}$, $\tau_r = \inf\{t: \Im \gamma(t) = r\}$,
 and let $\eta_t = \psi \circ \gamma_t$.
Let
\[          \tilde \gamma_t = \bigcup_{k \in \Z \setminus \{ 0\}}
                      (\gamma_t + 2 \pi k), \;\;\;\;
              \hat \gamma_t = \gamma_t \cup \tilde \gamma_t, \]
 \[            T = \inf\{t: \gamma_t \cap \tilde \gamma_t \ne \emptyset\}. \]
Equivalently,
 $T$ is the first time that the curve
$\eta_t$ is not simple.  Note that $T \neq \tau_r$ for each $r$; indeed,
by the definition of $T$, there must be an $s < T$ with $\Im \gamma(s)
= \Im \gamma(T)$.  Let
\[                   S_{r,t} = S_r \setminus \gamma_t, \;\;\;\;
         \hat S_{r,t} = S_r \setminus \hat \gamma_t. \]

If $t < T \wedge \tau_r$, there is a unique $r(t)
 = r(t,\gamma_t) \in (0,r]$ such that there
is a conformal transformation
\[              \bar
 h_t: A_r \setminus \eta_t \rightarrow A_{r(t)}, \]
 with $\bar h_t(C_0) = C_0$.
The transformation $\bar h_t$ is unique up to a rotation.  This transformation
can be raised to the covering space $S_r$ to give a conformal
transformation
\[                 h_t:  \hat S_{r,t} \rightarrow S_{r(t)} \]
with $h_t(\pm \infty) = \pm \infty.$ This transformation
is unique up to a real translation, and we specify it uniquely
by requiring
\[                      h_t(U_t) = U_t. \]
We define $\phi_t$ by
\[                      h_t = \phi_t \circ g_t. \]
Note that $\phi_t$ is the unique conformal transformation of $g_t(S_{r,t})$
onto $S_{r(t)}$ with $\phi_t(\pm \infty) = \pm \infty$ and
$\phi_t(U_t) = U_t . $  Although $r(t)$ depends on
the curve $\gamma$, the next lemma shows that its derivative
at $0$ is independent of $\gamma$ assuming $\gamma$ has
the capacity parametrization.

\begin{lemma} \label{estimateone}
If $\gamma$ is a curve with $\hcap[\gamma_t] =
a t $, then
$             \dot r(0) = -  a/ 2 = -1/\kappa .$
\end{lemma}

\begin{proof} We will consider excursion measure
  defined
by
\[      \exc_{D}(V_1,V_2) = \frac 1{2 \pi^2} \int_{V_1}\int_{V_2}
H_D(z,w) |dz|\, |dw|.\]
This definition assumes $V_1,V_2$ are nice boundaries, but
this is a conformal invariant (see \cite[Chapter 5]{LBook}) and
hence is defined for rough boundaries as well.   In this
normalization, $\exc_r:= \exc_{A_r}(C_0,C_r) = 1/r$. 
Consider $D_t = A_r \setminus \eta_t$ where $\eta = \psi \circ
\gamma$.  We only need to consider small $t$ for which
$\eta$ is a simple curve in $A_r$.
 Let $\exc(t) = \exc_{D_t}(C_r,C_0 \cup \eta_t)$.  By
definition of $r(t)$ and conformal invariance of excursion 
measure, 
$            \exc(t) = 1/r(t).$
Therefore, by the chain rule
\begin{equation}  \label{oct25.1}             
 \dot \exc(0) = \frac{\dot r(0)}{r^2}.
\end{equation}

Suppose $r>1$ and $t$ is sufficiently
small so that $\Disk_1 \subset D_t$. 
Then using the strong Markov property,
\[ \exc_{A_r}(C_r,C_1) -\exc_{D}(C_r,D_t) =  
\exc_{\Disk_1 \cap A_r}
 (C_r, C_1) \; \E\left[q(B_{\tau_t})\right] 
  =\frac{1}{ r-1} \, \E\left[-\frac{\log |B_{\tau_t}|}r\right].\]
Here $B$ is a Brownian motion started uniformly
on $C_s$, $\tau_t$ is the first time that it leaves $D_t$
and $q(z)$ denotes the probability that a Brownian
motion starting at $z$ hits $C_r$ before $C_0$,
\[    q(z) = \frac{-\log |z|}{r} .\]
Therefore,
\[   \dot \exc(0) = \frac1{r^2} \, \frac{r}{r-1}
  \, \p_t \E[\log |B_{\rho_t \wedge \sigma_r}|]\mid_{t=0}. \]
where $\rho_t$ is the first time to leave $D_t$ and $\sigma_r$
is the first time to hit $C_r$. We claim that
 \begin{equation}  \label{dec2.1}
  \p_t\E[\log |B_{\rho_t \wedge \sigma_r}|]\mid_{t=0}
   = \frac{r-1}r \, \p_t\E[\log |B_{\rho_t}|]
 \mid_{t=0} .
 \end{equation}
 To see this, we first note that the probability starting
 at $C_1$ of hitting $C_r$ before $C_0$ is $1/r$.  Also,
 given $\rho_t < \sigma_r$, the probability of hitting
 $C_r$ before $C_0$ is $O(d_t/r)$ where $d_t =
 \diam(\gamma_t) = o(1)$.  Also, since we start
 with the uniform distribution on $C_1$, the distribution
 of $\sigma_r$ given that $\sigma_r < \sigma_0$ is
 also uniform.  Therefore,
 \[  \E[\log |B_{\rho_t}|\; ;\;\sigma_r < \rho_t]
  = \frac 1r \,   \E[\log |B_{\rho_t}|] \,[1+O(d_t)].\]
and hence
\[   \p_t  \E[\log |B_{\rho_t}|\; ;\;\sigma_r < \rho_t]
  |_{t=0} = \frac 1r   \, \p_t  \E[\log |B_{\rho_t}|]
   |_{t=0}.\]
from which \eqref{dec2.1} follows.  Note that the
right-hand side of \eqref{dec2.1} is the same if we
start the Brownian motion at the origin.

By comparison with \eqref{oct25.1}, we see that $\dot r(0)$
is independent of $r(0)$, and we can compute $\dot r(0)$ by
letting  $r \downarrow 0$.  In this case, we get the 
comparison of the chordal Loewner equation to the radial
Loewner equation. 
\end{proof}

We define
\[   \sigma_s = \inf\{t: r(t) = s\}. \]
Let  $\gamma^*$ be $\gamma$ with the ``annulus parametrization''
\[         \gamma^*(s) = \gamma(\sigma_{s}), \;\;\;\;
    0 \leq s \leq r , \]
and let 
\[        U^*_s = U_{r(s)}, \;\;\;\;
     h^*_s = h_{\sigma_s}.\]
The direction of ``time'' is reversed 
so one must be careful with minus signs.
 
\labove \textsf
{\begin{small}  \Heuristic  
 In the annulus parametrization, the radius takes the place
 of time.  However, the direction of ``time'' is reversed,
so one must take some care with minus signs.
\end{small}}
\lbelow

We will just state the annulus Loewner equation (see, e.g.,
\cite{BauerF,Komatu}).  It can also be described
  in terms of excursion
reflection Brownian motion (this helps motivate the formulas),
see \cite{Dren,LLap}.  
We review the facts here.
Let $\cpois_{S_r}(z,x) = \cpois_{S_r}(z-x)$ 
denote the complexification of the Poisson
kernel in $S_r$ which recall by \eqref{dec1.3}
is given by
\[ \cpois_{S_r}(z) =  - \frac{\pi}{2r} \,
                   \coth \left(\frac{\pi z}{2r} \right),
\]
 and satisfies
 \[  \Im \cpois(z) = H_{S_r}(z,0),\]
\[               \cpois_{S_r}(z) =  - \frac 1 z+ O(|z|), \;\;\;\; z
 \rightarrow 0 , \]
and if $x \in \R$, 
 \[ \Re \cpois_{S_r}(x) =
  - \frac{\pi}{2r} \,
                   \coth \left(\frac{\pi x}{2r} \right),\;\;\;\;
   \Re \cpois_{S_r}(x + ir) =
  - \frac{\pi}{2r} \,
                   \tanh \left(\frac{\pi x}{2r} \right).            
 \]

  There exists a unique
  holomorphic function with period $2\pi$
  \[                     \cpois_r: S_r \rightarrow \Half_r , \]
  such that
  \[                    \cpois_r(z) =  - \frac 1 z  + o(1) , \;\;\;\;
    z \rightarrow 0, \]
  and  such that the induced map
  \[       \bar  \cpois_r(e^{iz}) = \cpois_r(z) \]
   is a conformal transformation of $A_r$ onto a domain of the
   form $\Half \setminus L$ for some horizontal line segment
   $L$.
  One can find
   this using excursion reflected Brownian motion (ERBM) as
  we now sketch. The imaginary part 
  $H_r = \Im \cpois_r$   will
  be the Poisson kernel for ERBM in the annulus.
   We can write
\begin{equation}  \label{nov17.1}
      H_r(z)  =  \frac{\Im(z)}{2r} +  H_{A_r}(e^{iz},1)   
    =  \frac{\Im(z)}{2r}
     - \frac {\pi}{2r} 
   \sum_{k \in Z} \Im \coth \left(\frac{\pi z}{2r} \right).
  \end{equation}
In this formula, the infinite sum represents the contribution
to the ERBM Poisson kernel by paths that
do not hit the ``hole'' $\Disk \setminus A_r$.  The first
term gives the contribution of paths that hit the hole first.
The probability of hitting the hole before hitting
$C_0$ is $\Im(z)/r$.  Given that it hits the hole, the
distribution of the first visit to $C_0$ is uniform
on the circle and hence the value of the kernel   
is $1/2$ (recall that in
our normalization, $H_\Disk(0,1) = 1/2$.)

One can check that the sum in \eqref{nov17.1} absolutely convergent.
However, the real parts are not absolutely convergent so we must
take a little care in the definition of $\cpois_r$.
We write
\begin{eqnarray*} 
  \cpois_r(z) &   =  & 
    \frac z{2r} - \frac \pi{2r} \, \coth\left(\frac{z\, \pi}{2r}\right) 
                        \\
                        &  & \hspace{.5in}
                          - \frac \pi {2r} \sum_{k=1}^\infty 
                       \left[ \coth\left(\frac{(z + 2 k \pi)\pi}{2r}\right) + \coth
                       \left(\frac{(z-2k\pi) \pi}{2r}\right)\right]\\
                       & = & 
 \frac{z}{2r}
 - \frac \pi{2r}\ppsum_k \coth\left(\frac{(z + 2 k \pi)\pi}{2r}\right),
 \end{eqnarray*}
 where we write 
 \[   \ppsum_k f(k) = \lim_{N \rightarrow \infty}
   \sum_{k=-N}^N f(k) . \]

 \begin{lemma}  As $z \rightarrow 0$,
 \begin{equation}  \label{expand}
   \cpois_r(z) = -\frac 1z + z \left(\frac 1{2r} -
   \Gamma(r)
           + \frac 1{12}
         \right) +O(|z|^3), 
         \end{equation}
%
         where $\Gamma(r)$ is as defined in \eqref{Gammadef}.
\end{lemma}

\begin{proof}
 We
  use the first expression for the definition of $\cpois_r$. 
   Note that as $z \rightarrow 0$, 
 \[  \coth z =
 \frac 1z  + \frac{z}{3}
                     + O(|z|^3), \]
and hence
\[  \frac \pi {2r}\, \coth\left(\frac{z\, \pi}{2r}\right) 
= \frac \pi {2r} \left[ \frac{2r}{z \pi} + \frac{z\pi}{6r}
  + O(|z|^3)\right] = \frac 1 z + \frac{\pi^2z}{12 r^2} +O(|z|^3) \]
 Also the derivative at $z=0$ of 
\[ -\frac \pi {2r} \sum_{k=1}^\infty 
                       \left[ \coth\left(\frac{(z + 2 k \pi)\pi}{2r}\right) + \coth
                       \left(\frac{(z-2k\pi) \pi}{2r}\right)\right]\]
is 
$\frac 1{12}
 - \delta(r) . $

\end{proof}

%
%
%
%
%
%
%
%

 Note that
 \[   \cpois_r(z+ir) = \frac{z+ir}{2r} - 
  \frac \pi {2r}  \ppsum_k
                      \tanh\left(\frac{(z + 2 k \pi)\pi}{2r}\right) 
                      = - \frac{\zhan_I(r,x)} 2 + \frac i2 , \]
where $\zhanh_I$ is as defined in 
Section \ref{annfunction}.


The chordal equation \eqref{chordal} can be written as
\[         \p_t g_t(z) = - a\, \cpois_{\Half}(g_t(z) - U_t).\]
The annulus Loewner equation is similar,
\[         \p_t h_t(z) =  2 \, \dot r(t) \, \cpois_{r(t)}
     (h_t(z)-U_t) ,\]
or equivalently,
 \begin{equation}  \label{anneq}
          \p_r h_r^*(z) =  2\, \cpois_{r}
     (h_r^*(z)-U_r^*) . 
\end{equation}
An important observation is that if $r(0) = r$,
then 
for small $t$, the
functions $g_t$, $h_t$, and $h_{r- \frac{at}2}^*$
 are very close near the origin.  For future
 reference, we also note that
\begin{equation}  \label{future}
   \p_s \log (h_s^*)'(x+ir) \mid_{s=r}   =  
    2 \, \cpois_r'(x+ir) =    - \zhan_I'(r,x).
    \end{equation}

 \labove \textsf
{\begin{small}  \Heuristic  
  There may appear to be some arbitrariness in the choice of
  the real translation for the complex kernel 
  $\cpois_{\Half}(g_t(z) - U_t).$  It turns out that this choice
  is not so important.  We will write
  \[   d [h_r^*(z)- U_r^*] =  2\, \cpois_{r}
     (h_r^*(z)-U_r^*) - d U_r^* .\]
 If we had chosen a different
 real translation of $\cpois_{r}$, it would
 cancel here when we took the difference.
\end{small}}
\lbelow

\labove \textsf
{\begin{small}  \Heuristic  
  We have written the annulus equation
in the covering space $S_r$.  We would also consider the
function given by
\[           f_s(e^{iw}) = e^{i h_s(w)}, \;\;\;\;  0 \leq s \leq r. \]
There is a curve $\eta:(0,r) \rightarrow A_r$ with $\eta(0+) = 1$
such that $f_s$ is a conformal transformation of $A_r \setminus
\gamma_s$ onto $A_{r-s}$.  Such a transformation is defined
up to a rotation, but specifying continuity and
 $f_s(\eta(r-s))= U_s^*$ determines the rotation.
\end{small}}
\lbelow

We will need to compare the chordal and annulus 
equations at time $t=0$.  Recall that
$\phi_t$ is defined by
\[            h_t(z) = \phi_t(g_t(z)), \]
and that $\phi_t(U_t) = U_t = g_t(\gamma(t))$.
Although $g_t$ is not smooth at $\gamma(t)$, it
is not difficult to show that $\phi_t$ is analytic
in a neighborhood of $U_t$ and we can give
the derivatives.  We summarize the facts we need
in this lemma whose simple prove we omit.

\begin{lemma}  Suppose $K_{j,t}(z), j=1,2,
t \in [0,\epsilon]$ are analytic
functions in a punctured neighborhood of the origin
and are continuous in $t$.  Suppose $U_t$ is
a continuous function with $U_0=0$ and $g_t,h_t$
satisfy 
\[  \p_t g_t(z) = K_{1,t}(g_t(z) - U_t), \;\;\;\;
   \p_t h_t(z) =  K_{2,t}(h_t(z) - U_t) , \]
with $g_0(z) = h_0(z)$.  Suppose that for all $t$,
$K_{1,t}-K_{2,t}$ is analytic in the (unpunctured)
neighborhood.  If $\phi_t$ is defined by
$h_t(z) = \phi_t(g_t(z))$, then
\begin{equation}  \label{nov18.1}
\dot \phi_0(z) = [K_{2,0} - K_{1,0}](z),
\;\;\;\;   \dot \phi_0'(z) = [K_{2,0}- K_{1,0}]'(z). 
\end{equation}
\end{lemma}

We now return to the locally chordal $SLE_\kappa$ from $0$ to $z_0 = x + ir$
in $S_r$.  Given the path $\gamma_t$, 
 the process is moving infinitesimally like $SLE_\kappa$ in
$\hat S_{r,t}$ from $\gamma(t)$ to $z_0$.   By conformal invariance
we can also view it in $g_t(\hat S_{r,t})$ from $U_t$ to $g_t(z_0)$
or in $h_t(\hat S_{r,t}) = S_{r(t)}$ from $U_t$ to $h_t(z_0)$. 
Using the last perspective and \eqref{nov7.1} and \eqref{fun3},
we see that
\[           dU_t = b \,\funthree(r(t),R_t)\, dt - dW_t , \]
where $R_t = \Re[h_t(z_0)] - U_t$ and
$W_t$ is a standard Brownian motion.  We choose a time
parametrization so that the radius evolves linearly.  If
$U_t^* = U_{\sigma(t)}$ as above,
\[          dU_t^* =  b\kappa 
\,\funthree(r-t,R_t^*)\, dt - \sqrt \kappa \,dB_t . \] 
Using \eqref{future}, we see that if $f_t = h_{r-t}^*$, 
   \[    \p_t [\Re f_t(z_0)] = \zhan_I (r-t,R_t^*)
, \]
and hence
\begin{equation}  \label{feb1.1}
      dR_t^* =  \left[\zhan_I(r-t,R_t^*)  - b\kappa 
\,\funthree(r-t,R_t^*)\right]
  \, dt + \sqrt \kappa \, dB_t. \end{equation}

  We have written locally chordal $SLE_\kappa$ in the annulus
  as a one-dimensional SDE stopped at a finite time $r$.  The
  next lemma shows that the process leaves $S_r$ at $z_0$.  The
  equivalent statement is the following.
  
   \begin{lemma}  \label{locchord}
If $X_t$ satisfies 
\[  dX_t =  \left[\zhan_I(r-t,X_t)  - b\kappa 
\,\funthree(r-t,X_t)\right]
  \, dt + \sqrt \kappa \, dB_t, \;\;\;\; 0 \leq t < r, \]
     then with probability
one $X_{r-} = 0$. 
\end{lemma}

\labove \textsf
{\begin{small}  \Heuristic 
 This lemma should not be surprising.  If we considered chordal $SLE_\kappa$
 from $0$ to $x+ir$ in $S_r$ we know that (for $\kappa \leq 4$) the path
 leaves the domain at $x+ir$.  This lemma stays that the same thing
 for locally chordal $SLE_\kappa$.  Since for $r$ near zero, locally chordal and
 chordal $SLE_\kappa$ are almost the same, the lemma has to be true.
 One should expect $\kappa \leq 4$ to come into the proof, and this
 is the case.
\end{small}}
\lbelow

\begin{proof}
We discuss the most delicate case,
 $\kappa = 4$ for which
$b \kappa = 1$; if 
$\kappa < 4$, then $b \kappa > 1$ and the argument
is easier.  Our equation is 
\[    dX_t = \left[\zhan_I( r-t,X_t) -   \funthree( r-t,
   X_t) \right] \, dt + 2 \, dB_{t}.\]
 If $Y_s = X_{r-e^{-s}}$, then
 $Y_s$ satisfies
 \[   dY_s = m(s,Y_s) \, ds + 2 \, e^{-s/2} \, dW_s, \]
 where
 \[  m(s,y) = e^{-s} \,  
   \left[\zhan_I( e^{-s},y) -   \funthree( e^{-s},
   y) \right], \]
   and $W_s$ is a standard Brownian motion.   It suffices
   to show that for every $\epsilon > 0$, with probability one,
   $|Y_s| \leq \epsilon$ for all $s$ sufficiently large.  
   By symmetry it suffices to show that that $\limsup Y_s
    \leq 0$. Let
  \[     Z_s = \int_0^s 2\, e^{-r/2} \, dW_r, \]
  and note that with probability one $Z_\infty$
  exists and is finite.
    
  Using Lemma \ref{dec2.lemma5}, we can see that
  there exists $s_\epsilon$ such that 
  $m(s,y) \leq 0$ for $s  \geq s_\epsilon, y \geq
  \epsilon/2$.   Therefore,  if $Y_s \geq \epsilon$
and  $s \geq s_\epsilon$, 
   \[    Y_r \leq \epsilon + \max_{t \geq s_\epsilon}
                |Z_t - Z_{s_\epsilon}|. \]
 Therefore, it suffices to show that with probability
 one $\liminf Y_n \leq 0$.  In other words, for every
 $\epsilon > 0, s < \infty, y >0$, the probability
 that the process reaches $\epsilon$ given $Y_s = y$
 equals one.
 
 Although the drift $m(s,y)$ is negative, the absolute
 value is very small at $y$ slightly larger than an
 integer multiple of $2\pi$.  However, we also know
 from Lemma \ref{dec2.lemma5} that for all $y$,
 $m(s,y) \leq - c \, e^{-s}$.  Given this, we can
 see that if we start near $2 \pi k$, there is
 at least a positive probability that there will
 exist $s$ with 
$Y_s \leq 2 \pi k - c_1 \, e^{-s}$.  Given this, there
is a positive probability that the process will never
return to $\{y \geq 2 \pi k - (c_1/2) \, e^{-s}\}$
and since the drift is negative, this will imply that
it will get near $2 \pi (k-1)$.  This happens with
positive probability, but if it fails and we are near
$2 \pi k$ at a larger time $s'$ we can find $s'' >
s'$ for which $Y_{s''} \leq 2 \pi k - c_1 \, e^{-s''}$.
Eventually we will succeed and get to $2 \pi (k-1)$.
We can iterate this argument.
   
 \end{proof}

\section{Definition of $\mu_D(z,w)$}  \label{defsec}

\subsection{Definition of boundary
$SLE_\kappa$ for $\kappa \le 4$}  \label{boundarysec}

 We fix $\kappa  \in (0,4]$.  In this section, we will define
 boundary
$SLE_\kappa$ 
 as proposed
in \cite{LJSP}.  It is a  (positive) measure $\mu_D(z,w)$ on
simple curves $\gamma$ in a domain $D$ connecting
distinct $\p D$-analytic
 boundary points $z$ and $w$.  
If $D$ is simply connected, then
 the definition is the same as that of chordal $SLE_\kappa$.
   We write
\[  \tmass_D(z,w) = \|\mu_D(z,w)\| \]
for the total mass of the measure.
We {\em conjecture}  that $\tmass_D(z,w) < \infty$ for all
$D,z,w$.
In the case of simply connected domains, we know this is true,
and in this paper we will show it for $1$-connected domains for
$\kappa \leq 4$.   
From the construction it will follow
that $\tmass_D(z,w) < \infty$   for all  domains if $\kappa
\leq 8/3$ ($\cent \leq 0$).

Suppose $D_1 \subset D$ is a subdomain of
$D$ that agrees with $D$ in neighborhoods of $z$ and $w$.  
We let $\mu_D(z,w;D_1)$ be $\mu_D$ {\em restricted} to curves
$\gamma \subset D_1$.  Let
\[      \tmass_D(z,w;D_1) = \|\mu_D(z,w;D_1)\| .\] 
We will show that $\mu_D(z,w;D_1) < \infty$ for all
such simply connected $D_1$ for $\kappa \leq 4$. 
The measure $\mu_D^\#(z,w;D_1)$ is defined to be
the probability measure obtained by normalization
\[         \mu_D^\#(z,w;D_1) =
  \frac{\mu_D(z,w;D_1)}{\tmass_D(z,w;D_1)}. \]
If $\tmass_D(z,w) < \infty$, we write $\mu^\#(z,w)$ for
the probability measure.

\labove \textsf
{\begin{small}  \Heuristic  
What we call boundary $SLE$ should really be called
boundary/boundary $SLE$, but this terminology is
a bit cumbersome. 
In later subsections, we also discuss
boundary/bulk, bulk/boundary, and bulk/bulk cases.
\end{small}}
\lbelow


%

%
%
%

In this definition and later on we use the convention
as described below equation \eqref{poisscale} that if
formulas are written with derivatives, then 
sufficient smoothness is assumed.

\begin{definition}  If $\kappa \leq 4$ and $b,\cent$
are as in \eqref{bandc}, \textit{boundary  $SLE_\kappa$} is 
 the unique
family of measures (modulo reparametrization)
 $\{\mu_D(z,w)\}$, where $D \subset \C$
and $z,w$ are distinct $\p D$-analytic points,  satisfying
the following.

\begin{itemize}

 \item   For each $D,z,w$, $\mu_D(z,w)$ is a positive measure
on curves   $\gamma:[0,t_\gamma]
\rightarrow D$ with $\gamma(0) = z, \gamma(t_\gamma) = w,
\gamma \subset D$.  The total mass is denoted by
\[   \tmass_D(z,w) = \|\mu_D(z,w)\|. \]
The normalization is chosen
so that  $\tmass_\Half(0,1) = 1$.

\item {\bf Conformal covariance}
If 
 $f:D \rightarrow f(D)$
is a conformal transformation,  then
\begin{equation} \label{cc}
   f \circ \mu_D(z,w) =  |f'(z)|^b \, |f'(w)|^b \, \mu_{f(D)}(z,w) .
   \end{equation}

\item 
   It follows from \eqref{cc} that
 the  probability measures are conformally
 invariant,
 \[  f \circ \mu_D^\#(z,w;D_1) = \mu_{f(D)}^\#(
 f(z),f(w); f(D_1)), \]
 and if $\tmass_D(z,w) <\infty$, 
 \begin{equation} \label{ci}
   f \circ \mu_D^\#(z,w) = \mu_{f(D)}^\#(z,w) .
   \end{equation}
In particular, $\mu_D^\#(z,w;D_1)$ (resp.,
$\mu_D^\#(z,w))$  can be defined for nonanalytic boundary
points provided that there is a conformal transformation $f:D 
\rightarrow f(D)$ such that $f(z),f(w)$ are $\p f(D)$-analytic
(resp., with $\tmass_{f(D)}(f(z),f(w)) < \infty$).
 
\item {\bf Domain Markov property}.  
If $\tmass_D(z,w) < \infty$,
then for  the probability measure
$\mu_D^\#(z,w)$, the conditional  probability measure of the remainder
of a curve $\gamma$
given an initial segment $\gamma_t$, is that of $\mu_{D \setminus \gamma_t}^\#
(\gamma(t),w)$.   If $D_1 \subset D$ is simply connected, 
for the probability measure
$\mu_D^\#(z,w;D_1)$, the conditional  probability measure of the remainder
of a curve $\gamma$
given an initial segment $\gamma_t$, is that of $\mu_{D \setminus \gamma_t}^\#
(\gamma(t),w;D_1 \setminus \gamma_t)$.

\item  {\bf Boundary perturbation}.  Suppose $D' \subset D$
are  domains that agree in neighborhoods of $\p D'$-analytic boundary
points $z,w$.  Then $\mu_{D'}(z,w)$ is absolutely continuous
with respect to $\mu_D(z,w)$ with Radon-Nikodym derivative
$Y = Y_{D,D',z,w}$ given by
\begin{equation} \label{perturb}
  Y(\gamma) = \frac{d\mu_{D'}(z,w)}{d\mu_D(z,w)} \, (\gamma)
  = 1\{\gamma \subset D'\} \, \exp\left\{\frac \cent 2\,
   m_D(\gamma,D \setminus D')\right\} . 
\end{equation}
\end{itemize}

\end{definition}

We will now construct the measure and in the process show
uniqueness.   For simply connected domains, we set
$\tmass_D(z,w) = H_{\p D}(z,w)^b$ and $\mu_D^\#(z,w)$
to be the conformal image of $\tmass_\Half(0,\infty)$
under a conformal transformation.  The discussion
in Section \ref{simplesec} shows that this is the
unique  family of measures that satisfy the conditions
above for simply connected $D$.

\begin{definition}  Suppose $D$ is a domain and $z,w$ are distinct
$\p D$-analytic boundary points.  Let $D_1$ be a simply connected
subdomain of $D$ that agrees with $D$ in neighborhoods of $z,w$.
Then  $\hat \mu_D(z,w;D_1)$  is  the measure absolutely continuous
with respect to $\mu_{D_1}(z,w)$ with Radon-Nikodym derivative
\begin{equation}  \label{oct26.1}
   \frac{d \hat \mu_D(z,w;D_1)}{d\mu_{D_1}(z,w)} (\gamma)
    = 1\{\gamma \subset D_1\} \, \exp \left\{- \frac \cent 2
     \, m_D(\gamma,D \setminus D_1) \right\}. 
     \end{equation}
\end{definition}

\labove \textsf
{\begin{small}  \Heuristic  
  A minus sign appears on the right-hand side above.  This
  is because we are writing the derivative of the measure
  on  the larger
  domain with respect to that on the smaller domain.
\end{small}}
\lbelow

The next proposition establishes a necessary consistency condition
for the measures $\hat \mu_D(z,w;D_j)$ in order to define
$\mu_D(z,w)$.

\begin{proposition}  \label{consistent}
 Suppose $D$ is a domain and $z,w$ are distinct
$\p D$-analytic boundary points.  Let $D_1,D_2$ be  simply connected
subdomains of $D$ that agree with $D$ in neighborhoods of $z,w$.
For $j = 1,2$, let $\nu_j$ be $\hat \mu_D(z,w;D_j)$ restricted to curves
$\gamma$ with $\gamma \subset D_1 \cap D_2$.  Then $\nu_1 = \nu_2$.
\end{proposition}

\begin{proof}  Suppose $\gamma \subset D_1 \cap D_2$.  Then there
exists simply connected $\hat D \subset D_1 \cap D_2$ that agrees
locally with $D$ near $z,w$ such that $\gamma \subset \hat D$.  Hence
it suffices to show that for every simply connected domain $\hat D$,
$\nu_1$ and $\nu_2$, restricted to curves in $\hat D$, agree.
Suppose $\gamma \subset \hat D$.
Since $D_j, \hat D$ are simply connected,
\[   \frac{d\mu_{D_{j}}(z,w)}{d\mu_{\hat D}
   (z,w)} (\gamma)  =  \exp \left\{- \frac \cent 2
     \, m_{D_j}(\gamma,D_j \setminus \hat D) \right\}
.\]
Combining this with \eqref{oct26.1}, we get
 \[  \frac{d \hat \mu_D(z,w;D_j)}{d\mu_{\hat D}(z,w)} (\gamma)
    =  \exp \left\{- \frac \cent 2
     \, m_D(\gamma,D \setminus \hat D) \right\}. \]
 Here we use the fact that the loops in $D$ that intersect $\gamma$
 and $D \setminus \hat D$ can be partitioned into two sets:
 those that intersect $D \setminus D_1$ and those that are contained
 in $D_1$. 
\end{proof}

Given Proposition \ref{consistent} we can make the following
definition.

\begin{definition}
 Suppose $D$ is a domain and $z,w$ are distinct
$\p D$-analytic boundary points.  Then $\mu_D(z,w)$
is the measure on simple paths (modulo parametrization)
such that for each  simply connected $D_1 \subset 
 D$, $\mu_D(z,w)$ restricted to curves $\gamma \subset D_1$
 is $\hat \mu_D(z,w;D_1)$.  
  \end{definition}

 In other words, $\mu_D(z,w;D_1) = \hat \mu_D(z,w;D_1)$
 for simply connected $D_1$.  It follows immediately from
 the definition that the family of measures
 $\{\mu_D(z,w)\}$ satisfies \eqref{perturb}.
 Suppose $D$ is a domain
 and $z,w$ are distinct $\p D$-analytic points and $D_1$
 is a simply connected domain as above.
 Suppose  $f:D \rightarrow f(D)$ is a conformal
 transformation.  Then $f:D_1 \rightarrow f(D_1)$ is
 also a conformal transformation, and hence
 \[   f \circ \mu_{D_1}(z,w) = |f'(z)|^b \, |f'(w)|^b
  \, \mu_{f(D_1)}(f(z),f(w)). \]
  Conformal invariance of the loop
 measure then  implies that 
 \[  f \circ \mu_D(z,w;D_1) = |f'(z)|^b
  \, |f'(w)|^b \, \mu_{f(D)}(z,w; f(D_1)). \]
  Since this is true for every simply connected $D_1$,
  the family $\{\mu_D(z,w)\}$ satisfies \eqref{cc}.
  
  In this paper, we will show the following. (While
  we prove it in this paper, we could also derive
  this from \cite{Zhanannulus}.)
  
\begin{proposition}  If $D$ is a conformal annulus,
then $\tmass_D(z,w) < \infty$ and the family
$\{\mu_D(z,w)\}$ restricted to conformal annuli
satisfies the domain Markov property.
\end{proposition}

When considering the measure $\mu_D(z,w)$ for
multiply connected domains, there are
two cases.
\begin{itemize}
\item  The {\em chordal} case: $z,w$ in the same component
of $\p D$.  Then there exists simply connected $\hat D$ such
that $D \subset \hat D$.
\item  The {\em crossing} case: $z,w$ in different components
of $\p D$.  Then there exists $1$-connected $\hat D$
such that $D \subset \hat D$.
 \end{itemize}

 \begin{proposition} $\;$  Suppose $D$ is a domain
 and $z,w$ are distinct $\p D$-analytic points.
 
 \begin{itemize}
 
 \item  If $\kappa \leq 8/3$, then $\tmass_D(z,w)
   < \infty$.

  \item  If $8/3 < \kappa \leq 4$, then for every
  simply connected $D_1 \subset D$ that agrees
  with $D$ near $z,w$, $\tmass_D(z,w;D_1) < \infty$.
  
\end{itemize}
\end{proposition}

\begin{proof}   If $\kappa \leq 8/3$,
we can  consider $D$ as a subdomain of a simply
connected or $1$-connected domain $\hat D$ and
since $\cent \leq 0$, 
\eqref{perturb} implies that $\tmass_D(z,w)
\leq \tmass_{\hat D}(z,w)<\infty$.  If $8/3 < \kappa \leq 4$,
then $\cent >0$, and \eqref{oct26.1} implies 
that $\tmass_D(z,w;D_1) \leq \tmass_{D_1}(z,w) < \infty$.
\end{proof}
 
 \begin{proposition}  The family $\{\mu_D(z,w)\}$
 satisfies the domain Markov property.
 \end{proposition}

   \begin{proof}  Without loss of generality we may
   assume that $D$ is a subdomain of $\Half$ whose boundary
   includes $\R$ and $z=0$.  Let $D_1$ be a simply connected
   domain as above for which we know $\tmass_D(z,w;D_1) < \infty$
    and let $\gamma_t$ be an initial segment.
   To be more precise, let $t$ be a finite stopping
   time for chordal $SLE_\kappa$ in $D_1$. 
   Let $ \F_t$ be the corresponding $\sigma$-algebra
   generated by $\gamma_t$. 
For $\gamma \subset D_1$, let 
   \[  Y(\gamma) = \frac{\mu_D(z,w;D_1)}{\mu_{D_1}
   (z,w)} (\gamma) =  \exp \left\{\frac  \cent 2 \, m_D(
   \gamma,D \setminus D_1)\right\}. \] 
Let $\Prob,\E$ denote probability and expectation with
respect to the  probability
measure $\mu_{D_1}^\#(z,w)$.  Then,
\[   \tmass_D(z,w;D_1) = \tmass_{D_1}(z,w) \, 
             \E\left[Y\right]. \]
 By the domain Markov property for $SLE_\kappa$ is 
 simply connected domains,
 \[     \E[Y \mid \F_t] =  \exp \left\{\frac  \cent 2 \, m_D(
   \gamma_t,D \setminus D_1)\right\}\, \E_t^*[Y], \]
   where $\E_t^*$ denotes expectation with respect
   to  $\mu_{D_1 \setminus \gamma_t}^\#
   (\gamma(t),w)$.

    We will do the chordal case comparing
   to simple connected domains.  The crossing case is similar
   using conformal annuli.  Suppose $z,w$
   are in the same component of $\p D$.  Without loss of generality,
   we may assume that $D$ is a subdomain of $\Half$ and $z,w
   \in \R$.  We know that
   \[   \frac{d \mu_D(z,w)}{d \mu_\Half(z,w)} (\gamma)
   = 1\{\gamma \in D\}\, \exp\left\{\frac \cent 2
    \, m_\Half(\gamma,\Half \setminus D)\right\}. \]
    Let $\Prob,\E$ denote probabilities and expectations 
with respect to
    the measure $\mu^\#_\Half(z,w)$.  Let
    \[     Y_t = 1\{\gamma_t \subset D\} \, \exp \left\{\frac \cent 2
      \, m_\Half(\gamma_t,\Half \setminus D)\right\}, \;\;\;\;
         Y = Y_\infty.\]
    Suppose we are given an initial segment $\gamma_t$ and let
    $H_t = \Half \setminus \gamma_t$.   Here $t$
    can be a stopping time and we assume that $t < T=\inf\{s> 0:
    \gamma(s)  \in \R\}= \inf\{s >0: \gamma(s) = w\}$.
    (The equality is true with $\Prob$ probability one.)  Let $g_t$ denote the corresponding
    map and let $\F = \F_t$ denote the $\sigma$-algebra generated
    by $t$.  By the domain Markov property of
    $SLE_\kappa$ in simply connected domains,
    \[       \E\left[Y \mid \F\right]
      =     Y_t  \, \E^*_t
                 \left[\exp \left \{
                 \frac \cent 2 m_{H_t}(\gamma, H_t \setminus D) 
                 \right\} \right], \]
   where $\E^*_t$ denotes expectations with respect to $\mu_{H_t}^\#(\gamma(t),
   w)$.  More generally if $E$ is an event depending on the path $\gamma \setminus
   \gamma_t$,
    \[       \E\left[Y\, 1_E \mid \F\right]
      =     Y_t  \, \E^*_t
                 \left[1_E \, \exp \left \{
                 \frac \cent 2 m_{H_t}(\gamma, H_t \setminus D) 
                 \right\} \right], \]
                 
  If $\tmass_D(z,w) < \infty$, the proof for $\mu_D^\#(z,w)$ is similar
  and we omit it.
   
   \end{proof}

\begin{proposition}
If $z,w$ are $\p D$-analytic, then $\mu_D(w,z)$
is the same as the reversal of $\mu_D(z,w)$.
\end{proposition}

\begin{proof}
In the case of simply connected domains, this was
proved by Zhan \cite{Zhanrev}.  Given this, the general
case follows.
\end{proof}
   
%
%
%
%
%

We end this section with a number of remarks.

\begin{itemize}

\item 
In our definition we have started with the parameter $\kappa$
and defined the quantities $b, \cent$ in terms of $\kappa$.
We could have made $b,\cent$ free parameters, but then we
would find out that there was only a one-dimensional family
of pairs $(b,\cent)$ for which we could define such measures.
To establish this fact, we would use Schramm's argument and
$\kappa$ (as a function of $b$ or $\cent$) would be introduced.

\item
 Implicit in the domain Markov property is the
 assumption that the 
the initial segment may be chosen using a stopping time.  This
makes it a condition on curves modulo reparametrization.  Perhaps
this should be called the {\em strong} domain Markov property.

\item  It is also useful to have the measures
$\mu_D(x,\infty)$ where $D \subset \Half$ with 
$\Half \setminus D$ bounded and $\dist(x,\Half \setminus D)
> 0$.  To get this we find a conformal transformation
\[     f: D' \rightarrow D \]
with $f(z) = 0, f(w) = \infty$ and
use the conventions about derivatives as in Section \ref{poissec}.
%
Under this convention, we see that $\tmass_\Half(0,\infty) = 1$.
If $D \subset \Half$ is simply connected with $\Half \setminus
D$ bounded and $\dist(0, \Half \setminus D) > 0$, then
$\tmass_D(0,\infty) = \Phi_D'(0)^b$ where $\Phi_D:D \rightarrow
\Half$ is a conformal transformation with $\Phi_D(\infty) =
\infty, \Phi_D'(\infty) = 1$.  
 
\end{itemize}

\subsection{Definition of boundary/bulk and bulk/bulk $SLE_\kappa$
for $\kappa \leq 4$}

The boundary $SLE_\kappa$ is a measure on curves connecting two
boundary points in a domain $D$.  We extend this definition
to allow one boundary point and one interior point (the radial 
or reverse radial case)
or two interior points (the bulk case).  In all the cases we will
write  $\mu_D(z,w)$ for the measure, $\tmass_D(z,w)$ for the total
mass, and if $\tmass_D(z,w) < \infty$ $\mu_D^\#(z,w)$
for the corresponding probability measure.   
The definition will be the same as the first definition 
in Section \ref{boundarysec} except that \eqref{cc} is
replaced with the following more general formula.
Note that this definition subsumes the previous one.

\begin{itemize}

\item {\bf Conformal covariance}
If $f:D \rightarrow f(D)$
is a conformal transformation, $z,w$ are  $D$-analytic,
and $f(z), f(w)$ are  $f(D)$-analytic, then
\begin{equation} \label{ccrad}
   f \circ \mu_D(z,w) =  |f'(z)|^{b_z} \, |f'(w)|^{b_w} \, \mu_{f(D)}(z,w) ,
   \end{equation}
 where $b_\zeta = b$ if $\zeta$ is a boundary point and $b_\zeta = \tilde b$
 if $\zeta$ is an interior point.  
\end{itemize}

\labove \textsf
{\begin{small}  \Heuristic  
We are writing $\mu_D(z,w)$ for all the cases in order
not to add more notation.  It is important to remember
that  the definitions of these measures
are different (although related, of course) depending on whether
$z,w$ are boundary or interior points.
\end{small}}
\lbelow

If $D$ is simply connected,
 $z$ is $\p D$-analytic and $w \in D$, then we  define
 $\mu_D(z,w)$ by
 \[      \mu_D(z,w) = \tmass_D(z,w) \, \mu_D^\#(z,w)
,\]
 where $\mu^\#_D(z,w)$ is radial $SLE_\kappa$ as in
 Section \ref{fslesec}.  The partition function $\tmass_D(z,w)$
 is determined up to a multiplicative constant by the
 rule \eqref{ccrad}, and we choose the constant so that
 $\tmass_\Disk(1,0) = 1$.  Using the relationship
 in Section \ref{fradsec}, one can check that this satisfies
 the necessary conditions.  In particular, the boundary
 perturbation rule \eqref{perturb} holds for simply
 connected domains.

It was essentially shown in \cite{Zhanannulus}, and we will
reprove it here, that radial $SLE_\kappa$ can be given
as a limit of boundadry/boundary  $SLE_\kappa$ in the annulus.  The following
theorem makes a more precise estimate.  

\begin{theorem}  There exists $c< \infty, q > 0$ such that
the following holds.  Let $t > 0$ and let $\gamma_t$
denote an initial segment of a path in $\Disk$ starting
at $1$ such that if $g: \Disk \setminus \gamma_t \rightarrow
\Disk$ is a conformal transformation with $g(0) = 0, g'(0) > 0$,
then $g'(0) = e^{t}.$  Suppose that $r \geq t + 2, 0 \leq \theta <
2\pi$, and let $\mu_1 = \mu_\Disk(1,0), \mu_2 = \mu_{A_r}^\#(1,
e^{-r+i\theta})$, both considered as probability measures on
initial segments $\gamma_t$.  Let $Y = d\mu_2/d\mu_1$.  Then
\begin{equation}  \label{nov28.1}
  |Y(\gamma_t) - 1| \leq  c\, e^{(t-r)q}. 
  \end{equation}
  Moreover, there exists $c_0 \in (0,\infty)$ such that
\begin{equation}  \label{nov28.2}
  \tmass(1,e^{-r+ix}) = c_0 \, e^{(b-\tilde b)r}
   \, r^{ \cent/2}\, [1 + O(e^{-qr})]. 
   \end{equation}
\end{theorem}

%
%

We will write
\[     \mu_{A_r}(1,e^{-r + ix}) = c_0 \, e^{(b-\tilde b)r}
   \, r^{ \cent/2} \, \mu_\Disk(1,0) \, 
      [1 + O_t(e^{-qr})], \]
      as shorthand for \eqref{nov28.1} and
      \eqref{nov28.2}.
      
    \labove \textsf
{\begin{small}  \Heuristic  
We can see the interior scaling exponent as coming from
a computation from the annulus partition function.  Suppose
$D$ is a bounded domain, $0 \in D$ and $w \in \p D$ is $D$-analytic.
Suppose that $\epsilon$ is small and $|z| = \epsilon$.  Let
$D_\epsilon$ denote the conformal annulus obtained by removing
the closed disk of radius $\epsilon$.  Then by analysis of the
annulus partition function which is a boundary/boundary
quantity, we see as $\epsilon
\rightarrow 0$,
\[         \tmass_{D_\epsilon}(1,z) \sim c \, \epsilon^{\tilde b-b}
 \, [\log (1/\epsilon)]^{\cent/2}, \]
 and hence we can define $\tmass_D(1,0)$ (up to an arbitrary multiplicative
 constant) by 
 \[     \tmass_{D}(1,0) \sim  \epsilon^{b-\tilde b } \,
   [\log (1/\epsilon)]^{-\cent/2}\, \tmass_{D_\epsilon}(1,\epsilon).\]
If $f:D\rightarrow f(D)$ is a conformal transformation with $f(0) = 0$,
then $f(D_\epsilon)$ is approximately the disk of radius $f'(0) \, \epsilon$,
and
\begin{eqnarray*}
              \tmass_{D_\epsilon}(1,z)
 & \sim & |f'(1)|^b \, |f'(z)|^b \, \tmass_{D_{\epsilon f'(0)}}(f(1),f(z))\\
   & \sim & |f'(1)|^b \, |f'(0)|^b \, \tmass_{D_{\epsilon f'(0)}}(f(1),f(z))
   \end{eqnarray*}
Therefore, if $u = |f'(0)|$,
\begin{eqnarray*}
         \tmass_{D}(1,0) &  \sim &
  \epsilon^{b-\tilde b} \,
   [\log (1/\epsilon)]^{-\cent/2}\, \tmass_{D_\epsilon}(1,z)\\
   & \sim & |f'(1)|^b \, u^{\tilde b}\,(u\epsilon)^{b-\tilde b }\,
      [\log (1/\epsilon)]^{-\cent/2}\,  \tmass_{f(D)_{u\epsilon}}(f(1),
     f(uz))
      \end{eqnarray*} 
  Note that the logarithmic term 
     which includes the central charge does not contribute to the scaling
     exponent.
\end{small}}
\lbelow

We now define  boundary/bulk  and bulk/boundary
$SLE$.
 The consistency of this definition follows from the
fact that \eqref{perturb} holds for simply
connected domains.
 
%
 
\begin{definition}  If $z \in D$ and $w$ is a
$\p D$-analytic boundary point, then 
$\mu_D(w,z)$ and $\mu_D(z,w)$ are
 defined as follows.
\begin{itemize}
\item  If $D$ is simply connected, 
\[       \mu_D(w,z)  =
  |f'(w)|^{-b} \, |f'(z)|^{-\tilde b} \,
  f\circ \mu_\Disk(1,0),\]
  where $f:\Disk \longrightarrow D$ is the
  conformal transformation with $f(1) = w,
  f(0) = z$.
  \item  If $D\subset D_1$ where $D_1$ is simply
  connected and agrees with $D$ near $z$ and $w$,
  then
  \[   \frac{d \mu_D(w,z)}
      {d \mu_{D_1}(w,z)}(\gamma)
        = 1\{\gamma \subset D\}
          \, \exp \left\{\frac \cent 2 \, 
            m_{D}(\gamma, D_1 \setminus D)
            \right\}. \]
  
\item   $\mu_D(z,w)$ is defined to be the measure obtained
  from $\mu_D(w,z)$ by reversing the paths.
 \end{itemize}
  \end{definition}


 We can define bulk/bulk $SLE_\kappa$ similarly.  There
 is technical issue if $D$ is all of $\C$.  Let us define
 $D$ to be {\em regular} if with probability one
 a Brownian motion exits the domain $D$.

 \begin{definition}  If $z,w$ are distinct points of
 a regular domain $D$, then $\mu_D(z,w)$ is defined by
 \[   \mu_D(z,w) =  c_0^{-1}
 \lim_{r \rightarrow \infty}
             e^{2(\tilde b - b)r}
                \, r^{\cent/2} \, \mu_{D_r}(z + e^{-r},
             w + e^{-r}),\]
  where 
  \[  D_r = \{\zeta \in D: |\zeta-z|
   > e^{-r} , \, |\zeta-w| > e^{-r}\}. \]
 \end{definition}
 
We could also have defined 
\[ \mu_D(z,w) =  c_0^{-1}\lim_{r \rightarrow \infty}
              e^{2(\tilde b - b)r}
                \, r^{\cent/2} \, \mu_{D_r}(z + e^{-r+i
               \theta},
             w + e^{-r+i\theta'}),\]
 for any $\theta,\theta'$. Alternatively,
 we could define
 \[  \mu_D(z,w) = c'\lim_{r \rightarrow \infty}
              e^{(\tilde b - b)r}
                 \, \mu_{D_{r,z}}(z + e^{-r+i
               \theta},
             w + e^{-r+i\theta'}),\]
where
\[    D_{r,z} = \{\zeta \in D: |\zeta - z| >
e^{-r} \}. \]
Our choice of definition has the advantage that it
follows immediately that $\mu_D(w,z)$ is the
reversal of $\mu_D(z,w)$.
If we want to let $D = \C$, we have to
renormalize.

\begin{proposition} If $z,w \in \Disk$, then
There exists $\tmass(z,w) \in (0,\infty)$ such
that  
\[     \tmass_{\Disk_{-r}}(z,w)  =
       \tmass(z,w) \, r^{-\cent/2} \,[
       1 +O(r^{-1})]. \]
 \end{proposition}

 \begin{proof} This essentially follows from  
  Proposition \ref{prop.nov30}.
 \end{proof}
  
 Using this as a guide, we define
 \[  \mu(z,w) =  c'
   \lim_{r \rightarrow \infty}
     r^{\cent/2} \, \mu_{A_{-r}}(z,w).\]
This satisfies the conformal covariance
rule
\[     f \circ \mu(z,w) = |f'(z)|^{\tilde b}
 \, |f'(w)|^{\tilde b} \, \mu(f(w),f(w) ),\]
where $f$ is a linear fractional transformation
(conformal transformation of the Riemann sphere).
Conformal covariance implies that there exists
$c'' \in (0,\infty)$ such that for all $z,w$,
\[      \tmass(z,w) = c''\, |z-w|^{-2 \tilde b}.\]
The probability measure $\mu^\#(z,w)$,  which is
invariant under linear fractional transformations,
is called \textit{whole plane $SLE_\kappa$}.
While we have defined $\mu(z,w)$ as a limit,
we could also imagine being able to define it
directly.  In this case, we get $\mu_D(z,w)$
by a (normalized) boundary perturbation rule.

\begin{proposition} \label{prop.subtle}
 If $D$ is a domain and
$z, w\in D$ are distinct,  then
\[ \frac{d\mu_D(z,w) }{d\mu(z,w)}(\gamma)
  = 1\{\gamma \subset D\}
  \,   \exp \left\{-\frac \cent 2 \,\Lambda(\gamma,
           \p D)\right\}
     \]
 where $\Lambda(\gamma,\p D)$ is as defined
 in Proposition \ref{prop.nov30}.
     \end{proposition}
     
\begin{proof}
For $r$ sufficiently large so that $\gamma 
\subset \Disk_{-r}$,
\[    \frac{d\mu_D(z,w)}
      {d\mu_{A_{-r}}(z,w)}(\gamma)
     = \exp \left\{\frac \cent 2 \, m_{A_{-r}}
      (\gamma,A_{-r} \setminus D)\right\}. \]
Proposition \ref{prop.nov30} implies that as
$r \rightarrow \infty$,
\[  m_{A_{-r}}
      (\gamma,A_{-r} \setminus D)
      = \log r - \Lambda(\gamma, \p D)
       + o(1) . \]
Therefore,  
  \[     \frac{d\mu_D(z,w)}
      {r^{\cent/2} \,
      d\mu_{A_{-r}}(z,w)}(\gamma)   
       \sim 
         \exp \left\{-\frac \cent 2 \,\Lambda(\gamma,
           \p D)\right\},\;\;\;\;
            r \rightarrow \infty . \]    
      
\end{proof}

\labove \textsf
{\begin{small}  \Heuristic  
While it might seem natural to define $\mu(z,w)$ 
using whole plane $SLE$ and then
   the proposition to define $\mu_D(z,w)$, there
is a disadvantage in this approach.  The reason is
that it is not so easy to prove that $\mu_D(z,w)$
satisfies the conformal covariance relation for 
conformal transformations of $D$ since the quantity
$\Lambda(\gamma,\p D)$ is not  conformally
invariant under transformations of $D$.
\end{small}}
\lbelow

\begin{example}  If $\kappa = 2$, then $\tmass_D(z,w)$
is proportional to the usual Green's function
for Brownian motion with
Dirichlet boundary conditions.  For this, it is well known that
\[    \tmass_{A_{-r}}(0,1) \sim r , \]
which agrees with the formula since $\cent = -2$.
Also, $\tilde b = 0$ which implies that $\tmass_D(z,w)$
is a conformal \textit{invariant}.  This is well
known for the Green's function.
   \end{example}

  \subsection{Multiple paths}
  
 Extending the definition of $SLE_\kappa$ to multiple
 is straightforward as in \cite{KL}.  Suppose
 $\z = (z^1,\ldots,z^k), \w = (w^1,\ldots,w^k)$
 are distinct analytic points in a domain $D$.
 The points can be bulk or boundary points.  The
 measure $\mu_D(\z,\w)$ is defined by giving its
 Radon-Nikodym derivative $Y$ with respect to
 the product measure
 \[  \mu_D(z^1,w^1) \times \cdots \times
           \mu_D(z^k,w^k). \]
 Let $\bar \gamma = (\gamma^1,\ldots,\gamma^k)$
 be a $k$-tuple of paths (modulo reparametrization)
 in $D$ where $\gamma^j$ goes from $z^j$ to $w^j$.
 Then 
\begin{equation}  \label{dec7.1}
  Y = 1\{\gamma^j \cap \gamma^l = \emptyset,
 \, j \neq l \} \, \exp \left\{\frac \cent 2
    \sum_{j=2}^k m_{D}(\gamma^j, 
      \gamma^1 \cup \cdots \cup \gamma^{j-1}) 
        \right\}. 
        \end{equation}

 \labove \textsf
{\begin{small}  \Heuristic  
One can consider the measure on multiple paths in
the context of the $\lambda$-SAW.  On the discrete
level, the measure on a $k$-tuple of paths 
$\bar \omega = (\omega^1,\ldots,\omega^k)$
is
\[      \exp\left\{-\beta (|\omega^1| + \cdots +
 |\omega^k|)  + \lambda  
   \, m^{RW}(\omega^1 \cup \cdots \cup
    \omega^k,D,n) \right\}.\]
The exponential factor on the right hand side of
\eqref{dec7.1} compensates for overcounting of
loops that intersect $\bar \gamma$.
\end{small}}
\lbelow       

\section{Crossing $SLE_\kappa$ in an annulus}  \label{crosssec}

In this section we study
 the measure $\mu_{A_r}(1,e^{-r+i\theta})$
  which
is a  measure on simple paths (modulo
reparametrization) $\eta$ from $1$ to $e^{-r+i\theta}$ in $A_r$.
Let us recall the definition.  Suppose $D'$ is a simply connected
subdomain of $A_r$ that agrees with $A_r$ in neighborhoods
of $1$ and $w = e^{-r + i \theta}.$  Then if $\eta$ is
a curve in $D'$ connecting $1$ and $w$,
\[   \frac{d \mu_{A_r}(1,w)}{d \mu_{D'}(1,w)}
  (\eta) = \exp \left\{- \frac \cent 2 \, m_{A_r}(\eta,
   A_r\setminus D')  \right\}. \]
We can write
\begin{equation}  \label{winding}
  m_{A_r}(\eta,A_r \setminus D') = \hat m_{A_r}(\eta,A_r \setminus
D') +    m^*(r) , 
\end{equation}
where $m^*(r)$ denotes the measure of the set of loops in $A_r$ of
nonzero winding number and  $ \hat m_{A_r}(\eta,A_r \setminus
D') $ is the measure of the set of loops of zero winding number
that intersect both $\eta$ and $A_r \setminus D'$. Here
we use the fact that every loop of nonzero winding number
intersects  both $\eta$ and $A_r \setminus D'$.   (This construction
assumes that there is a unique point on the Brownian loop
that goes through the point $\eta(t)$.  For each curve $\eta$ this is true
up to a set of loops of measure zero.  See the discussion after Theorem
12 in \cite{LW}.)

Let $\gamma$ be the continuous
preimage under $\psi$  of $\eta$ with $\gamma(0) = 0$, 
and let $D$ be the simply connected domain containing
$\gamma$ such that $\psi(D) = D'$.   Each loop $\ell'$
in $A_r$ has an infinite number of preimages under
$\psi$.   For each loop $\ell'$
in $A_r$ that intersects $\eta$,  we choose a unique
such preimage  as follows.  Consider the first time
$t$ such that $\eta(t) \in \ell'$. We   make $\ell'$
a rooted loop by choosing the root to be $\eta(t)$.
 Then we choose $\ell$
to be the (rooted) preimage of $\ell'$ that is rooted
at $\gamma(t)$.  The definition of $\ell$ implies that
if it is rooted at $\gamma(t)$, then 
\begin{equation}
\label{j5}
 \ell \cap  \tilde \gamma_t = \emptyset , 
 \end{equation}
where, as before, 
\[  \tilde \gamma_t = \bigcup_{k \in \Z \setminus
\{0\}} (\gamma_t + 2 \pi k).\]
We will call a loop $\ell$ {\em $\gamma$-good} if it
intersects $\gamma$ and  satisfies
\eqref{j5}.
Then $\ell \leftrightarrow \ell'$ gives a
 bijection between $\gamma$-good loops
in $S_r$ and loops in $A_r$
of zero winding number that intersect $\eta$.

If $r > 0$, $x \in \R$, we define the measure $\nu_{S_r}(0,x+ir)$
by the relation
\[  \frac{d\nu_{S_r}(0,x+ir)}{d\mu_{D}(0,x+ir)}
   (\gamma) = \exp \left\{ - \frac \cent 2 \, m_{S_r}(\gamma,S_r
   \setminus D;*) \right\},\;\;\;\;  \gamma \subset D \]
   where
$ m_{S_r}(\gamma,S_r
   \setminus D;*)$
denotes the Brownian loop measure of $\gamma$-good loops
in $S_r$ that intersect both $\gamma$ and $S_r \setminus D$.   
Recall that
\[  \frac{d\mu_{S_r}(0,x+ir)}{d\mu_{D}(0,x+ir)}
   (\gamma) = \exp \left\{ - \frac \cent 2 \, m_{S_r}(\gamma,S_r
   \setminus D) \right\}, \] 
   This leads to an alternative, equivalent definition
   of $\nu_{S_r}(0,x+ir) $.
 Note that $\psi \circ \gamma$ is a simple curve if and only if
 $\gamma \cap \tilde \gamma = \emptyset $. 
   
  \begin{definition}  The measure  $\nu_{S_r}(0,x+ir)$
  is the measure absolutely continuous with respect to 
  ${\mu_{S_r}(0,x+ir)}$ with Radon-Nikodym derivative
  \begin{equation} \label{anndef2}
   \frac{d\nu_{S_r}(0,x+ir)}{d\mu_{S_r}(0,x+ir)}(\gamma)
     = 1\{\gamma \cap \tilde \gamma = \emptyset \} \,
     \exp \left\{ \frac \cent 2 \, m_{S_r}(\gamma)\right\} , 
     \end{equation}
  where
$ m_{S_r}(\gamma) $ is the measure of loops in $S_r$
that intersect $\gamma$ but are not $\gamma$-good.
   We call this {\em annulus $SLE_\kappa$ in $S_r$ from
$0$ to $x+ir$}. 
   \end{definition}
   
 We can relate annulus $SLE_\kappa$ in $S_r$ to $SLE_\kappa$
 in $A_r$ by conformal covariance.  We define
$      \nu_{A_r}(1,x)  $ by
\begin{eqnarray} 
 \nu_{A_r} (1,x)
   &
   = & |\psi'(0)|^{-b} \, |\psi'(x+ir)|^{-b}   
\, e^{-\cent\,   m^*(r)/2}
   \, \psi \circ \nu_{S_r}(0,x+ir) 
 \nonumber \\
  & = &   e^{br} \, e^{-\cent \, m^*(r)/2}
   \, \psi \circ \nu_{S_r}(0, x + ir) ,  \label{anndef}
\end{eqnarray}
We think of this as annulus $SLE_\kappa$ from $1$
to $e^{-r + i x}$  restricted to curves
of a particular winding number. 
The term $e^{-\cent m^*(r)/2}$ is discussed
in Proposition \ref{windloops}.
 Annulus $SLE_\kappa$ is obtained
by summing over all winding numbers
\begin{equation}  \label{anndefsum}
   \mu_{A_r}(1,e^{-r + i\theta}) = \sum_{k \in \Z}
   \nu_{A_r}(1, \theta + 2 \pi k). 
   \end{equation}

 \subsection{Main result}
 We will show that the partition function for
 annulus $SLE$ on $S_r$ can be given in terms of a functional of
 locally chordal $SLE_\kappa$.  
Recall the functions  $\zhan_I$  from  Section
   \ref{annfunction}, $\funone$ from \eqref{fun1}, and $\funthree$  
 from  \eqref{fun3}.
 
%
%
   
\begin{theorem}  \label{main1}
If
$        \tilde \tmass(r,x) =  \|\nu_{S_r}(0,x+ri)\| , $
then 
\[     \tilde \tmass(r,x)   = V(r,x) \, \tmass_{S_r}(0,x+ri) . \]
Here
\begin{equation}  \label{vdef}
   V(r,x) = \E^x\left[\exp\left\{-2b\int_0^{ r} \funone(
  r -s,X_s) \, ds
 \right\} \right], 
 \end{equation}
 where $X_t, 0 \leq t \leq r$ satisifes
 \begin{equation}  \label{xsde2}
       dX_t = \left[\zhan_I( r-t,X_t) - b\kappa \, \funthree( r-t,
   X_t) \right] \, dt + \sqrt \kappa \, dB_{t} , 
   \end{equation}
   and $B_t$ is a standard Brownian motion.
In particular,  $ \tilde \tmass(r,x)$ is  $C^1$ in $r$, $C^2$ in
$x$   and
$\tilde \tmass(r,x) \leq \tmass_{S_r}(0,x+ri) .$
\end{theorem}

We used the functional in \eqref{vdef} as our definition,
but as we show now,  it is the solution of a PDE.  Let
us define $V(0,x) \equiv 1$.

 \begin{proposition}  \label{vprop}
The function $V(r,x)$  satisifes
 $0 \leq V(r,x) \leq 1$, is continuous
 on $[0,\infty) \times (-\pi,\pi)$
and  for $r > 0$ satisfies the equation
 \begin{equation}  \label{pde}
  \dot V = - 2b \,  \funone \, V
    + \left[ \zhan_I - b \kappa \, \funthree \right]
      \, V' + \frac \kappa 2 \, V'', 
      \end{equation}
where dot refers to $r$-derivatives and primes refer
to $x$-derivatives.

Moreover, for fixed $r$, $x \mapsto V(r,x)$ is an odd function
that is decreasing in $|x|$. 
\end{proposition}

\begin{proof}  For $r > 0$, the function $\zhan_I,\funthree$
are smooth and $\funone  \ge 0$.  Hence \eqref{pde}
follows from the Feynman-Kac formula, see, e.g,
 \cite[Section 6.5]{Friedman}  or \cite[Section 5.7.b]{KarS}.
Combining   \eqref{oct7.2} with Lemma \ref{locchord},
 we see that
$V(0+,x) = 1$ for $|x|  < 2 \pi$. 
For the last assertion, we use Proposition \ref{feb1.prop1} which
states that $\funone(r,x)$
is an increasing function of $|x|$.  It is not difficult to see that
if $0 < x_1 < x_2 < \infty$, then we can couple process
$X_t^1,X_t^2$ on the same probability space, each satisfying
\eqref{xsde} with $X_0^j = x_j$ and such that
$|X_t^1| \leq |X_t^2|$ for all $t$.  In this coupling, we have
\[ \int_0^{ r} \funone(
  r -s,X_s^1) \, ds  \leq  \int_0^{ r} \funone(
  r -s,X_s^2) \, ds.\]

\end{proof}

%
%

\subsection{Radon-Nikodym derivative}

Similarly to the approach for simply connected domains
as in Section \ref{simplesec}, we will 
  find an appropriate nonnegative local martingale and  use the
Girsanov theorem to analyze
the process weighted by the  local martingale.
  Suppose $(\Omega,\F,\hat \Prob)$ is
a probability space under which $U_t = - B_t$ is a standard Brownian
motion.  Let $g_t$ be the solution to the Loewner equation \eqref{chordal}
producing the random curve $\gamma$.  
Let $\gamma_t,\tilde \gamma_t, \hat \gamma_t$ be as above and fix
$r, z_0 = x+ ir$.
The following proposition is the particular case of
Section \ref{simplesec} for $D = D_r, w = x+ir$..

\begin{proposition}  If
\[           J_t = |g_t'(z_0)|^b \, H_{\p g_t(S_r\setminus
 \gamma_t)}(U_t,
             g_t(z_0))^b \, \exp \left\{\frac \cent 2
  \, m_\Half(\gamma_t,\Half \setminus S_r) \right\}, \]
then $J_t$ is a local martingale for $t < \tau_r$.  Moreover,
if one uses Girsanov, then under the weighted measure $\gamma$
has the distribution of $SLE_\kappa$ from $0$ to $x + ir$.  
\end{proposition}

Let $\Prob,\E$ denote expectations in the weighted measure under which
$\gamma$ has the distribution of  $\mu_{S_r}^\#(0,x+ir)$.

  If
$\ell$ is an (unrooted) loop in $S_r$, let
\[               \tilde s(\ell)  = \min\{t: \ell \cap \tilde \gamma_t
  \neq \emptyset \},\]
  \[    s(\ell) = \min\{t: \ell \cap \gamma_t \neq \emptyset\}. \]
It is not hard to show, using the fact that two-dimensional Brownian motion
does not hit points, that the loop measure of the set of loops with
$s(\ell) = \tilde s(\ell) < \infty$ is zero.
Let
\[         \Lambda_t =  \Lambda_t(\gamma_t, r) = 1\{T > t\} \, 
\exp\left\{  m_t \right \} , \]
where $m_t= m_{t,r}(\gamma_t)$ denotes the measure of the set of loops in $S_r$  that
satisfy
\[                  \tilde s (\ell) <  s(\ell) \leq t . \]
Theorem \ref{main1} can be rephrased as follows.

\begin{theorem}
If $\gamma$  has distribution $\mu_{S_r}^\#(0,x+ir)$,
 then 
\begin{equation}  \label{main11}
    \E\left[\Lambda_{\tau_r}^{\cent/2}\right] = V(r,x) . 
    \end{equation}
\end{theorem}

We will prove \eqref{main11} in a series  of  propositions.
Recall the definition of $\funone$ from \eqref{fun1}.
  Let
 \[    R_t = \Re[h_t(z_0)] - U_t , \;\;\;\;
   V_t =  V(r(t),R_t) , \]
\[   Q_t = Q_{S_{r} \setminus \gamma_t}(\gamma(t),z_0; S_r \setminus
 \hat \gamma_t),\;\;\;\;
  K_t = \exp \left\{2\int_0^t \dot r(s) \, \funone(r(s), R_{s})
 \, ds\right\}. \]
\begin{equation}  \label{nt}
  N_t = \Lambda_t^{\cent/2} \, Q_t^b \, K_t^{ab} , \;\;\;\;
  O_t = K_t^{-ab} \, V_t , 
\end{equation}
\[   M_t = N_t \, O_t =  \Lambda_t^{\cent/2} \, Q_t^b \, V_t . \]

By conformal invariance, 
\[  H_{\p g_t( S_{r} \setminus \hat \gamma_t)}(U_t,
             g_t(z_0))  = Q_t \, H_{\p g_t(S_r\setminus \gamma_t)}(U_t,g_t(z_0)).\]
Therefore,
\[   \phi_t'(U_t) \, |\phi_t'(g_t(z_0))| \,
 H_{\p h_t( S_{r}\setminus \hat \gamma_t)}(U_t^*,h_t(z_0)) =
              Q_t \, H_{\p g_t(S_r \setminus \gamma_t)}(U_t,g_t(z_0)), \]
              and hence
\[   \phi_t'(U_t) \, |h_t'(z_0)| \,
 H_{\p h_t( S_{r}\setminus \hat \gamma_t)}(U_t^*,h_t(z_0))  = |g_t'(z_0)|\, 
              Q_t \, H_{\p g_t(S_r \setminus \gamma_t)}(U_t,g_t(z_0)).\]
Therefore, we can write
\[ J_t \, N_t =    C_t(z_0) \,
 H_{\p h_t( S_{r}\setminus \hat \gamma_t)}(U_t^*,h_t(z_0)) ^b
     , \]
   where
   \[  C_t(z_0) =   \phi_t'(U_t)^{-b} \,  |h_t'(z_0)|^b  
     \,  \,\exp \left\{ - \frac \cent
      2 \,  m_\Half(\gamma_t,\Half \setminus S_r) \right\}\,
   \Lambda_t^{\cent/2} \, K_t^{ab},  
   \]          
Important observations are that $C_t(z_0)$ is $C^1$ in $t$ and
$C_t(z_0) = C_t(z_0 + 2\pi)$.

\begin{lemma}  \label{estimatetwo}
Suppose $\gamma$ is a parametrized
with $\hcap[\gamma(0,t]] = at$. Let $\tilde \gamma_t$
be as above and $Q_t =Q_{S_r \setminus \gamma_t}(\gamma(t),x+ir;S_r
\setminus \hat \gamma_t) $. 
Then
\[  \p_t Q_t \mid_{t=0}
  =  -a\funone(r,x). \]
\end{lemma}

%
%
%
%

\begin{proof}  See \eqref{tomato}. \end{proof}

\begin{proposition} \label{prop1}   $\;$

 \begin{itemize}
 \item
  $N_t$ is a local martingale with respect
  to $\Prob$ for
$t < T \wedge \tau_r$.  In particular, $J_t \, N_t$ is
a $\hat \Prob$-local martingale.
\item With respect to $\Prob^*$, the curve $\gamma$
at time $t$ grows
like $SLE_\kappa$ from $\gamma(t)$ to $z_0$ in $\tilde S_{t,r}$.
\end{itemize}
\end{proposition}

\begin{proof}  This is a particular case of
Section \ref{shrinksec}.
\end{proof}

Let  $\Prob^*, \E^*$ denote the probabilities and expectations
 obtained from $\Prob$ by weighting by the local martingale $N_t$.
This is the same as the measure obtained from $\tilde \Prob$
by weighting by $J_t \, N_t$.  We have seen that this
is locally chordal $SLE_\kappa$ and we can consider the  
 path in the annulus parametrization.

%
%
%

\begin{proposition}  Suppose $V$ is as defined
in \eqref{vdef}.
Then
\[        M_t^* = \exp\left\{-2b\int_0^t \, \funone(r-s, R_s^*)
 \, ds\right\} \, V(r-t, R_t^*) , \]
is a local martingale satisfying
\begin{equation}  \label{dec11.1}
   dM_t^* = \sqrt \kappa \,  \frac{V'(r-t,R_t^*)}{
        V(r-t,R_t^*)}\, M_t^* \, dB_t.
        \end{equation}
Moreover, if we weight by the local martingale using
Girsanov theorem then with probability one in the
weighted measure, $R_{r-}^* = 0$.
\end{proposition}
 
\begin{proof}  The relation \eqref{dec11.1}
follows immediately from It\^o's formula.
For the second claim, we note that in the unweighted
measure we have $R_{r-}^* = 0$.   
Since $V$ is decreasing in
$|x|$, the additional drift given by the weighting points
toward the origin.
\end{proof}

\begin{proposition}  Suppose $\gamma$ is a simple curve
in $S_r$ from $0$ to $z_0$ with $T > \tau_r$.  Then,
\[      M_{\tau-} = \Lambda_{\tau-}^{\cent/2}  \in (0,\infty).\]
\end{proposition}

\begin{proof}  Easy estimates show that under the assumptions,
$Q_{\tau - } = 1$, $r(\tau-) = 0$,
 $R_{\tau - } = 0$.  Proposition
\ref{vprop}, then gives $V_{\tau-}=1$.
The assumptions also imply that $\dist(\gamma_\tau,
\tilde \gamma_\tau) > 0$, which implies 
$0 < \Lambda_{\tau-} < \infty$.
\end{proof}

\begin{proposition}
$O_t, t < \tau_r \wedge T$ is a local martingale with respect to $\Prob^*$.
In particular, $M_t = N_t \, O_t, t < \tau_r \wedge T$ is a local martingale
with respect to $\Prob$, and $J_t \,M_t$ is a local martingale with
respect to $\hat \Prob$.
\end{proposition}

\begin{proof}  This is a restatement of the previous
proposition in terms of
 the original parametrization. \end{proof}

Let $\Prob'$ denote the probability measure
obtained from weighting by the local martingale $M_t$.  

\begin{proposition}  With
$\Prob'$ probability one, $\tau_r < T$ and 
\begin{equation}  \label{oct15.2}
             M_0 = V(r,x) , \;\;\;\;
             M_{\tau_r} = \Lambda_{\tau_r}^{\cent/2} ,
\end{equation}
In particular,
\[    V(r,x) = M_0 = \E\left[M_{\tau_r}\right]
  = \E\left[ \Lambda_{\tau_r}^{\cent/2} \right]. \]
\end{proposition}

\begin{proof}  The drift given by weighting by this martingale
has a stronger drift to the origin than for locally chordal $SLE_\kappa$
and we know that that the latter one is good.
\end{proof}

\labove \textsf
{\begin{small}  \Heuristic  
 There is a general principle that is being used here that is worth
 stressing.  Suppose $M_t$ is a positive local martingale for
 $t < \tau$.  The
 martingale convergence theorem implies that with probability one
 the limit $M_\tau = \lim_{t \rightarrow \tau-} M_t$ exists.  However,
 one cannot conclude $\E[M_0] = \E[M_\tau]$ without more assumptions.
 One way to establish this equality is to consider the paths
 weighted by the local martingale.  If $M_\tau$ exists and is finite
 with probability one \textit{in the new measure}, then we have
 uniform integrability and $\E[M_0] = \E[M_\tau]$.    In our
 case we establish that in the new measure we have $R^*_{r-} = 0$.
 If the latter fact holds, then we use an easy  
 deterministic estimate about curves to see that $M_\tau < \infty$.
\end{small}}
\lbelow

\labove \textsf
{\begin{small}  \Heuristic  
 At this point of the paper, the argument went very quickly, so it
 is a good idea to explain what has happened.  The goal was to
 estimate the expectation (with respect to chordal $SLE_\kappa$
 in $S_r$ from $0$ to $z_0$)
 of a random variable which is the
 exponential of the measure of a certain set of bad loops.  
 For a curve $\gamma$ and a loop $l$, we say that $l$ is bad
 if $l$ intersects $\gamma$, say at first time $s'$, but
 also intersects $\tilde \gamma$ at first time $s < s'$.  Suppose
 we have seen $\gamma_t$.  Then we can split the bad loops
 into three sets: those with $s < s' \leq t$; those with $s < t<
 s'$; and those with $t < s < s'$.  When we weight only by
 the first two sets of loops, we get the local martingale
 $N_t$, and the probability
  measure   is
 locally chordal $SLE_\kappa$.  Lemma \ref{locchord}
 shows that this is supported
  simple curves with $\gamma \cap \tilde \gamma= \emptyset$.
  We then weight again to include the third set of loops and
  this leads to the function $V$.  Since we can show
  directly that $V$ is decreasing in $|x|$ (and here we were
  lucky with the monotonicity proved in Proposition 
  \ref{feb1.prop1}), we can see that
  the extra drift given by weighting by these loops points
  towards the origin and hence this measure is also
  supported
  simple curves with $\gamma \cap \tilde \gamma= \emptyset$.
  This allows us to justify the equation $\E[M_0] = \E[M_{\tau_r}].$
 \end{small}}
 \lbelow

  \section{Annulus $SLE_\kappa$ from $0$ to $x$ in $S_r$}
  
The same ideas can be used to
analyze  $\nu_{S_r}(0,x)$ where $0 <| x|  < 2 \pi$.  For ease,
 we will assume $x > 0$, but the $x < 0$ case is done
the same way.   We will only sketch the ideas, since this case
is also considered in \cite{Zhanannulus}.
 Topological constraints restrict the values of $x$;
if $|x| \geq 2 \pi$ and $\gamma$ connects $0$ and $x$, then
$\eta = \psi \circ \gamma$ is not simple.  As before, 
 we define the measure by giving the Radon-Nikodym derivative
 as in \eqref{anndef2}
 \[    \frac{d\nu_{S_r}(0,x)}{d\mu_{S_r}(0,x)} (\gamma)= 
   Y(\gamma) = 1\{\gamma \cap \tilde \gamma = \emptyset\}
  \, \exp\left\{\frac \cent 2\, m(\gamma)\right\} . \]
  The relevant functions are the following.
  
\begin{equation}  \label{funner1}
   \funnerone(r,x) = \frac{\pi^2} {4  r^2}
  \sum_{ k \in \Z \setminus \{0\}}  \frac{
    \sinh^2(\pi x/2r) }{\sinh^2(\pi^2 k/r)
      \, \sinh^2(\pi(x-2 \pi k)/2r)},
      \end{equation}

\begin{equation}  \label{funner2} 
  \funnertwo(r,x)  
  = \frac{\pi}{2r} \,  \coth\left(\frac{\pi x}{2r}
  \right) + \frac{\pi}{2r} \sum_{k=1}^\infty \left[ \coth\left(\frac{\pi (x
   + 2 \pi)}{2r}\right) +
   \coth\left(\frac{\pi (x
   - 2 \pi)}{2r}\right)\right].
 \end{equation}

 \begin{equation} \label{funner3}
 \funnerthree(r,x)= - \frac{\p_x H_{S_r}(0,x)}{b\, H_{S_r}(0,x)} = 
 \frac{ \pi 
 }{r} \, \coth \left(\frac {\pi x }{2r}\right)  . 
 \end{equation}

 \begin{lemma}  \label{oct12.lemma11}
If $y \in \R$ and
\[         f(x) = \frac{\sinh^2x}{\sinh^2(x-y)}
          + \frac{\sinh^2 x}{\sinh^2(x+y)}, \]
then $f$ is increasing for $0 \leq x < y$.
\end{lemma}

\begin{proof}
Since
\[   f(x) = \frac{\cosh(2x) - 1}{\cosh(2x - 2y) - 1}
              + \frac{\cosh(2x) - 1}{\cosh(2x + 2y) - 1},\] 
               it suffices to show for every $y \in \R$,  that
\[   F(x) = \frac{\cosh x - 1}{\cosh(x-y)- 1}
              + \frac{\cosh x - 1}{\cosh(x+y)- 1}, \]
is increasing for $0 \leq x < y$. 
Using the sum rule, we
get
\[ \cosh(x-y)- 1 +  \cosh(x +y)- 1 =
         2 \cosh x \cosh y - 2 , \]
         Letting $r = \cosh y \geq 1$, we get
\begin{eqnarray*}
 [\cosh(x-y) - 1] \, [\cosh(x+y) - 1]
& = &      (\cosh x \cosh y - 1)^2 - \sinh^2 x \sinh^2 y \\
 & = &     (r \cosh x - 1)^2 - (r^2 - 1) (\cosh^2x - 1) \\
 & = & \cosh^2 x - 2r \cosh x + r^2 \\
  & = & (\cosh x - r)^2 .
 \end{eqnarray*}
Therefore,
\[
    F(x)    =  \frac{2r\, ( \cosh x - r^{-1}) \, (\cosh x - 1)}{(\cosh x - r)^2}
      = 2 r \, e^{G(\cosh x)}, \]
  where
  \[   G(t) = \log (t - \frac 1 r) + \log (t+1) -
       2 \log \, (t-r) . \]
 Since $r \geq 1$, $G'(t) > 0$ for $0 < t < r$ and hence $G$ and
 $F$ are increasing.
 
 \end{proof}

 \begin{definition}  The function $\tilde V(r,x), 0 \leq r < \infty,
0 < x < 2 \pi $ is defined by 
\begin{equation}  \label{vdef2}
   \tilde V(r,x) = \E^x\left[\exp\left\{-2b\int_0^{ \sigma} \funnerone(
  r -s,X_s) \, ds
 \right\} \right], 
 \end{equation}
 where $X_t, 0 \leq t < \sigma$ satisifes
 \begin{equation}  \label{xsde22}
       dX_t = \left[2\funnertwo( r-t,X_t) - b\kappa \, \funnerthree( r-t,
   X_t) \right] \, dt + \sqrt \kappa \, dB_{t} , 
   \end{equation}
with $B_t$ is a standard Brownian motion and
\[  \sigma = \inf\{t: X_t = 0 \}. \]
We define $V(r,0) = 1$. 
   \end{definition}

 An important observation is that if $X_t$ satisfies \eqref{xsde22}
 with $X_0 \in [0,2\pi)$, then
 with probability one $\sigma < r$ and $X_t \in [0,2\pi)$
 for $0 \leq t \leq \sigma$.  Hence this is well defined.  
 The function $\tilde V, 0 < r < \infty, 0 <  x < 2 \pi$
 \begin{equation}  \label{pde2}
  \dot {\tilde V}(r,x) = - 2b \,  \funnerone(r,x) \, \tilde V(r,x)
    + \left[ 2\funnertwo(r,x) - b \kappa \, \funnerthree(r,x) \right]
      \, \tilde V'(r,x) + \frac \kappa 2 \, \tilde V''(r,x), 
      \end{equation}
where dot refers to $r$-derivatives and primes refer
to $x$-derivatives.

 The definition of $\mu_{A_r}(1,e^{i\theta})$ for $0 < \theta < 2 \pi$
 takes a little more thought.  We write
 \[   \mu_{A_r}(1,e^{i\theta}) = \mu_{A_r}(1,e^{i\theta};R)+
 \mu_{A_r}(1,e^{i\theta};L) \]
 where $\mu_{A_r}(1,e^{i\theta};R) $ denotes $\mu_{A_r}(1,e^{i\theta})$
 restricted to curves $\eta$ such that the origin lies in the
 component of $\Disk \setminus \eta$ whose boundary includes
 $(e^{i\theta},1)$.   Then similarly to \eqref{anndef} we write
 \[   \frac{ d \mu_{A_r}(1,  e^{ix})}{d[\psi \circ \nu_{S_r}(0,x)]}
 (\eta)
    = e^{br} \, \exp\left\{-\frac \cent 2 m^*(r,\eta) \right\}, \]
    where $m^*(r,\eta)$ denotes the measure of the set of
    loops in $A_r$ of nonzero winding number that intersect
  $\eta$.  Unlike the crossing case, the quantity on the right
  hand side depends
  on $\eta$.   It is not hard to give an expression for this.  Let
  $\tilde A$ denote the component of $A_r \setminus \eta$ that
  contains $C_r$ on its boundary. let $r_\gamma = r_{\gamma,r}$ be such that
  $\tilde A$ is conformally equivalent to $A_{r_\gamma}$.  Then
  $m^*(r_\gamma)$ denotes the measure of loops of nonzero
  winding number in $\tilde A$ and hence 
  \[     m^*(r,\eta) = m^*(r) - m^*(r_\gamma). \]

 We could have also defined $\mu_{A_r}(1,e^{i\theta})$ by
 \[  \frac{d \mu_{A_r}(1,e^{i\theta})}{d \mu_\Disk(1,e^{i\theta})}(\gamma)
    = 1\{\gamma \subset A_r\} \, \exp \left\{\frac \cent 2 \,
               m_\Disk(\gamma,\Disk \setminus A_r)\right\}. \]
   Since these both satisfy \eqref{oct26.1}, they must give the
   same measure.

   There is a subtlety that is worth mentioning.  Let $J$ denotes the closed disk about $0$ of
   radius $e^{-r}$ so that $A_r = \Disk \setminus J$
   and   $f: A_r
   \rightarrow D \subset \Disk$ is a conformal transformation that sends $\p \Disk$ to
   $\p \Disk$.   Informally we can write $f(J) = K$ where $D =
   \Disk \setminus K$, but the conformal map $f$ is not defined on $J$. 
   If $z,w \in \p \Disk$, then
   \[    f \circ \mu_{A_r}(z,w) = |f'(z)|^b \, |f'(w)|^b \, \mu_{D}(f(z),f(w)).\]
   This gives one way to construct $\mu_D(f(z),f(w))$.  But we also define
   it by the Radon-Nikodym derivative.  Suppose $\gamma \subset A_r$,
   then    $ f\circ\gamma \subset D$ and 
   \[   \frac{d \mu_{A_r}(z,w)}{d\mu_\Disk(z,w)}(\gamma) = 
    \exp \left\{\frac \cent 2 m_\Disk(\gamma,J)\right\}, \]
   \[   \frac{d \mu_{D}(f(z),f(w))}{d \mu_\Disk(f(w),f(w))} (f \circ \gamma)
       =   \exp \left\{\frac \cent 2 m_\Disk(f \circ \gamma,K)\right\}. \]
    However, since $f$ is not a conformal transformation of the disk,
    we have no reason to believe that 
$  m_\Disk(\gamma,J) = m_\Disk(f \circ \gamma,K).$

       \section{Annulus $SLE_\kappa$ in $A_r$}
       \label{partsec}

 In the last section we considered the measure $\nu_{S_r}(0,x+ir)$
 which we called annulus $SLE_\kappa$ in the strip $S_r$.
 This was analyzed by comparing the measure to chordal $SLE_\kappa$
 in $S_r$.     Recall from \eqref{anndef} that the 
 measure on paths given by
  annulus $SLE_\kappa$ restricted to a particular
 winding number is 
 \[   \nu_{A_r}(1,x) = e^{br} \, e^{-\cent m^*(r)/2} \,
    \psi \circ \nu_{S_r}(0,x+ir).\]
The term $e^{br} = |\psi'(x+ir)|^b$ comes from conformal covariance
and $m^*(r)$ is the Brownian loop measure of loops in $A_r$ of nonzero
winding number.  Annulus $SLE_\kappa$ in $A_r$ from $1$ to
$e^{-r+i\theta}$ is obtained from summing
over all winding numbers
\[   \mu_{A_r}(1,e^{-r+i\theta}) = \sum_{k \in \Z}
   \nu_{A_r}(1,\theta + 2\pi k). \]
 In this section we will compare $  \nu_{A_r}(1,x)$ and
 $   \mu_{A_r}(1,e^{-r+i\theta})$ to 
 to radial $SLE_\kappa$ in order to derive PDEs for the annulus
 partition
 functions.  We will rederive an equation   
   from \cite{Zhanannulus}.
 
 \subsection{The differential equation}  \label{diffsec}

   Let $\tilde \tmass(r,x)
  = |\nu_{S_r}(0,x+ir)|$ be
 as in the previous section, and  let  $\hat F(r,x)$ and
 $F(r,x)$ denote the partition functions associated to
 annulus $SLE_\kappa$ and annulus $SLE_\kappa$ restricted to
 a particular winding number, respectively.  In other words,
 \[    F(r,x) = |\nu_{A_r}(0,x)| = \beta(r) \, \tilde \tmass(r,x) , \]
 where
\[  \beta(r) =  \exp \left\{br - \frac{\cent m^*(r)}{2} \right\}
    = e^{br} \, e^{-\cent r/12}
     \, \exp\left\{\cent  \int_0^r \delta(s) \,ds \right\}
      ,\]  and
      \[ \hat F(r,x) = \tmass_{A_r}(1,e^{-r + ix})
       =  \sum_{k=-\infty}^\infty   F(r,
        x + 2 \pi k). \]
Since
\begin{equation}  \label{converge}
     F(r,x) = \beta(r) \, \tilde \tmass(r,x)
    \leq \beta(r) \, \tmass_{S_r}(0,e^{x + ir})
      \asymp \beta(r) \, r^{-2b} \, \left[
      \cosh \left(
     \frac{\pi x}{2r} \right) \right]^{-2b}, 
     \end{equation}
     we see that $\hat F(r,x) < \infty$. 
       Recall the functions $\nzhan$ and $\zhan_I$
        from Section \ref{annfunction}.
       As before, we will use dot for $r$-derivatives
       and primes for $x$-derivatives.
 
 \begin{proposition}  \label{greatprop}
  $F$ satisfies the differential equation
\begin{equation}  \label{fpde2}
  \dot F = \frac{\kappa}{2} \, F''
  + \zhan_I\, F' +  \left[b \zhan_I' + b +
  \tilde b \, (6\Gamma(r) - 1)
      - \frac br\right]
    \, F.
\end{equation}
Moreover, $\hat F$ satisfies the same equation.
   \end{proposition}

   \labove \textsf
{\begin{small}  \Heuristic  
As in \cite{Zhanannulus}, we
 check that this is consistent with what we
   know about $\kappa = 2$ for which $b=1, \tilde b=0$.
   For $\kappa = 2$, from  arguments based on the loop-erased
   walk we know that the $SLE_2$ partition function for any domain $D$
   should
   be given by a multiple of the excursion Poisson kernel,
   $H_{\p D}(z,w)$.
Hence a solution to \eqref{fpde2} should be 
   \begin{eqnarray*}
      \hat F(r,x) & =  &  H_{\p A_r}(1,e^{-r+ix})\\
    & =  &  e^{r} \sum_{k\in \Z} H_{\p S_r}(0, x + 2\pi k + i)\\
      &  = & \frac 12 \, {e^r} \, \nzhan(r,x) .
      \end{eqnarray*}
If this is so, then
Proposition \ref{greatprop} implies that if
$
  \Phi(r,x) = 2re^{-r} \hat F(r,x) = r \, \nzhan(r,x) , 
$
then
$        \dot \Phi =  \Phi'' + \zhanh_I \, \Phi' + 
\zhanh_I'\, \Phi . $
But we noted this relation in   \eqref{kappa2}.
 \end{small}}
 \lbelow
  
 We set
 \[ \alpha (r) =   b + \tilde b \, [6\Gamma(r) - 1] =
  b -\tilde b +  (2b + \cent) \, 
\,\Gamma(r) , \]
\[ \newzhan(r,x) = \newzhan_\kappa(r,x) = 
\zhan_I'(r,x) +  \frac{\alpha(r)}b
      - \frac 1r ,\]
which allows us to write \eqref{fpde2}
as
\begin{equation}  \label{fpde}
  \dot F = \frac{\kappa}{2} \, F''
  + \zhan_I\, F' +  b \,\newzhan 
    \, F.
\end{equation} 
  
  
  We will establish \eqref{fpde} for $F$.  We note that 
 $F(r,x)$ 
 is $C^1$ in $r$ and $C^2$ in $x$.  Indeed, in the
 previous section we showed the same for $\tilde \Psi(r,x)$, and
 it is easy to show that $m^*(r)$ is continuous in $r$
 and hence $\beta(r)$ is differentiable.
 Hene we can  use It\^o's formula freely.
 Before proceeding, let us show that this will
 also imply the result for $\hat F$.
  Let $X_t^{(r)},0\leq t \leq r$, denote
 a solution to the SDE
 \begin{equation}  \label{xsde}
   dX_t^{(r)} = \zhan_I(r-t,X_t^{(r)})
  \, dt + \sqrt \kappa \, dB_t. 
  \end{equation}
  Then \eqref{fpde} and the Feynman-Kac formula implies
 that for $r > t > 0$,
 \[  F(r,x) = \E^x\left[F(r-t,X_t^{(r)})
  \, \exp \left\{b\int_0^t \newzhan(r-s,X_s^{(r)})
    \, ds\right\} \right], \]
    where $\E^x$ denotes expectations assuming
    $X_0^{(r)} = x$.
  (We do not need to consider the delicate case $t=r$
  so the conditions for the Feynman-Kac formula are easily
  verified.)
 Using this and \eqref{converge}, we can see that
 \[   \hat F(r,x) = \E^x\left[\hat F(r-t,X_t^{(r)})
  \, \exp \left\{b\int_0^t \newzhan(r-s,X_s^{(r)})
    \, ds\right\} \right], \]
 and by invoking  the Feynman-Kac theorem again, we see
 that $\hat F$ also satisfies \eqref{fpde}.

To prove the proposition for $F$
we compare radial $SLE_\kappa$ (from $1$ to $0$
in $\Disk$) and
annulus $SLE_\kappa$ (from $1$ to $e^{-r + i \theta}$ in 
$A_r$) for $\kappa = 2/a \leq 4$.  These measures,
restricted to an initial segment of the path which has
not reached $C_r$, are absolutely continuous.

It is useful to view radial $SLE_\kappa$ raised
onto the covering space $\Half$ as we now describe. 
  We describe radial
$SLE_\kappa$ as a periodic function on $\Half$.   
We return to the radial Loewner equation \eqref{radloew}
which we write as 
 \begin{equation}  \label{feb7.1}
 \p G_t(z) =  \frac a 2 \, \cottwo(G_t(z) - U_t) ,
 \;\;\;\;  G_0(z) = z ,
 \end{equation}
 and view as an equation on $\Half$.
Here $U_t$
 is a standard Brownian motion with $U_0 = 0$
 and $\cot_2(z) = \cot(z/2)$.
 There is a corresponding curve $\gamma$ in $\Half$
 such that with probability one, for all $t$,
 $\gamma_t \cap \tilde \gamma_t = \emptyset.$
%
%
%
%
%
Let $\eta_t = \psi \circ \gamma_t$ and define
 $\tilde g_t$ by
 \[       \tilde g_t(e^{iz}) = e^{i G_t(z)}. \]
Then $\tilde g_t$ is the unique conformal transformation
of $\Disk \setminus \eta_t$ onto $\Disk$ with
$\tilde g_t(0) = 0, \tilde g_t'(0) > 0$.  In fact,
$\tilde g_t'(0) = e^{at/2}$.    Radial $SLE$ is usually
described in terms of the differential
equation for $\tilde g_t$.

We now relate the equation \eqref{feb7.1}
to the annulus Loewner
equation described in Section \ref{anneqsec}. 
We fix an ``initial radius'' $r$.  As in
that section, 
  we   define $r(t)$ and $h_t$   by saying that 
\[  h_t:  S_r \setminus \hat \gamma_t
\rightarrow S_{r(t)} \]
is a conformal transformation satisfying
$h_t(z+2\pi) = h_t(z) + 2\pi$ 
with $h_t(\pm \infty) = \pm \infty$ and $h_t(\gamma(t))
= U_t$. Recall that
\[  \p_t h_t(z) =  2 \dot r(t)\, \cpois_{r(t)}
  (h_t(z) - U_t) . \]
    We define $\Phi_t$ by
\[          h_t = \Phi_t \circ G_t, \]
and define $\tilde h_t, \tilde \Phi_t$ by
\[  \tilde h_t(e^{iz}) = e^{i h_t(z)},\;\;\;\;
  \tilde  \phi_t(e^{iz}) = e^{i \Phi_t(z)}, \]
  so that
  $     \tilde h_t = \tilde \phi_t \circ \tilde g_t.$
Note that $\tilde h_t$ is the unique conformal transformation
of $A_r \setminus \eta_t$ onto $A_{r(t)}$ with
$\tilde h_t(\eta(t)) = e^{iU_t}$.  Also, for real $x$, 
\[  |\tilde \phi_t'(e^{ix})| = \Phi_t'(x).\]
We note that \eqref{dec4.2} implies that
for $r(t) \geq 2$ and $x \in \R$, 
\[  |\Phi_t'(x)| = 1 + O(e^{-r(t)}), \;\;\;\;
    |\Phi_t''(x)| = O(e^{-r(t)}). \]

As in that section, we let
\[  \sigma_s = \inf\{t: r(t) = s \}, \;\;\;\;
  h_s^* = h_{\sigma_s} , \]
and we set
\[    \tilde h_s^* = \tilde h_{\sigma_s} ,\;\;\;\;
  \tilde \phi_s^* = \tilde \phi_{\sigma_s}, \;\;\;\;
    \tilde g_s^* = \tilde g_{\sigma_s}. \]

\begin{lemma} \label{magic} Under the assumptions above,
\[  \p_t |\tilde \phi_t'(1)| \mid_{t=0}=
 \p_t \Phi_t'(0) \mid_{t=0} = a \, \left[
\Gamma(r)
  -
 \frac{1}{2r} \right], \]
%
 \[   \p_s |(\phi_{r-s}^*)'(1)|\mid_{s=0} = 
  2\, \Gamma(r)-  \frac 1{r}
  .\]
  Here $\Gamma(r)$ is as defined in \eqref{Gammadef}.
  \end{lemma}
%

\begin{proof}
Note that
\[  \frac a2 \, \cot_2(z) = a \left[
\frac 1z - \frac{z}{12}\right]  + O(|z|^3).\]
Recall that $\dot r(0) = -a/2$ and from \eqref{expand}
we have 
\[  -a\cpois_r(z) =  a\left[\frac 1z + z \, \left(\Gamma(r) - 
\frac 1{12} 
  -
 \frac{1}{2r} \right)\right] + O(|z|^3),\]
Therefore, the first result follows from \eqref{nov18.1}
and the second from $\tilde \phi_t = \tilde \phi^*_{r(t)}.$
\end{proof}

%
  
%
 
Let $\mu_1,\mu_2,\mu_3$ denote $\mu_{\Disk}(1,-1),
\mu_\Disk(1,0),$ and $ \nu_{A_r}(1,x)$, respectively,
and let $w = e^{-r + i x}$.  Let $z_t = e^{i U_t}
= \tilde g_t(\eta(t)), \zeta_t
= \tilde g_t(-1), w_t = \tilde g_t(w), x_t = \arg w_t$,
where $x_t$ is chosen to be continuous in $t$ with $x_0 = x$.
If $t < \tau_r$, these three measures are absolutely continuous
with respect to each other and we can write down the 
Radon-Nikodym derivatives.  Recall from Section \ref{fradsec} that 
\[       \frac{d\mu_2}{d\mu_1}(\eta_t)
           =\frac{\tilde g_t'(0)^{\tilde b}\,
             \tmass_{\Disk}(z_t,0)}{
               |\tilde g_t'(-1)|^b \, \tmass_{\Disk}(z_t,
                \zeta_t)}  = \frac{\tilde g_t'(0)^{\tilde b}}{
               |\tilde g_t'(-1)|^b \, \tmass_{\Disk}(z_t,
                \zeta_t)}. \]
Using similar reasoning for annulus $SLE$ with respect
to chordal $SLE$, we get 
\[       \frac{d\mu_3}{d\mu_1}(\eta_t)
           =\frac{|\tilde g_t'(w)|^{ b}\,
              |\nu_{\tilde g_t(A_r)}( z_t,
              x_t)|\,\exp\left\{
                \frac \cent 2 \, m_{\Disk}(\Disk_r,\eta_t)\right\}}{
               |\tilde g_t'(-1)|^b \, \tmass_{\Disk}(z_t,
                \zeta_t)} .\]  
 We have not actually defined the measure
 $\nu_{\tilde g_t(A_r)}( z_t,
              x_t)$, so let us describe it now.
               Since $\tilde g_t(A_t)$ is a conformal
  annulus whose outer boundary is the unit circle, we can
  define $\nu_{\tilde g_t(A_r)}( z_t,
              x_t)$ in the same way that $\nu_{A_r}(1,x)$ was defined.
     In other words, it is annulus $SLE$ between $z_t$ and $w_t$ in
     the conformal annulus $\tilde g_t(A_r)$ restricted to curves
     of a particular winding number.  The choice of winding number
 is determined  by continuity in $t$.

Let
\[   M_t =  \frac{d\mu_3}{d \mu_2}(\eta_t)
          =
       \tilde g_t'(0)^{-\tilde b} \,
           |\tilde g_t'(w)|^{ b}\,
             |\nu_{\tilde g_t(A_r)}( z_t,
              x_t)|  \,\exp\left\{
                \frac \cent 2 \, m_{\Disk}(\Disk_r,\eta_t)\right\}
 .\]
We see that   $ M_t$ is a
local martingale for radial $SLE_\kappa$.
  Let
$\tilde h_t = \tilde \phi_t \circ \tilde g_t$.
Conformal covariance implies that
\[    |\nu_{\tilde g_t(A_r)}( z_t,
              x_t)|= |\tilde \phi_t'
              (e^{iU_t})|^b \, |\tilde \phi_t'(\tilde{g}
            _t(w))|^b \, \tmass_{A_{r(t)}}(e^{iU_t},
            \tilde \phi_t(\tilde g_t(w))) . \]
Therefore,  
\[    M_t = 
  \tilde g_t'(0)^{-\tilde b}\,
              |\tilde\phi_t'(e^{iU_t})|^b
              \, \exp\left\{\frac \cent 2 \, 
                 m_{\Disk}(\Disk_r,\eta_t)\right\}
  \,
           |\tilde h_t'(w)|^{ b}\,
             \annmass(r(t),
         R_t ),\]
 where
 \[           
  R_t = \Re [  h_t(z) - U_t].\]
We have shown   the following.
 
 \begin{proposition}
 If $U_t$ is a standard Brownian motion, then
 \[   M_t = 
           J(t) \, \annmass(r(t),
         R_t )
            ,\]
 is a local martingale
 where
 \[      J(t) = \tilde g_t'(0)^{-\tilde b}\,
              |\tilde\phi_t'(e^{iU_t})|^b
              \, \exp\left\{\frac \cent 2 \, 
                 m_{\Disk}(\Disk_r,\eta_t)\right\}
  \,
           |\tilde h_t'(w)|^{ b},\]
and  $R_t = \Re[h_t(x+ir) -
 U_t]$.
\end{proposition}

Using this proposition, we can write down a differential
equation for $ \annmass(s,x)$.  It is convenient 
to 
write the local martingale in the annulus parametrization.
Let $U_s^* = U_{\sigma_{r-s}}$.  Then $U_s^*$ is a
martingale with quadratic variation $\sigma_{r-s}$.
Let $R_s^* = h_{r-s}^*(z) - U_s^*$.   Then
\[    dR_s^* =  \p_s h_{r-s}^*(z) \,dt
    + d U_s^*\]
Note  that
\[ \p_s h_{r-s}^*(z) = \zhan_I(r,x),\;\;\;\;
 \p_s \sigma_{r-s}\mid_{s=r} = 2/a = \kappa.\]
The last proposition becomes the following.

\begin{proposition}
For fixed $r >0$, if  $R_s^* = h_{r-s}^*(z) - U_s^*$
and 
\[ 
 M_s^* =  J^*(r - s) \, F(r-s, R_s^*) , \]
where
\[ J^*(s) = 
 (\tilde  g_s^*)'(0)^{-\tilde b} \, |\tilde(\phi_s^*)'(e^{iU_{\sigma_s}})|^b
              \, \exp\left\{\frac \cent 2 \, 
                 m_{\Disk}(\Disk_s
                 ,\eta_{\sigma_s})\right\}
  \,
           |(\tilde  h_s^*)'(w)|^{ b}  
            ,\]
 then $M_s^*$ is a martingale.       
   \end{proposition}

If we write dots for $r$-derivatives, then by considering the
martingale at time $s=0$ and using It\^o's formula, we get
the equation
\[   \dot F =   
              \frac \kappa 2 \, F''+
              \zhan_I \, F'   - \dot J \,F , \]
 where
 \[   -\dot J(r) = \p_s J(r-s) \mid_{s=0}. \]
 All the remains for proving Proposition \ref{greatprop}
 is to calculate $-\dot J(r)$.

\begin{lemma}
\[  -\dot J(r) = \alpha(r)
+  b \zhan_I'(r,x)  - \frac{b}r.\]
\end{lemma}

\begin{proof}
 
We have parametrized radial $SLE_\kappa$ such that
\[ \p_t \tilde g_t'(0) = (a/2) \, g_t'(0), \]
and hence
\[  \p_t\log \left[\tilde g_t'(0)^{-\tilde b}
  \right]|_{t=0} = - \frac{a \tilde b}{2} = \frac{b(1-a)}{4},\]
  \[  \p_s \log \left[ (\tilde g_{r-s}^*)'(0)^{-\tilde{b}}
   \right]_{s=0}  = -\tilde b . \] 
The relationship between the Brownian loop measure
and the bubble measure implies
\[    \p_s \frac{\cent}{2} \, m_\Disk(\Disk_r,
\eta_{\sigma_{r-s}})\mid_{s=0} = \frac 2 a  \p_t 
               \frac{\cent}{2} \,  m_{\Disk}(\Disk_r,\eta_t)
                 |_{t=0} 
  =  \cent \, \Gamma_{\Disk}
 (1,A_r) = \cent\, \Gamma(r). \]  
Lemma \ref{magic} shows that 
  \[
   \p_s \log |(\tilde  \phi_{r-s}^*)'(U_s^*)|^b \mid_{s=0}
=-\frac b{r} + 2b\Gamma(r).\]
 Recall that if $z = x + ir,
w = e^{iz} = e^{-r + i x} ,  $
$  \tilde h_s^*(w) = e^{i h_s^*(z)}, $
and hence
\[        |(\tilde h_{r-s}^*)'(w)|
          =   e^r \, e^{- \Im[h_{r-s}^*(z)]}
            \, |(h_{r-s}^*)'(z)| = 
              e^s \, |(h_{r-s}^*)'(z)|. \]
Therefore, using  \eqref{future}, we have
\[   \p_s  \log |(\tilde h_{r-s}^*)'(w)|^b \mid_{s=0}
   = b +    b \zhan_I'(r,x)  .\]
Adding all the terms, 
 gives
\[ b -\tilde b +  (\cent + 2b) \, 
\,\Gamma(r)
+  b \zhan_I'(r,x)  - \frac{b}r=
 \alpha(r) +  b \zhan_I'(r,x)  - \frac{b}r.\]
\end{proof}

\subsection{Comparing annulus $SLE$ with radial $SLE$ large $r$}

We now have an essentially complete description of annulus
$SLE_\kappa$.  In our framework, this is a measure $\mu_{A_r}
(1,e^{-r+ix})$ of total
mass $\hat F(r,x)$.  In the next subsection, we will prove
the following.

\begin{theorem}  \label{bast.theorem}
There exist $c_* ,q\in(0,\infty)$ such
that uniformly in $x$,
\[    \hat F(r,x) = 
 c_* \, r^{\cent/2} \, e^{(b-\tilde b)r}
  \, [1 + O(e^{-qr})], \;\;\;\; r \rightarrow \infty. \]
\end{theorem}
Let $\mu_2 = \mu_{\Disk}(1,0)$ as before and
let $\mu_4 = \mu_{A_r}
(1,e^{-r+ix})$ with corresponding probability measure
$\mu_4^\#$.   Suppose $t$ is sufficiently small so
that a curve starting at the unit disk cannot reach $C_r$ by
time $t$. Then, similarly to the previous section, if
$w = e^{-r + ix}$ and $\zeta_t = \tilde g_t(\gamma(t))$,
 we can
write 
\[
 \frac{d\mu_4}{d\mu_2} (\eta_t)  =  \frac{|\tilde g_t'(w)|^b}{\tilde g_t'(0)^
 {\tilde b}}
 \, \exp \left\{\frac \cent 2 m_\Disk(\Disk_r, \eta_t)\right\}
     \,  
       {|\mu_{\tilde g_t(A_r \setminus \eta_t)}(\zeta_t , \tilde g_t(w) )|}
        .\]
        
\begin{proposition} There exists $q > 0$ such
that uniformly over 
$t> 0$, $r \geq   \frac {ta}{2} + 2$, and all initial segments
$\gamma_t$,
\[  \frac{d\mu_4}{d\mu_2} (\eta_t) =
c_* \,e^{r(b - \tilde b)}
  \,r^{\cent/2} \,  [1 + O(e^{-qu})] ,\]
  where 
$u = r -\frac {ta}{2}$.   In particular, there exists $c < \infty$
such that 
\[   \left| \frac{d\mu_4^\#}{d\mu_2^\#} (\eta_t)-1 \right| \leq c \, e^{-qu}.\]
\end{proposition}

\begin{proof}
 Let $\phi_t: \tilde g_t(A_r \setminus \eta_t) \rightarrow A_s$ be a conformal
 transformation sending $C_0$ to $C_0$ and let $h_t = \phi_t \circ \tilde g_t$. 
Using conformal covariance, we write
\[  \frac{d\mu_4}{d\mu_2} (\eta_t) = 
  \frac{|h_t'(w)|^b\, |\phi_t'(\zeta_t)|^b}{\tilde g_t'(0)^
 {\tilde b}}\, \exp \left\{\frac \cent 2 m_\Disk(\Disk_r, \eta_t)\right\}
       {|\mu_{A_s}(h_t(w),\phi_t(\zeta_t) )| } .\]

Suppose $t$ is given, $r \geq   \frac {ta}{2} + 2$ and let
$u = r -\frac {ta}{2}$.  
Recall that in our normalization $\tilde g'_t(0) = e^{at/2}$.
Using the deterministic estimates from Lemma \ref{jul13.lemma1}, we get
\[              |h_t'(w)|^b = e^{atb/2} \, [1 + O(e^{-u})], \]
\[             |\phi_t'(\zeta_t)|^b = 1 + O(e^{-u}),\]
\[    \exp \left\{\frac \cent 2 m_\Disk(\Disk_r, \eta_t)\right\}
  = (r/u)^{\cent/2} \, [1 + O(e^{-u})], \]
  \[  s = u + O(e^{-u}) , \]
  \[   {|\mu_{A_s}(h_t(w),\phi_t(\zeta_t) )| } 
  =  c_* \, u^{\cent/2} \, e^{(b-\tilde b)u}
  \, [1 + O(e^{-u})] . \]
  Combining these estimates gives the first equality and
  since the dominant factor does not depend on the initial
  segment, the second equality follows.
 \end{proof}

%
%
  
\subsection{Proof of Theorem \ref{bast.theorem}}

 Let
\[  \lambda(r) = r^{b}\,\exp\left\{-\int_1^r \alpha(s)\, ds\right\},\]
\[   K_1(r,x) = \lambda(r) \, F(r,x), \]
 \[  K(r,x) = \lambda(r)
 \, \tmass_{A_r}(1,e^{-r+ix})
 =\lambda(r)
 \, \hat F(r,x) = \sum_{k \in \Z}
    F(r, x+ 2\pi k) .\]
Proposition \ref{deltaprop} gives
\[     \alpha(r) = b - \tilde b + (2b+ \cent) \, \Gamma(r)
  = b - \tilde b  + \frac{2b + \cent +O(e^{-r}) }{2r}  ,\]
  and hence
      \[ \lambda(r) 
    = \lambda_\infty \, r^{- \cent/ 2} \, e^{(\tilde b - b)r}
     \, [1 + O(r^{-1}e^{-r})].\]
 Therefore, to prove Theorem \ref{bast.theorem}, it suffices
 to show that there exists $K_\infty \in (0,\infty)$ and $c < \infty$
 such that
 \[             |K(r,x) - K_\infty|\leq c\, e^{-r}. \]

     Since
  \[   \dot \lambda(r) = \lambda(r) \, \left[\frac{b}{r} - \alpha(r) \right],\]
   it follows from Proposition \ref{greatprop} that 
 $K_1,K$ satisfy
 \[  \dot K_1
 = \frac{\kappa}{2} \, K_1''
  +   \zhanh_I K_1' + b\, \zhanh_I'\,
   K_1 ,\]
\begin{equation}  \label{kpde}  \dot K
 = \frac{\kappa}{2} \, K''
  +   \zhanh_I K' + b\, \zhanh_I'\,
   K .
 \end{equation}
 
%
The Feynman-Kac representation  tells us that
if $r > t > 0$,
\begin{equation}  \label{f-k}
K(r,x) = \E^x\left[
   K(r-t,X_t^{(r)}) \, \exp \left\{\int_0^t
       \nzhan(r-s,X_s^{(r)}) \, ds \right\}
         \right], 
  \end{equation}
         where $X_t^{(r)}$ satisfies \eqref{xsde}.
Recall that
\begin{equation}  \label{nov25.30}
   |\zhanh_I(r,x)|,   \,
   |\nzhan(r,x) | \leq  c_0 \, e^{-r}
    ,
   \;\;\;\; r \geq 1, 
   \end{equation}  
which implies
 \begin{equation}  \label{thank.1}
     \left|\int_0^{r-t}
    \zhanh_I'(z,X_{s}^{(r)}) \, ds \right| \leq c \, e^{-t}
    , 
     \;\;\;\;
  \exp\left\{b\int_0^{r-1}
    \zhanh_I'(z,X_{s}^{(r)}) \, ds \right\} 
     \asymp 1 , 
     \end{equation}
     and for $r \geq 1$, 
    \[  K(r,x) \asymp \E^x \left[K(X_{r-1}) \right]
           \leq c \, \E^x\left[ \exp \left\{-2b X_{r-1} \right\}
           \right], \]
where $X_s = X_{s}^{(r)}$.

   \labove \textsf
{\begin{small}  \Heuristic  
Those experienced with PDEs can probably skip the rest of this
section.  Since $|\zhan_I| +|\zhan_I'| = O(e^{-r})$, for large $r$ the
equation 
\eqref{kpde} is well approximated by the standard heat equation
$\dot K = \frac \kappa 2 \, K''$.  One just needs to keep track
of the error terms.  I have taken a probabilistic approach using
coupling, but this is just personal preference.
 \end{small}}
 \lbelow

We will use standard coupling techniques to analyze the
equation.  Here is the basic estimate.  We write 
$x \equiv y$ if $(y-x)/2\pi \in \Z$. 

\begin{lemma}  \label{couple}
  There exist $u > 0, c< \infty$
such that the following holds.  Suppose $r \geq 2$
and $X_t
= X_t^{(r)}, Z_t = Z_t^{(r)}$ are independent
solutions to 
\eqref{xsde} with $X_0 = x , Z_0 = y$ with
$x \leq y < x+ 2 \pi$.  Let
\[  T = \inf\{t: X_t \equiv Z_t  \}
. \]
Then,
\[   \Prob\{T \geq t \} \leq c \, e^{-ut} , \]
and if $t \leq 1$,
\[    \Prob\{T \geq t^2\}  \leq 
            c \,t^{-1} \, (y-x).\]
If we define
\[   Y_t = \left\{ \begin{array}{ll} Z_t & t < T\\
                    Z_T + (X_t - X_T) & t \geq t 
                    \end{array} \right. \]
Then $Y_t$ satisfies \eqref{xsde} with $Y_0 = y$
and $Y_t \equiv X_t$ for $t \geq T$.  

\end{lemma}

   
 \begin{proposition} $\;$
 
 \begin{itemize}
 
 \item There exist $0 < c_1 < c_2 < \infty$
 such that 
 \begin{equation}  \label{thank.5}
    c_1 \leq K(r,x) \leq c_2 , \;\;\; r \geq 1, \,
   x \in \R. 
   \end{equation}
   
   \item There exists $K_\infty \in (0,
 \infty)$ and $u > 0$ and $ c < \infty$ such that  
 \[    |K(r,x) - K_\infty| \leq c \, e^{-ur}.\]
   
   \end{itemize}

   \begin{proof}  For fixed $r$, $x \leq y \leq x +
   2\pi$, let $X_t,Y_t,T$ be as in Lemma 
  \ref{couple} and let $m_-(r), m_+(r)$ be
  the minimum and maximum, respectively, of $K(r,x)$
  for $0 \leq x \leq 2\pi$.   
  From \eqref{f-k} and \eqref{thank.1},
  we see that
$ c_1 \, m_-(1)  \leq 
     K(r,x) \leq c_2 \, m_+(1). $
 Using \eqref{thank.1},
 \[   K(r,x) = \E^x\left[F(r/2,K_{r/2})\right]
   \, [1 + O(e^{-r/2})]. \]
     This gives \eqref{thank.5}.
Combining this with the coupling, we see that
\[    K(r,x) = K(r,y) \, \left[1 + O(e^{-ur}) \right]. \]
  
  \end{proof}
   
 \end{proposition}

$  $

\end{document}